\documentclass[11pt,a4paper]{article}

\usepackage[headinclude]{typearea}
\usepackage[english]{babel}           
\usepackage[T1]{fontenc}
\usepackage{scrtime}
\usepackage{amsmath,amsthm,amsfonts,amssymb}
\usepackage{xcolor}                                  
\usepackage {pifont}                                 
\usepackage {multicol,multirow}                             
\usepackage[numbers]{natbib}     
\usepackage[normalem]{ulem}  
\usepackage [scaled=.90]{helvet}   
\usepackage {courier}     
\usepackage {graphicx}       
\usepackage {array}  
\usepackage {longtable}   
\usepackage{fancyvrb}
\usepackage{enumerate} 
\usepackage{ifpdf}
\usepackage{caption}
\usepackage{mathrsfs}
\usepackage{setspace}
\usepackage{upgreek}
\usepackage{marvosym}
\usepackage{mathtools}
\mathtoolsset{showonlyrefs} 
\usepackage{verbatim}
\usepackage{dsfont}
\usepackage{hyperref}
\makeatletter
\def\namedlabel#1#2{\begingroup
    #2%
    \def\@currentlabel{#2}%
    \phantomsection\label{#1}\endgroup
}

\numberwithin{equation}{section}
\newtheorem{theorem}{Theorem}[section]

\newtheorem{lemma}[theorem]{Lemma}

\theoremstyle{definition}
\newtheorem{assumption}{Assumption}[section]
\newtheorem{definition}[theorem]{Definition}

\newtheorem{remark}[theorem]{Remark}

\newcommand{\E}{{\mathbb{E}}}
\newcommand{\N}{{\mathbb{N}}}
\renewcommand{\P}{{\mathbb{P}}}

\newcommand{\R}{{\mathbb{R}}}
\newcommand{\F}{\mathbb{F}}

\newcommand{\cF}{{\cal F}}
\newcommand{\diff}{\mathop{}\!\mathrm{d}}

\title{Strong order $1$ adaptive approximation of jump-diffusion SDEs with discontinuous drift}

\author{Verena Schwarz}
\date{}

\begin{document}

\maketitle

\begin{abstract}
We present an adaptive approximation scheme for jump-diffusion SDEs with discontinuous drift and (possibly) degenerate diffusion.
This transformation-based doubly-adaptive quasi-Milstein scheme is the first scheme that has strong convergence rate $1$ in terms of the number of evaluations of the driving noise processes in $L^p$ for $p\in[1,\infty)$. 
To obtain our result, we prove that under slightly stronger assumptions, which are still weaker than those in the existing literature, a related doubly-adaptive quasi-Milstein scheme has convergence order $1$. 
This scheme is doubly-adaptive in the sense that it is jump-adapted, i.e.~all jump times of the Poisson noise are grid points, and it includes an adaptive step-size strategy to account for the discontinuities of the drift. \\
\noindent Keywords: jump-diffusion stochastic differential equations, discontinuous drift, strong convergence rate, adaptive scheme, higher order scheme\\
Mathematics Subject Classification (2020): 60H10, 65C30, 65C20
\end{abstract}

\section{Introduction}\label{sec:intro}

We consider the following time-homogeneous jump-diffusion stochastic differential equations (SDEs),
\begin{equation}\label{eq:SDE}
\begin{aligned}
\diff X_t = \mu(X_t) \diff t + \sigma(X_t)\diff W_t + \rho(X_{t-}) \diff N_t, \quad t\in[0,T],\quad X_0 =\xi,
\end{aligned}
\end{equation}
where $\xi\in\R$, $\mu,\sigma,\rho\colon\R\to \R$ are measurable functions, $T\in(0,\infty)$, $W=(W_t)_{t\geq 0}$ is a
standard Brownian motion, and  $N= (N_t)_{t\geq 0}$ is a homogeneous Poisson process with intensity $\lambda \in (0,\infty)$ on the complete probability space $(\Omega,\cF,\P)$. Further, the filtration $\mathbb{F}=(\mathbb{F}_t)_{t\geq 0}$ is the augmented natural filtration of $N$ and $W$, i.e.~for all $t\geq 0$, $\mathbb{F}_t = \sigma(\{(W_s,N_s) \colon s\leq t\}\cup \mathcal{N})$, where $\mathcal N = \{A\in \cF\colon \P(A)=0\}$. Then $(\Omega,\cF,\mathbb{F},\P)$ is a filtered probability space satisfying the usual conditions. 

We study the numerical approximation of these SDEs in the case of a discontinuous drift. To be more precise, we assume that the drift coefficient is piecewise Lipschitz, the diffusion and jump coefficients are Lipschitz continuous, the diffusion is non-zero at the discontinuities of the drift, and the drift and diffusion coefficients are differentiable with Lipschitz continuous derivatives in between neighbouring discontinuities. Under our assumptions existence and uniqueness in the strong sense is settled by \cite{PS20,PSX21}. Such SDEs are, for example, used for modelling energy prices or control actions on energy markets, cf.~\cite{situ2005,PBL2010,benth2014,SS16, shardin2017}. We construct an adaptive approximation scheme that has convergence rate $1$ in terms of the number of evaluations of the driving noise processes for jump-diffusion SDEs with discontinuous drift, the so-called transformation-based doubly-adaptive quasi-Milstein scheme. The main novelty is that for the first time, we combine two different types of adaptivity in this scheme. It is jump-adapted, i.e.~all the jump times of the Poisson process are added to the grid, and it is adaptive with respect to the discontinuity points of the drift, i.e.~the closer the approximated solution is to a discontinuity point, the smaller the size of the next time step. This is the first approximation scheme for jump-diffusion SDEs with discontinuous drift that has convergence rate $1$ in terms of the number of evaluations of the driving noise processes in $L^p$ for $p\in[1,\infty)$ and hence recovers the optimal rate for non-adaptive as well as adaptive schemes for a large class of SDEs without jump noise and with Lipschitz-continuous coefficients \cite{mullergronbach2002,mullergronbach2004}. 

The numerical approximation of SDEs with discontinuous drift has recently received a lot of attention. For SDEs without jump noise see, e.g., \cite{halidias2006,halidias2008,gyongy2011,LS16,ngo2016,LS17,ngo2017a,ngo2017b,pamen2017,LS18,hefter2019,muellergronbach2019,muellergronbach2019b,NSS19,Dareiotis2020,MuellerGronbach2023,Neuenkirch2021,Butkovsky2021,Dareiotis2023,Yaroslavtseva2022,Mullergronbach2022,Hu2022, ellinger2024sharp, ellinger2024optimal, muller2025performance}.
Convergence results for numerical approximations of SDEs with jump noise can be found under standard assumptions in  \cite{Gardon2004,Gardon2006,PBL2010,buckwar2011}, and under non-standard assumptions in \cite{Higham2005,Higham2006,Higham2007,dareiotis2016,Przybylowicz2016,Kaluza2018,Deng2019a,Deng2019b,Tambue2019,Przybylowicz2019a,Przybylowicz2019b,Chen2020,Liu2020,Butkovsky2022,Kieu2022,Nasroallah2022,Przybylowicz2022,Kaluza2022,przybylowicz2024randomized,kelly2025strong}. The approximation of jump-diffusion SDEs with discontinuous drift has first been studied in \cite{PS20} where $L^2$-convergence order $1/2$ for the Euler--Maruyama approximation is proven. In \cite{PSS2024JMS} the first higher-order scheme, the so-called transformation-based jump-adapted quasi-Milstein scheme, is defined and strong convergence order $3/4$ in $L^p$, $p\in[1,\infty)$ is proven. This rate is proven to be optimal in $L^1$ among all approximation schemes that are only allowed to be jump-adapted in \cite{przybylowicz2024lower}.

The transformation-based doubly-adaptive quasi-Milstein scheme defined in this article is inspired by some ideas from the literature on the construction of algorithms. Transformation-based schemes have been introduced in \cite{LS16,LS17} for SDEs with discontinuous drift without a jump noise. The transformation is defined in such a way that the drift coefficient of the transformed SDE gains additional regularity. This technique inspired many of the subsequent contributions. A transformation-based quasi-Milstein scheme is proven to have convergence order $3/4$ in $L^p$ in \cite{muellergronbach2019b} for SDEs with discontinuous drift without a jump noise. 
Schemes with adaptive step-size function originate in \cite{NSS19}, where an adaptive Euler--Maruyama scheme has been studied for the approximation of SDEs with discontinuos drift without a jump noise. A similar idea has been used in \cite{Yaroslavtseva2022} to construct a transformation-based adaptive quasi-Milstein scheme with convergence rate $1$ in terms of the number of evaluations of the driving Brownian motion in $L^p$ for SDEs with discontinuous drift without jump noise. The schemes and results in \cite{muellergronbach2019b, Yaroslavtseva2022} cannot be extended straightforwardly to jump-diffusion SDEs. 
Instead jump-adapted approximation schemes are needed, which are for example studied in \cite{PBL2010, kelly2025strong, PSS2024JMS}. In \cite{PSS2024JMS} the transformation-based jump-adapted quasi-Milstein scheme has been studied for jump-diffusion SDEs but the best rate that could be achieved by this scheme is $3/4$. 

In this paper, we combine the approximation schemes from \cite{Yaroslavtseva2022} and \cite{PSS2024JMS}, which are based on ideas introduced in \cite{NSS19} resp.~\cite{PBL2010}, to define the first doubly-adaptive approximation scheme and prove that it converges with strong rate $1$ in terms of the number of evaluations of the driving noise processes. For this a transformation $G\colon\R\to\R$ is applied to $X$ to obtain an SDE for $Z=G(X)$, which satisfies the assumptions made on the coefficients of $X$ and additionally has a continuous drift. This implies that the drift of the transformed SDE is Lipschitz continuous. For the approximation of the solution $Z$ we define the same step-size function $h^\delta\colon\R\to\R$ as in \cite{Yaroslavtseva2022}, which has value $\delta$ away from the discontinuities of the drift and smaller values closer to the discontinuities of the drift. Given the value of the approximation $Z^{(\delta)}$ at some grid point, we determine the next grid point by adding the minimum of the step-size function evaluated at this point and the time until the next jump of the Poisson process. Then we compute the left limit of $Z^{(\delta)}$ at the next grid point via a quasi-Milstein scheme. In case this grid point is a jump time of the Poisson process, we then add the corresponding jump term;  
for details see \eqref{eqDefMS0}--\eqref{EqDefMS}. 

For this scheme we prove in Theorem \ref{ConvResTSDE} that under our assumptions for all $p\in[1,\infty)$ there exists a constant $c\in(0,\infty)$ such that for small enough values of $\delta$,
\[\E\Big[\sup_{t\in[0,T]}\big| Z(t) - Z^{(\delta)}(t)\big|^p \Big]^{\frac{1}{p}} \leq c \delta.\]
To ensure that the cost of our adaptive algorithm stays proportional to $\delta$, a cost analysis is necessary. For this, we denote by $n(Z^{(\delta)}_T)$ the maximum number of time steps needed to reach $T$, which is proportional to the computational cost of the numerical approximation $Z^{(\delta)}$. In Theorem \ref{thm:cost} we prove that there exists a constant $\widetilde c\in(0,\infty)$ such that for small enough values of $\delta$,
\[ \E[n(Z^{(\delta)}_T)]\leq c(\delta^{-1}+\E[N_T]).\]
Using these results and analytical properties of the transformation $G$ and its inverse $G^{-1}$ we prove in Theorem \ref{MainWeakAss} under our assumptions on the coefficients of SDE \eqref{eq:SDE} that there exist $C,\widetilde C\in(0,\infty)$ such that for small enough values of $\delta$,
\begin{equation}
    \begin{aligned}
    \E\Big[ \sup_{t\in[0,T]} |X_t- G^{-1}(Z^{(\delta)}_t) |^p\Big]^\frac{1}{p} \leq C \delta, \qquad 
    \E\Big[n\big(G^{-1}(Z^{(\delta)}_T)\big)\Big]\leq \widetilde C (\delta^{-1}+\E[N_T]).
    \end{aligned}
\end{equation}

Note that in case $\rho \equiv 0$ our scheme coincides with the scheme considered in \cite{Yaroslavtseva2022} and our proofs follow the general strategy presented in \cite{Yaroslavtseva2022}. However, adding the driving Poisson noise and the corresponding jump coefficient requires important changes in the definition of the approximation scheme compared to \cite{Yaroslavtseva2022}. Similar to \cite{PSS2024JMS} a jump-adapted grid needs to be included in the scheme. In combination with the adaptive step-size function, this two-fold adaptivity is the main novelty of the scheme. Together with the jump noise itself, they present challenges for the theoretical proofs. To meet these challenges, we introduce a series representation of the coefficients of the SDE and the approximation scheme, we condition on either the Poisson noise having a certain value or on additionally introduced filtrations, and we carefully estimate the newly appearing terms. The big advantage of using a jump-adapted scheme is that the approximation scheme jumps at exactly the same times as the solution itself. This allows for improved estimates on the jump terms of the approximation scheme, which are crucial to obtain strong convergence rate $1$ in terms of the number of evaluations of the driving noise processes in $L^p$ for $p\in[1,\infty)$.

The paper is structured as follows: In Section \ref{sec:setting} we introduce the setting. In Section \ref{SecQMS} we define the doubly-adaptive quasi-Milstein scheme, prove its strong convergence rate and estimate the computational cost. In Section \ref{ConvDAqMS} we define based on the previous section the transformation-based doubly-adaptive quasi-Milstein scheme and prove the main result of this article.

\section{Setting}\label{sec:setting}

In the following, we denote the compensated Poisson process by $\tilde N=(\tilde N_t)_{t\geq 0}$ with $\tilde N_t=N_t-\lambda t$ for all $t\in[0,\infty)$. Note that the compensated Poisson process is a square-integrable $\mathbb{F}$-martingale. Further, let  $L_f$ be the Lipschitz constant of a generic Lipschitz continuous function $f$, and $\operatorname{Id}$ the identity mapping.
Moreover, for two measurable spaces $(A_1,\mathcal{A}_1)$ and $(A_2,\mathcal{A}_2)$, for measurable functions $Y_1\colon(\Omega,\mathcal{F})\to (A_1,\mathcal{A}_1)$, $Y_2\colon(\Omega,\mathcal{F})\to (A_2,\mathcal{A}_2)$, and $f\colon A_1\times A_2 \to \R$, and $\mathcal{G}$ a sub-$\sigma$-algebra of $\mathcal{F}$, we denote by $\E[f(Y_1,y)|\mathcal{G}] |_{y=Y_2}$ the random variable $g(Y_2)$, where $g\colon A_2\to\R$ is defined by $g(y)= \E[f(Y_1,y)|\mathcal{G}]$ for all $y\in A_2$. In particular, this notation is used in case that $\mathcal{G} = \{\emptyset, \Omega\}$, then conditioning is omitted and we write $\E[f(Y_1,y)] |_{y=Y_2}$.
We recall the definition of piecewise Lipschitz:

\begin{definition}[\text{\cite[Definition 2.1]{LS17}}]
Let $I\subseteq\R$ be an interval and let $m\in\N$. A function $f\colon I\to\R$
is called piecewise Lipschitz, if there are finitely many points $\zeta_1<\ldots<\zeta_m\in
I$ such that $f$ is Lipschitz on each of the intervals $(-\infty,\zeta_1)\cap I,
(\zeta_m,\infty)\cap I$, and $(\zeta_k,\zeta_{k+1}),k=1,\ldots,m-1$.
\end{definition}

In this article we work under the same assumptions as in \cite{PSS2024JMS}.
\newpage

\begin{assumption}\label{assX}
Assume that there exist $m\in\N$ and $-\infty = \zeta_0 < \zeta_1 < \ldots< \zeta_m< \zeta_{m+1} =\infty$ such that the coefficients of SDE \eqref{eq:SDE} satisfy:
\begin{itemize}
\item[(i)] $\mu\colon\R \to \R$ is piecewise Lipschitz continuous with potential discontinuities in $\zeta_1,\dots,\zeta_m\in\R$;
\item[(ii)] $\sigma\colon\R\to\R$ is Lipschitz continuous and for all $k\in\{1,\dots,m\}$, $\sigma(\zeta_k) \neq 0$;
\item[(iii)] $\rho\colon\R\to\R$ is Lipschitz continuous;
\item[(iv)] $\mu$ and $\sigma$ are differentiable with Lipschitz continuous derivatives on $(\zeta_{i-1},\zeta_i)$ for all $i\in\{1,\ldots,m+1\}$.
\end{itemize}
\end{assumption}

\begin{remark}
    Under Assumption \ref{assX} it is well-known that $\mu$, $\sigma$, and $\rho$ satisfy a linear growth condition, i.e.~for $f\in\{\mu,\sigma,\rho\}$ there exists a constant $c_f \in (0,\infty)$ such that $|f(x)|\leq c_f(1+|x|)$ and that there exists a unique strong solution to SDE \eqref{eq:SDE}, e.g.~\cite[Lemma 2]{PS20}.
\end{remark}

\section{Convergence of the doubly-adaptive quasi-Milstein scheme}\label{SecQMS}

In this section, we introduce the doubly-adapted quasi-Milstein scheme and prove under slightly stronger assumptions than Assumption \ref{assX} that this scheme has convergence rate $1$ in $L^p$ for $p\in[1,\infty)$. We need this result in Section \ref{ConvDAqMS} to prove our main result using a transformation technique.

Consider the following time-homogeneous jump-diffusion SDE
\begin{align}\label{eq:TSDE}
\diff Z_t &=\widetilde \mu(Z_t)  \diff t + \widetilde \sigma(Z_t) \diff W_t+\widetilde \rho(Z_{t-})\diff N_t , \quad t\in[0,T], \quad Z_0=\widetilde \xi,
\end{align}
with $\widetilde \xi\in\R$, $\widetilde\mu,\widetilde \sigma,\widetilde\rho\colon\R\to \R$ are measurable functions, $T\in(0,\infty)$, $W=(W_t)_{t\geq 0}$, $N= (N_t)_{t\geq 0}$, and  $(\Omega,\cF,\F,\P)$ are defined as in Section \ref{sec:intro}.

\begin{assumption}\label{assZ} 
We assume on the coefficients of SDE \eqref{eq:TSDE} that
\begin{itemize}
\item[\namedlabel{Zi}{(i)}] $\widetilde\mu\colon\R \to \R$,  $\widetilde\sigma\colon\R\to\R$, and $\widetilde\rho\colon\R\to\R$ are Lipschitz continuous;
\item[\namedlabel{Zii}{(ii)}] there exists $\widetilde m\in\N$ and $-\infty = \zeta_0 < \zeta_1 < \ldots< \zeta_{\widetilde m}< \zeta_{\widetilde m+1} =\infty$ such that $\widetilde\mu$ and $\widetilde \sigma$ are differentiable with Lipschitz continuous derivatives on $(\zeta_{k-1},\zeta_{k})$ for $k\in\{1,\ldots,\widetilde m+1\}$;
\item[\namedlabel{Zv}{(iii)}] for all $k\in\{1,\dots,\widetilde m\}$, $\widetilde\sigma(\zeta_k) \neq 0$.
\end{itemize}
\end{assumption}

\begin{remark}
    Under Assumption \ref{assZ} there exists a unique strong solution to SDE \eqref{eq:TSDE}, c.f.~\cite[p.~255,~Theorem~6]{protter2005}.
\end{remark}

Define for $f\in\{\widetilde\mu,\widetilde\sigma\}$ the functions $d_f\colon\R\to\R$ by
\begin{equation}
d_f(x) = \left\{
\begin{array}{ll}
f'(x) & f \text{ differentiable at } x \\
0 & \, \textrm{else.} \\
\end{array}
\right. 
\end{equation}

\begin{remark}\label{ConAss}
    Under Assumption \ref{assZ} by \cite[Lemma 3.1]{PSS2024JMS} it holds that
    \begin{itemize}
        \item[(i)] for $f\in\{\widetilde\mu,\widetilde\sigma,\widetilde\rho\}$ there exists  $c_f \in (0,\infty)$ such that for all $x\in\R$, $|f(x)|\leq c_{f}(1+|x|)$;
        \item[(ii)] for $g\in\{\widetilde\mu,\widetilde\sigma\}$ it holds that $\|d_{g}\|_{\infty}\leq L_{g} <\infty$ and there exists $b_{g} \in(0,\infty)$ such that for all $x,y\in\R$ for which $g$ is differentiable on $[x,y]$, $|g(y)-g(x)-g'(x)(y-x)|\leq b_{g} |y-x|^2$.
    \end{itemize}
\end{remark}

To construct the doubly-adapted quasi-Milstein scheme we first recall the step-size function $h^\delta\colon \R\to (0,1)$ from \cite{Yaroslavtseva2022}, which is an adaption of the step-size function introduced in \cite{NSS19}. For this we first define the set $\Theta = \{\zeta_1,\ldots, \zeta_{\widetilde m}\}$ and for all $\varepsilon\in(0,\infty)$, $\Theta^\varepsilon = \{x\in\R: d(x,\Theta)<\varepsilon\}$. Further, let $\varepsilon_0 \in(0,1]$ such that $\varepsilon_0 \leq \frac{1}{2}\min\{\zeta_i -\zeta_{i-1} : i \in\{2,\ldots, \widetilde m\}\}$ whenever $\widetilde m \geq 2$. For $\delta > 0$ define, 
\begin{equation}
    \varepsilon_1^\delta = \sqrt{\delta}\log^2(\delta^{-1}), \quad\quad
    \varepsilon_2^\delta = \delta\log^4(\delta^{-1}).
\end{equation}
Next we chose $\delta_0\in(0,1)$ such that for all $\delta \in(0,\delta_0]$, 
\begin{equation}
    \varepsilon_2^\delta \leq \varepsilon_1^\delta \leq \varepsilon_0/2.
\end{equation}
With this the step-size function $h^\delta:\R\to(0,1)$ is for all $\delta \in(0,\delta_0]$ defined by
\begin{equation}
    h^\delta (x) = \left\{
\begin{array}{ll}
\delta & x\notin \Theta^{\varepsilon_1^\delta} \\
\Big(\frac{d(x,\Theta)}{\log^2(\delta^{-1})}\Big)^2 & x \in \Theta^{\varepsilon_1^\delta} \setminus \Theta^{\varepsilon_2^\delta}\\
\delta^2\log^4(\delta^{-1}) & x\in \Theta^{\varepsilon_2^\delta}.
\end{array}
\right.
\end{equation}
Here $\log$ denotes the logarithm to the basis $e$. Note that $h^{(\delta)}$ is continuous and for all $\delta\in(0,\delta_0]$ and all $x\in\R$, $\delta^2 \log^4(\delta^{-1})\leq h^{\delta}(x) \leq \delta$.  Using $h^\delta$, denoting by $(\nu_i)_{i\in\N}$ the family of all jump times of the Poisson process and recalling the definition of the jump-adapted quasi-Milstein scheme from \cite{PSS2024JMS}, we define for all $\delta\in(0,\delta_0]$ the doubly-adaptive quasi-Milstein scheme $(Z^{(\delta)}_t)_{t\in[0,T]}$ by
\begin{equation}\label{eqDefMS0}
\begin{aligned}
\tau_0 = 0, \quad Z^{(\delta)}_0 =\widetilde\xi,
\end{aligned}
\end{equation}
and for all $n\in\N_0$ by 
\begin{equation}
\begin{aligned}\label{EqTauGen} 
\tau_{n+1} = & \big(\tau_n+ h^\delta(Z^{(\delta)}_{\tau_n})\big)\wedge \min\{\nu_i: i\in\N, \nu_i > \tau_{n}\} \wedge T, 
\end{aligned}
\end{equation}
\begin{equation}
\begin{aligned}\label{EqDefMSLL}
Z^{(\delta)}_{\tau_{n+1}-} =& Z^{(\delta)}_{\tau_{n}} + \widetilde \mu\big(Z^{(\delta)}_{\tau_{n}}\big)  \big(\tau_{n+1}- \tau_{n}\big) + \widetilde\sigma \big(Z^{(\delta)}_{\tau_{n}}\big)\big(W_{\tau_{n+1}} -W_{\tau_{n}}\big)  \\
&+ \frac{1}{2} \widetilde\sigma \big(Z^{(\delta)}_{\tau_{n}}\big) d_{\widetilde\sigma} \big(Z^{(\delta)}_{\tau_{n}}\big)  \big(\big(W_{\tau_{n+1}} -W_{\tau_{n}}\big)^2 -\big(\tau_{n+1}-\tau_{n}\big)\big)
\end{aligned}
\end{equation}
and 
\begin{equation}
\begin{aligned}\label{EqDefMS}
Z^{(\delta)}_{\tau_{n+1}} = Z^{(\delta)}_{\tau_{n+1}-} + \widetilde \rho\big(Z^{(\delta)}_{\tau_{n+1}-}\big) \big(N_{\tau_{n+1}}- N_{\tau_{n}}\big) = Z^{(\delta)}_{\tau_{n+1}-} + \widetilde \rho\big(Z^{(\delta)}_{\tau_{n+1}-}\big) \big(N_{\tau_{n+1}}- N_{\tau_{n+1}-}\big).
\end{aligned}
\end{equation}
Between the points of the time discretisation, i.e.~for $t\in\big(\tau_{n},\tau_{n+1}\big)$, $n\in\N_0$, we set
\begin{equation}
\begin{aligned}\label{EqDefMScont}
Z^{(\delta)}_{t} =& Z^{(\delta)}_{\tau_{n}} + \widetilde \mu\big(Z^{(\delta)}_{\tau_{n}}\big) \big(t- \tau_{n}\big) + \widetilde\sigma \big(Z^{(\delta)}_{\tau_{n}}\big) \big(W_{t} -W_{\tau_{n}}\big)  \\
&+ \widetilde\sigma \big(Z^{(\delta)}_{\tau_{n}}\big) d_{\widetilde\sigma} \big(Z^{(\delta)}_{\tau_{n}}\big) \frac{1}{2} \big(\big(W_{t} -W_{\tau_{n}}\big)^2 -\big(t-\tau_{n}\big)\big).
\end{aligned}
\end{equation}
Here the sequence of grid points $(\tau_n)_{n\in\N}$ is chosen in a two-fold adaptive way. Namely it contains all jump points of the Poisson process and it takes into account how far the last approximated value of the solution is from the discontinuities of $\mu$ (due to the step-size function $h^\delta$).
Note that the following integral notation is equivalent to \eqref{eqDefMS0}, \eqref{EqDefMSLL}, \eqref{EqDefMS}, \eqref{EqDefMScont},
\begin{equation}\label{SumAppr}
\begin{aligned}
Z^{(\delta)}_{t} & =  \widetilde\xi
+ \int_0^t \sum_{n=0}^{\infty} \widetilde \mu\big(Z^{(\delta)}_{\tau_{n}}\big)\mathds{1}_{(\tau_n,\tau_{n+1}]}(u) \diff u\\
&\quad+ \int_0^t \sum_{n=0}^{\infty} \Big(\widetilde \sigma\big(Z^{(\delta)}_{\tau_{n}}\big) + \widetilde\sigma \big(Z^{(\delta)}_{\tau_n}\big) d_{\widetilde\sigma} \big(Z^{(\delta)}_{\tau_n}\big) (W_{u}-W_{\tau_n}) \Big)\mathds{1}_{(\tau_n,\tau_{n+1}]}(u) \diff W_u \\
&\quad+ \int_0^t \widetilde \rho\big(Z^{(\delta)}_{u-}\big)\diff N_u.
\end{aligned}
\end{equation}
Similarly, we can express the solution on SDE \eqref{eq:TSDE} by
\begin{equation}
\begin{aligned}
Z_{t} = &\, \widetilde\xi
+ \int_0^t \sum_{n=0}^{\infty} \widetilde \mu\big(Z_{u}\big)\mathds{1}_{(\tau_n,\tau_{n+1}]}(u) \diff u
+ \int_0^t \sum_{n=0}^{\infty} \widetilde \sigma\big(Z_{u}\big)\mathds{1}_{(\tau_n,\tau_{n+1}]}(u) \diff W_u
+ \int_0^t \widetilde \rho\big(Z_{u-}\big)\diff N_u.
\end{aligned}
\end{equation}

\subsection{Preparatory lemmas}

In this section, we provide some basic properties of the time discretisation $(\tau_n)_{n\in\N}$ and moment estimates for the approximation scheme $(Z^{(\delta)}_t)_{t\in[0,T]}$. 
We consider a second filtration $\widetilde\F=(\widetilde\F_t)_{t\geq 0}$, which will be needed later on. It is defined for all $t\in[0,\infty)$ by 
\begin{equation}
\begin{aligned}
\widetilde\F_t = \sigma(\{W_s:s\leq t\}\cup \{N_s: s\geq 0\} \cup \mathcal N),
\end{aligned}
\end{equation}
where $\mathcal N$ is the set of all nullsets of $\cF$.
$W$ is a Brownian motion with respect to the filtration $(\widetilde\F_t)_{t\geq 0}$ as $W$ and $N$ are independent.
Moreover, $Z^{(\delta)}$ and $Z$ are $\widetilde \F$-adapted càdlàg processes, since $\F_t\subset\widetilde\F_t$ for all $t\in[0,\infty)$.

\begin{lemma}\label{PropDiscGrid}
Let $\delta \in(0,\delta_0]$. Then the sequence $(\tau_n)_{n\in\N}$ satisfies for all $n\in\N$,
\begin{itemize}
    \item[\namedlabel{STi}{(i)}] $\tau_n$ is an $\F$-stopping time and an $\widetilde\F$-stopping time;
    \item[\namedlabel{STii}{(ii)}] $Z^{(\delta)}_{\tau_n}$ is $\F_{\tau_n}$-measurable and $\widetilde\F_{\tau_n}$-measurable;
    \item[\namedlabel{STiii}{(iii)}] $(W_{\tau_n+s}- W_{\tau_n})_{s\geq 0}$ is a Brownian motion with respect to the filtration $(\F_{\tau_n+s})_{s\geq0}$,  $(N_{\tau_n+s}- N_{\tau_n})_{s\geq 0}$ is a Poisson process with respect to the filtration $(\F_{\tau_n+s})_{s\geq0}$, both processes are independent of $\F_{\tau_n}$, and it holds that $(W_{\tau_n+s}- W_{\tau_n})_{s\geq 0}$ is independent of $(N_{\tau_n+s}- N_{\tau_n})_{s\geq 0}$;
    \item[\namedlabel{STiv}{(iv)}] $(W_{\tau_n+s}- W_{\tau_n})_{s\geq 0}$ is a Brownian motion with respect to the filtration $(\widetilde \F_{\tau_n+s})_{s\geq0}$, which is independent of $\widetilde \F_{\tau_n}$.
\end{itemize}
\end{lemma}

\begin{proof}
By induction, we prove \ref{STi} and \ref{STii}. Note that $\tau_0 \equiv 0$ is an $\mathbb F$-stopping time and $Z^{(\delta)}_{\tau_0} \equiv \widetilde \xi$ is $\mathbb{F}_0= \mathbb{F}_{\tau_0}$-measurable. Further, if $\tau_{n}$ is a stopping time and $Z^{(\delta)}_{\tau_n}$ is $\mathbb{F}_{\tau_n}$-measurable, then $\tau_n+ h^\delta(Z^{(\delta)}_{\tau_n}) > \tau_n$ is obviously $\mathbb{F}_{\tau_n}$-measurable and hence an $\mathbb F$-stopping time, cf.~\cite[Exercise 1.2.14]{karatzas1991}. This implies that  
$\tau_{n+1} = \big(\tau_n+ h^\delta(Z^{(\delta)}_{\tau_n})\big)\wedge \min\{\nu_i: i\in\N, \nu_i > \tau_{n}\} \wedge T$ is a $\mathbb F$-stopping time. Hence, by construction $Z^{(\delta)}_{\tau_{n+1}}$ is $\mathbb{F}_{\tau_{n+1}}$-measurable.
Item \ref{STiii} is implied by \cite[Theorem 40.10]{Sato2013} considering the two-dimensional Lévy process $(N,W)$ and \ref{STiv} is proven in \cite[Theorem 11.11]{Kallenberg1997}.
\end{proof}

Next we define for $\delta \in(0,\delta_0]$ and $t\in[0,\infty)$
\begin{equation}
    \underline{t} = \max\{\tau_n : n\in\N_0, \tau_n \leq t\}.
\end{equation}
Note that $\underline t$ is a random time, but not a stopping time. An introduction to random times can, for example, be found in \cite[p.~5]{karatzas1991}.
Further we define for all $s,t\in[0,\infty)$ the random variable
$W^{\underline t}_s = W_{\underline{t}+s}-W_{\underline t}$
and the stochastic process $W^{\underline t} =(W^{\underline t}_s)_{s\geq 0}$. Note that $W^{\underline t}$ is $\P$-a.s.~continuous.

\begin{lemma}\label{BMRandomTime}
    Let $\delta\in(0,\delta_0]$ and $t\in[0,\infty)$. Then $W$ and $W^{\underline t}$ have the same distribution and $W^{\underline t}$ is independent of $Z^{(\delta)}_{\underline t}$.
\end{lemma}

\begin{proof}
    For all $A\in\mathcal{B}(C([0,\infty),\R))$ we have by Lemma \ref{PropDiscGrid} \ref{STii}, \ref{STiii}, and the fact that $\tau_{n+1}$ is $\widetilde{\mathbb{F}}_{\tau_n}$-measurable that
    \begin{equation}
        \begin{aligned}
        &\P(W^{\underline t}\in A) 
        = \sum_{n\in\N_0} \P(W^{\tau_n} \in A, t\in[\tau_n, \tau_{n+1}))
        = \sum_{n\in\N_0} \P(W^{\tau_n} \in A) \P( t\in[\tau_n, \tau_{n+1}))\\
        &= \sum_{n\in\N_0} \P(W \in A) \P( t\in[\tau_n, \tau_{n+1}))
        = \P(W \in A) \sum_{n\in\N_0} \P( t\in[\tau_n, \tau_{n+1}))
        = \P(W \in A).
        \end{aligned}
    \end{equation}
    Further for all $A\in\mathcal{B}(C([0,\infty),\R))$ and all $B\in\mathbb{B}(\R)$ we have by Lemma \ref{PropDiscGrid} \ref{STii}, \ref{STiii}, and the fact that $\tau_{n+1}$ is $\widetilde F_{\tau_n}$-measurable that
    \begin{equation}
        \begin{aligned}
        &\P(W^{\underline t}\in A, Z_{\underline t}^{(\delta)}\in B) 
        = \sum_{n\in\N_0} \P(W^{\tau_n} \in A, Z_{\tau_n}^{(\delta)}\in B, t\in[\tau_n, \tau_{n+1}))\\
        &= \sum_{n\in\N_0} \P(W^{\tau_n} \in A)\P(  Z_{\tau_n}^{(\delta)}\in B, t\in[\tau_n, \tau_{n+1}))
        = \sum_{n\in\N_0} \P(W \in A) \P(Z_{\underline t}^{(\delta)}\in B, t\in[\tau_n, \tau_{n+1}))\\
        &= \P(W \in A) \sum_{n\in\N_0} \P(Z_{\underline t}^{(\delta)}\in B, t\in[\tau_n, \tau_{n+1}))
        = \P(W \in A)\P(Z_{\underline t}^{(\delta)}\in B).
        \end{aligned}
    \end{equation}
\end{proof}

We recall the following two lemmas for the convenience of the reader as they are frequently used in our proofs.

\begin{lemma}[\text{\cite[Lemma 3.3]{PSS2024JMS}}]\label{LMEW}
Let $p\in\N$. Then there exists a constant $c_{W_p}\in(0,\infty)$ such that for all stopping times $\tau$, all $\delta\in(0,\delta_0]$, and all $t\in[0,\infty)$, 
\begin{equation}\label{MEW}
\begin{aligned}
&\E\big[|W_{\tau+t} - W_\tau|^p\big] \leq c_{W_p} \, t^\frac{p}{2},
\end{aligned}
\end{equation}
and 
\begin{equation}\label{MEW2}
\begin{aligned}
&\E\Big[\sup_{s\in[0,\delta]} |W_{\tau+s} - W_\tau|^p\Big] \leq c_{W_p} \, \delta^\frac{p}{2}.
\end{aligned}
\end{equation}
\end{lemma}

\begin{lemma}[\text{\cite[Lemma 3.4]{PSS2024JMS}}]\label{BDGaKunita}
Assume that $q\in [2,\infty)$, $a,b\in[0,T]$ with $a<b$, $Z\in\{\operatorname{Id},W,N\}$, and that $Y = (Y(t))_{t\in[a,b]}$ is a predictable stochastic process with respect to $(\F_t)_{t\in[a,b]}$ with
\begin{equation}
\begin{aligned}
\E\Big[\int_a^b |Y(t)|^q \diff t\Big] < \infty. 
\end{aligned}
\end{equation}
Then there exists a constant $\hat c\in(0,\infty)$ such that for all $t\in[a,b]$,
\begin{equation}
\begin{aligned}
\E\Big[ \sup_{s\in[a,t]} \Big|\int_a^s Y(u) \diff Z(u)\Big|^q \Big] \leq \hat c \int_a^t \E[|Y(u)|^q]\diff u. 
\end{aligned}
\end{equation}
\end{lemma}

The following lemmas are proving moment bounds for the doubly-adaptive quasi-Milstein scheme. The proofs follow the strategy of \cite[Lemma 3.5 - Lemma 3.7]{PSS2024JMS} and can be found in Appendix \ref{appendix}. 

\begin{lemma}\label{MW}
    Let Assumption \ref{assZ} hold. Then for all $p\in\N_0$, $q\in\N$, $t\in[0,T]$, $\delta \in(0,\delta_0]$, $n\in\N$, $k\in\N$, and the constants $c_{W_q}$ as in Lemma \ref{MEW},
    \begin{equation}\label{MWeq1}
        \begin{aligned}
        &\E\Big[\big(1+ \big|Z^{(\delta)}_{\tau_n}\big|^p\big)   |W_t-W_{\tau_n}|^q \mathds{1}_{[\tau_n,\tau_{n+1})}(t)  \mathds{1}_{\{N_t = k\}} \Big]\\
        &\leq c_{W_q} \delta^{\frac{q}{2}} \E\Big[ \big(1+ \big|Z^{(\delta)}_{\tau_n}\big|^p\big) \mathds{1}_{[\tau_n,\tau_{n+1})}(t)  \mathds{1}_{\{N_t = k\}} \Big],
        \end{aligned}
    \end{equation}
    and
    \begin{equation}\label{MWeq2}
        \begin{aligned}
        &\E\Big[ \big(1+ \big|Z^{(\delta)}_{\underline{t}}\big|^p\big)   |W_t-W_{\underline{t}}|^q \Big]
        \leq c_{W_q} \delta^{\frac{q}{2}} \E\Big[ 1+ \big|Z^{(\delta)}_{\underline{t}}\big|^p \Big].
        \end{aligned}
    \end{equation}
\end{lemma}

\begin{lemma}\label{FiniteMom}
Let Assumption \ref{assZ} hold. Then for all $\delta\in(0,\delta_0]$ and $p\in[2,\infty)$ there exists a constant $c_\delta \in(0,\infty)$ such that 
\begin{equation}\label{LFM1}
    \begin{aligned}
    \int_0^T \E \big[ 1+ \big|Z^{(\delta)}_{\underline t}\big|^p\big] \diff t
    +\E\Big[ N_T^{p-1} \sum_{n=1}^{N_T + M} \big(1+ \big|Z^{(\delta)}_{\tau_{n}-}\big|^p\big) \Big] \leq c_\delta,
    \end{aligned}
\end{equation}
where $M = \Big\lceil \frac{T}{\delta^2 \log^4(\delta^{-1})} \Big\rceil$.
\end{lemma}

\begin{lemma}\label{ME}
Let Assumption \ref{assZ} hold. Then for all $p\in[2,\infty)$ there exist constants $c_1, c_2, c_3 \in(0,\infty)$ such that for all $\delta\in(0,\delta_0]$, $\widetilde\xi\in\R$, $\Delta \in[0,T]$ and all $s\in[0,T-\Delta]$ it holds that 
\begin{equation}\label{LME1}
\begin{aligned}
\E\Big[\sup_{t\in[0,T]} \big|Z^{(\delta)}_t\big|^p\Big] \leq c_1;
\end{aligned}
\end{equation}
\begin{equation}\label{LME2}
\begin{aligned}
\E\Big[  \big|Z^{(\delta)}_t- Z^{(\delta)}_{\underline t}\big|^p\Big] \leq c_2 \delta^{\frac{p}{2}};
\end{aligned}
\end{equation}
\begin{equation}\label{LME4}
    \begin{aligned}
    \E\Big[\sup_{t\in[s,s+\Delta]} \big|Z^{(\delta)}_t- Z^{(\delta)}_{s}\big|^p\Big] \leq c_3 \Delta.
    \end{aligned}
\end{equation}
\end{lemma}

\subsection{Occupation time estimates}

To prove occupation time estimates, we combine and adapt strategies presented in \cite{PSS2024JMS} and \cite{Yaroslavtseva2022}. While the results and proofs of Lemma \ref{EVofOTF}, Lemma \ref{LemProbEst}, and Lemma \ref{LemExepSet} are closely based on the corresponding lemmas in \cite{Yaroslavtseva2022}, Lemma \ref{Lem:LocTime} and Lemma \ref{MWwithInd} present results with proofs heavily influenced by the additional jump noise that make use of similar ideas as in \cite{PSS2024JMS}.

In a first step, we estimate the local time of $Z^{(\delta)}$. For this we denote for all $a\in\R$ by $(L_t^a(Z^{(\delta)}))_{t\in[0,T]}$ the local time of $Z^{(\delta)}$ in $a$. The local time of SDE \eqref{eq:TSDE} is defined for all $t\in[0,T]$ and all $a\in\R$ by
\begin{equation}
    \begin{aligned}
    L_t^a(Z^{(\delta)}) = 2 \Big(&(Z_t^{(\delta)}-a)^+ - (Z_0^{(\delta)}-a)^+ - \int_0^t \mathds{1}_{\{Z^{(\delta)}_{u-} >a\}} \diff Z^{(\delta)}_{u} \\
    &- \int_0^t \Big( \big( Z^{(\delta)}_{u-} + \widetilde\rho\big(Z^{(\delta)}_{u-}\big)- a\big)^+ - \big( Z^{(\delta)}_{u-} - a\big)^+ - \mathds{1}_{\{Z^{(\delta)}_{u-} >a\}} \widetilde\rho\big(Z^{(\delta)}_{u-}\big)\Big) \diff N_u \Big),
    \end{aligned}
\end{equation}
where  $(\cdot)^+ = \max\{\cdot,0\}$, see \cite[Definition 156]{situ2005}.

\begin{lemma}\label{Lem:LocTime}
    Let Assumption \ref{assZ} hold. Then there exists a constant $c_4\in(0,\infty)$ such that for all $\delta\in(0,\delta_0]$, all $t\in[0,T]$ and all $a\in\R$, 
    \begin{equation}\label{LocTimeEst}
        \begin{aligned}
        &\E\big[L^a_t \big(Z^{(\delta)}\big)\big] 
        \leq c_4.
        \end{aligned}
    \end{equation}
\end{lemma}

\begin{proof}
     By \cite[Lemma 158]{situ2005} for all $a\in\R$, all $\delta \in(0,\delta_0]$ and all $t\in[0,T]$,
    \begin{equation}\label{pOTE8}
        \begin{aligned}
        &\big| Z^{(\delta)}_t - a\big|\\
        &= \big| \widetilde\xi -a\big|
        +\int_0^t \operatorname{sgn}\big(Z^{(\delta)}_{u-} -a\big) \diff Z^{(\delta)}_{u}
        +L^a_t \big(Z^{(\delta)}\big) \\
        &\quad+ \int_0^t \Big( \big| Z^{(\delta)}_{u-} + \widetilde\rho\big(Z^{(\delta)}_{u-}\big)- a\big| - \big| Z^{(\delta)}_{u-} - a\big| - \operatorname{sgn}\big(Z^{(\delta)}_{u-} -a\big) \widetilde\rho\big(Z^{(\delta)}_{u-}\big)\Big) \diff N_u.
        \end{aligned}
    \end{equation}
Here we denote $\operatorname{sgn}(x) = \mathds{1}_{(0,\infty)}(x) - \mathds{1}_{(-\infty,0]}(x)$ for all $x\in\R$. By \eqref{pOTE8} we obtain for all $a\in\R$, all $\delta \in(0,\delta_0]$ and all $t\in[0,T]$,
\begin{equation}\label{pOTE9}
    \begin{aligned}
    &L^a_t \big(Z^{(\delta)}\big) = \big|L^a_t \big(Z^{(\delta)}\big)\big|\\
    &\leq \Big|\big| Z^{(\delta)}_t - a\big| - \big| \widetilde\xi -a\big|\Big| 
    + \Big|\int_0^t \operatorname{sgn}\big(Z^{(\delta)}_{u-} -a\big) \diff Z^{(\delta)}_{u}\Big|\\
    &\quad + \Big|\int_0^t \Big( \big| Z^{(\delta)}_{u-} + \widetilde\rho\big(Z^{(\delta)}_{u-}\big)- a\big| - \big| Z^{(\delta)}_{u-} - a\big| - \operatorname{sgn}\big(Z^{(\delta)}_{u-} -a\big) \widetilde\rho\big(Z^{(\delta)}_{u-}\big)\Big) \diff N_u\Big|.
    \end{aligned}
\end{equation}
Estimating the expectation of each summand on the right hand side of \eqref{pOTE9} separately, we obtain for the first one using Lemma \ref{ME} \eqref{LME1} that there exists a constant $\widetilde c_1\in(0,\infty)$ such that for all $a\in\R$, all $\delta \in(0,\delta_0]$ and all $t\in[0,T]$,
\begin{equation}\label{pOTE10}
    \begin{aligned}
    &\E\Big[\Big|\big| Z^{(\delta)}_t - a\big| - \big| \widetilde\xi -a\big|\Big|\Big]
    \leq \widetilde c_1.
    \end{aligned}
\end{equation}
For the second summand of \eqref{pOTE9} we obtain by \eqref{SumAppr} for all $a\in\R$, all $\delta \in(0,\delta_0]$ and all $t\in[0,T]$,
\begin{equation}\label{pOTE12}
    \begin{aligned}
    &\Big|\int_0^t \operatorname{sgn}\big(Z^{(\delta)}_{u-} -a\big) \diff Z^{(\delta)}_{u}\Big|\\
    &\leq \Big|\int_0^t \operatorname{sgn}\big(Z^{(\delta)}_{u-} -a\big)\Big(\sum_{n=0}^{\infty} \widetilde \mu \big(Z^{(\delta)}_{\tau_n}\big)\mathds{1}_{(\tau_n,\tau_{n+1}]}(u)\Big) \diff u\Big|\\
    &\quad+\Big|\int_0^t \operatorname{sgn}\big(Z^{(\delta)}_{u-} -a\big)\Big(\sum_{n=0}^{\infty} \Big( \widetilde\sigma \big(Z^{(\delta)}_{\tau_n}\big)+ \int_{\tau_n}^u \widetilde\sigma \big(Z^{(\delta)}_{\tau_n}\big)d_{\widetilde\sigma} \big(Z^{(\delta)}_{\tau_n}\big)\diff W_s \Big)\mathds{1}_{(\tau_n,\tau_{n+1}]}(u)\Big) \diff W_u\Big|\\
    &\quad+\Big|\int_0^t \operatorname{sgn}\big(Z^{(\delta)}_{u-} -a\big)\widetilde\rho\big(Z^{(\delta)}_{u-}\big) \diff N_u\Big|.
    \end{aligned}
\end{equation}
By Lemma \ref{ME} \eqref{LME1}, we get for all $a\in\R$, all $\delta \in(0,\delta_0]$ and all $t\in[0,T]$,
\begin{equation}\label{pOTE13}
    \begin{aligned}
    &\E\Big[\Big|\int_0^t \operatorname{sgn}\big(Z^{(\delta)}_{u-} -a\big)\Big(\sum_{n=0}^{\infty} \widetilde \mu \big(Z^{(\delta)}_{\tau_n}\big)\mathds{1}_{(\tau_n,\tau_{n+1}]}(u)\Big) \diff u\Big|\Big]
    \leq c_{\widetilde\mu}\, \E\Big[\int_0^t 1+ \big|Z^{(\delta)}_{\underline u}\big|\diff u\Big]\\
    &\leq 2 c_{\widetilde\mu} T \,\E\Big[1+ \sup_{v\in[0,T]}\big|Z^{(\delta)}_{v}\big|^2\Big] 
    \leq 2 c_{\widetilde\mu} T (1+ c_1).
    \end{aligned}
\end{equation}
Using Cauchy-Schwarz inequality, Itô isometry, Lemma \ref{MW} \eqref{MWeq2}, and Lemma \ref{ME} \eqref{LME1} we obtain for all $a\in\R$, all $\delta \in(0,\delta_0]$ and all $t\in[0,T]$,
\begin{equation}\label{pOTE15}
    \begin{aligned}
    &\E\Big[\Big|\int_0^t \operatorname{sgn}\big(Z^{(\delta)}_{u-} -a\big)\Big(\sum_{n=0}^{\infty} \Big( \widetilde\sigma \big(Z^{(\delta)}_{\tau_n}\big) + \int_{\tau_n}^u \widetilde\sigma \big(Z^{(\delta)}_{\tau_n}\big)d_{\widetilde\sigma} \big(Z^{(\delta)}_{\tau_n}\big)\diff W_s \Big)\mathds{1}_{(\tau_n,\tau_{n+1}]}(u)\Big) \diff W_u\Big|\Big]\\
    &\leq \E\Big[\int_0^t \sum_{n=0}^{\infty} \Big|\widetilde\sigma \big(Z^{(\delta)}_{\tau_n}\big) + \int_{\tau_n}^u \widetilde\sigma \big(Z^{(\delta)}_{\tau_n}\big)d_{\widetilde\sigma} \big(Z^{(\delta)}_{\tau_n}\big)\diff W_s \Big|^2 \mathds{1}_{(\tau_n,\tau_{n+1}]}(u)\diff u\Big]^{\frac{1}{2}}\\
    &\leq \Big( 4 c_{\widetilde\sigma}^2 \int_0^t \E\Big[  1 + \big|Z^{(\delta)}_{\underline u}\big|^2 \Big]\diff u
    + 4 c_{\widetilde\sigma}^2 L_{\widetilde\sigma}^2 \int_0^t \E\Big[ 
    \big(1 + \big|Z^{(\delta)}_{\underline u}\big|^2 \big) (W_u -W_{\underline u})^2\Big]\diff u\Big) ^{\frac{1}{2}}\\
    &\leq \Big( 4 c_{\widetilde\sigma}^2 \int_0^t \E\Big[1+ \sup_{t\in[0,T]}\big|Z^{(\delta)}_{t}\big|^2\Big]\diff u
    + 4 c_{\widetilde\sigma}^2 L_{\widetilde\sigma}^2 c_{W_2} \delta \int_0^t \E\Big[  1 + \big|Z^{(\delta)}_{\underline u}\big|^2 \Big]\diff u\Big) ^{\frac{1}{2}}\\
    &\leq \Big( 4 c_{\widetilde\sigma}^2 T + 4 c_{\widetilde\sigma}^2 L_{\widetilde\sigma}^2 c_{W_2} T\Big)^\frac{1}{2} (1+c_1)^\frac{1}{2}.
    \end{aligned}
\end{equation}
Using Cauchy-Schwarz inequality, Lemma \ref{BDGaKunita}, and Lemma \ref{ME} \eqref{LME1} we obtain that there exists $\hat c\in(0,\infty)$ such that for all $a\in\R$, all $\delta \in(0,\delta_0]$ and all $t\in[0,T]$,
\begin{equation}\label{pOTE16}
    \begin{aligned}
    &\E\Big[\Big|\int_0^t \operatorname{sgn}\big(Z^{(\delta)}_{u-} -a\big)\widetilde\rho\big(Z^{(\delta)}_{u-}\big) \diff N_u\Big|\Big]
    \leq \Big(\hat c\,\E\Big[\int_0^t\big| \operatorname{sgn}\big(Z^{(\delta)}_{u-} -a\big)\big|^2\big|\widetilde\rho\big(Z^{(\delta)}_{u-}\big)\big|^2 \diff u\Big]\Big)^\frac{1}{2}\\
    &\leq \Big( 2\hat c\, c_{\widetilde\rho}^2 \int_0^t\E\Big[ \big( 1+ \big|Z^{(\delta)}_{u-}\big|^2\big) \Big]\diff u\Big)^\frac{1}{2}
    \leq \big(2\hat c\, c_{\widetilde\rho}^2  T ( 1+ c_1)\big)^\frac{1}{2}.
    \end{aligned}
\end{equation}
By \eqref{pOTE12}, \eqref{pOTE13}, \eqref{pOTE15}, and \eqref{pOTE16} there exists a constant $\widetilde c_2\in(0,\infty)$ such that for all $a\in\R$, all $\delta \in(0,\delta_0]$ and all $t\in[0,T]$,
\begin{equation}\label{pOTE17}
    \begin{aligned}
    &\E\Big[\Big|\int_0^t \operatorname{sgn}\big(Z^{(\delta)}_{u-} -a\big) \diff Z^{(\delta)}_{u}\Big|\Big]
    \leq \widetilde c_2.
    \end{aligned}
\end{equation}
For the expectation of the third summand of \eqref{pOTE9} we apply Cauchy-Schwarz inequality, Lemma \ref{BDGaKunita}, and Lemma  \ref{ME} \eqref{LME1} to obtain that for all $a\in\R$, all $\delta \in(0,\delta_0]$ and all $t\in[0,T]$,
\begin{equation}\label{pOTE18}
    \begin{aligned}
    & \E\Big[\Big|\int_0^t \Big( \big| Z^{(\delta)}_{u-} + \widetilde\rho\big(Z^{(\delta)}_{u-}\big)- a\big| - \big| Z^{(\delta)}_{u-} - a\big| - \operatorname{sgn}\big(Z^{(\delta)}_{u-} -a\big) \widetilde\rho\big(Z^{(\delta)}_{u-}\big)\Big) \diff N_u
    \Big|\Big]\\
    &\leq \hat c\, \E\Big[\int_0^t \Big| \big| Z^{(\delta)}_{u-} + \widetilde\rho\big(Z^{(\delta)}_{u-}\big)- a\big| - \big| Z^{(\delta)}_{u-} - a\big| - \operatorname{sgn}\big(Z^{(\delta)}_{u-} -a\big) \widetilde\rho\big(Z^{(\delta)}_{u-}\big)\Big|^2 \diff u
    \Big]^\frac{1}{2}\\
    &\leq \hat c\, \E\Big[2\, \int_0^t \Big| \big| Z^{(\delta)}_{u-} + \widetilde\rho\big(Z^{(\delta)}_{u-}\big)- a\big| - \big| Z^{(\delta)}_{u-} - a\big|\Big|^2 + \Big| \operatorname{sgn}\big(Z^{(\delta)}_{u-} -a\big) \widetilde\rho\big(Z^{(\delta)}_{u-}\big)\Big|^2 \diff u
    \Big]^\frac{1}{2}\\
    &\leq \hat c\, \Big(4\, \int_0^t \E\Big[ \big|\widetilde\rho\big(Z^{(\delta)}_{u-}\big)\big|^2\Big]\diff u\Big)^\frac{1}{2}
    \leq \hat c\, \Big(8\,c_{\widetilde\rho}^2 \int_0^t \E\Big[ 1+ \big|Z^{(\delta)}_{u-}\big|^2\Big]\diff u\Big)^\frac{1}{2}
    \leq \hat c\, \Big(8\,c_{\widetilde\rho}^2\, T\, (1+ c_1)\Big)^\frac{1}{2}.
    \end{aligned}
\end{equation}
By \eqref{pOTE9}, \eqref{pOTE10}, \eqref{pOTE17}, and \eqref{pOTE18}  we obtain that there exists a constant $ c_4 \in(0,\infty)$ such that for all $a\in\R$, all $\delta \in(0,\delta_0]$ and all $t\in[0,T]$,
\begin{equation}\label{pOTE19}
    \begin{aligned}
    &\E\big[L^a_t \big(Z^{(\delta)}\big)\big] 
    \leq c_4.
    \end{aligned}
\end{equation}
\end{proof}

\begin{lemma}\label{EVofOTF}
    Let Assumption \ref{assZ} hold, let $f\colon[0,\infty)\to[0,\infty)$ be a $\mathcal{B}([0,\infty))/\mathcal{B}([0,\infty))$-measurable and bounded function, and let $\gamma\in(0,3/2)$. Then there exists a constant $c_5\in(0,\infty)$ such that for all $\delta\in(0,\delta_0]$ and all $\varepsilon \in(0,\varepsilon_0]$,
    \begin{equation}
        \begin{aligned}
            \E\Big[\int_0^T f(d(Z^{(\delta)}_t,\Theta))\mathds{1}_{\Theta^\varepsilon}(Z^{(\delta)}_t)\diff t]
            \leq c_5 \int_0^\varepsilon f(x)\diff x + c_5 \sup_{x\in[0,\varepsilon]}f(x) \big(\varepsilon^{\frac{3}{2}-\gamma } +\delta^{\frac{3}{2}-\gamma}\big).
        \end{aligned}
    \end{equation}
\end{lemma}

\begin{proof}
    It is enough to show that for all $k\in\{1,\ldots,\widetilde m\}$ there exists a constant $\widetilde c_1\in(0,\infty)$ such that for all $\delta\in(0,\delta_0]$ and all $\varepsilon \in(0,\varepsilon_0]$,
    \begin{equation}
        \begin{aligned}
            \E\Big[\int_0^T f(|Z^{(\delta)}_t-\zeta_k| )\mathds{1}_{(\zeta_k-\varepsilon, \zeta_k + \varepsilon)}(Z^{(\delta)}_t)\diff t]
            \leq \widetilde c_1 \int_0^\varepsilon f(x)\diff x + \widetilde c_1 \sup_{x\in[0,\varepsilon]}f(x) \big(\varepsilon^{\frac{3}{2}-\gamma } +\delta^{\frac{3}{2}-\gamma}\big).
        \end{aligned}
    \end{equation}
    Hence we fix $k\in\{1,\ldots,\widetilde m\}$. 
    Since $\widetilde\sigma(\zeta_k)\neq 0$, there exist $\widetilde \varepsilon, \widetilde c_2 \in(0,\infty)$ such that
    \begin{equation}\label{pOTE26}
        \begin{aligned}
        &\inf_{\substack{z\in\R\\|z-\zeta_k|<\widetilde \varepsilon} } \big|\widetilde\sigma(z)\big|^2 >\widetilde c_2
        \end{aligned}
    \end{equation}
    and we obtain for all $\delta\in(0,\delta_0]$ and all $\varepsilon \in(0, \widetilde \varepsilon \wedge \varepsilon_0]$,
    \begin{equation}
        \begin{aligned}\label{pOTE27}
            &\E\Big[\int_0^T f(|Z^{(\delta)}_{t}-\zeta_k|)\mathds{1}_{(\zeta_k-\varepsilon, \zeta_k+\varepsilon)}(Z^{(\delta)}_{t}) \diff 
            t \Big]\\
            &\leq \frac{1}{\widetilde c_2} \E\Big[\int_0^T f(|Z^{(\delta)}_{t}-\zeta_k|)\mathds{1}_{(\zeta_k-\varepsilon, \zeta_k+\varepsilon)}(Z^{(\delta)}_{t})\widetilde\sigma^2 \big(Z^{(\delta)}_{t}\big) \diff 
            t \Big]\\
            &\leq  \frac{1}{\widetilde c_2} \E\Big[\int_0^T f(|Z^{(\delta)}_{t}-\zeta_k|)\mathds{1}_{(\zeta_k-\varepsilon, \zeta_k+\varepsilon)}(Z^{(\delta)}_{t})\\ 
            &\quad\quad\quad\quad \quad\quad\quad\quad \Big(\sum_{n=0}^{\infty}\Big( \widetilde\sigma \big(Z^{(\delta)}_{\tau_n}\big) + \int_{\tau_n}^t \widetilde\sigma \big(Z^{(\delta)}_{\tau_n}\big)d_{\widetilde\sigma} \big(Z^{(\delta)}_{\tau_n}\big)\diff W_s \Big)\mathds{1}_{(\tau_n,\tau_{n+1}]}(t) \Big)^2 \diff t \Big]\\
            &\quad +  \frac{1}{\widetilde c_2} \E\Big[\int_0^T f(|Z^{(\delta)}_{t}-\zeta_k|)\mathds{1}_{(\zeta_k-\varepsilon, \zeta_k+\varepsilon)}(Z^{(\delta)}_{t})\\
            &\quad\quad\quad\quad \quad\quad\quad\quad \Big( \widetilde\sigma^2 \big(Z^{(\delta)}_{t}\big) - \Big(\sum_{n=0}^{\infty} \Big( \widetilde\sigma \big(Z^{(\delta)}_{\tau_n}\big) + \int_{\tau_n}^t \widetilde\sigma \big(Z^{(\delta)}_{\tau_n}\big)d_{\widetilde\sigma} \big(Z^{(\delta)}_{\tau_n}\big)\diff W_s \Big)\mathds{1}_{(\tau_n,\tau_{n+1}]}(t) \Big)^2 \Big)\diff t \Big].
        \end{aligned}
    \end{equation}
    
The continuous martingale part of \eqref{SumAppr} is given by
\begin{equation}\label{pOTE20}
    \begin{aligned}
    &\Big(\int_0^t \sum_{n=0}^{\infty} \Big( \widetilde\sigma \big(Z^{(\delta)}_{\tau_n}\big) + \int_{\tau_n}^u \widetilde\sigma \big(Z^{(\delta)}_{\tau_n}\big)d_{\widetilde\sigma} \big(Z^{(\delta)}_{\tau_n}\big)\diff W_s \Big)\mathds{1}_{(\tau_n,\tau_{n+1}]}(u)\diff W_u\Big)_{t\in[0,T]}
    \end{aligned}
\end{equation}
and its predictable quadratic variation by 
\begin{equation}\label{pOTE20a}
    \begin{aligned}
    &\Big\langle \int_0^\cdot \sum_{n=0}^{\infty} \Big( \widetilde\sigma \big(Z^{(\delta)}_{\tau_n}\big) + \int_{\tau_n}^u \widetilde\sigma \big(Z^{(\delta)}_{\tau_n}\big)d_{\widetilde\sigma} \big(Z^{(\delta)}_{\tau_n}\big)\diff W_s \Big)\mathds{1}_{(\tau_n,\tau_{n+1}]}(u)\diff W_u\Big\rangle_t\\
    &= \int_0^t \Big(\sum_{n=0}^{\infty} \Big( \widetilde\sigma \big(Z^{(\delta)}_{\tau_n}\big) + \int_{\tau_n}^u \widetilde\sigma \big(Z^{(\delta)}_{\tau_n}\big)d_{\widetilde\sigma} \big(Z^{(\delta)}_{\tau_n}\big)\diff W_s \Big)\mathds{1}_{(\tau_n,\tau_{n+1}]}(u)\Big)^2\diff u,
    \end{aligned}
\end{equation}
see \cite[Theorem 88]{situ2005}.
Hence we use that $ x\mapsto f(|x-\zeta_k|)\mathds{1}_{(\zeta_k-\varepsilon, \zeta_k+\varepsilon)}(x)$ is a Borel-measurable and bounded function, apply \cite[Lemma 159]{situ2005} and Lemma \ref{LocTimeEst} to obtain for all $\delta\in(0,\delta_0]$, all $\varepsilon \in(0, \widetilde \varepsilon \wedge \varepsilon_0]$, and all $t\in[0,T]$,
\begin{equation}\label{pOTE21}
    \begin{aligned}
    &\E\Big[\int_0^t f(|Z^{(\delta)}_{u}-\zeta_k|)\mathds{1}_{(\zeta_k-\varepsilon, \zeta_k+\varepsilon)}(Z^{(\delta)}_{u})\\
    &\quad\quad\quad\quad \quad\quad\quad\quad \quad\quad\quad\quad \cdot \Big(\sum_{n=0}^{\infty} \Big( \widetilde\sigma \big(Z^{(\delta)}_{\tau_n}\big) + \int_{\tau_n}^u \widetilde\sigma \big(Z^{(\delta)}_{\tau_n}\big)d_{\widetilde\sigma} \big(Z^{(\delta)}_{\tau_n}\big)\diff W_s \Big)\mathds{1}_{(\tau_n,\tau_{n+1}]}(u) \Big)^2 \diff u \Big]\\
    &=\E\Big[\int_0^t f(|Z^{(\delta)}_{u}-\zeta_k|)\mathds{1}_{(\zeta_k-\varepsilon, \zeta_k+\varepsilon)}(Z^{(\delta)}_{u}) \diff \Big\langle \int_0^\cdot \sum_{n=0}^{\infty} \Big( \widetilde\sigma \big(Z^{(\delta)}_{\tau_n}\big) \\
    &\quad\quad\quad\quad \quad\quad\quad\quad \quad\quad\quad\quad \quad\quad\quad\quad \quad\quad\quad+ \int_{\tau_n}^v \widetilde\sigma \big(Z^{(\delta)}_{\tau_n}\big)d_{\widetilde\sigma} \big(Z^{(\delta)}_{\tau_n}\big)\diff W_s \Big)\mathds{1}_{(\tau_n,\tau_{n+1}]}(v)\diff W_v\Big\rangle_u\Big]\\
    &=\E\Big[\int_\R f(|a-\zeta_k|)\mathds{1}_{(\zeta_k-\varepsilon, \zeta_k+\varepsilon)}(a) L^a_t \big(Z^{(\delta)}\big) \diff a \Big]
    \leq 2 c_4 \int_0^{\varepsilon} f(a) \diff a.
    \end{aligned}
\end{equation}
Using $|a^2-b^2| = |a-b|\cdot|a+b|$ for all $a,b\in\R$ yields for all $\delta\in(0,\delta_0]$ and all $u\in[0,t]$, 
\begin{equation}\label{pOTE22a}
    \begin{aligned}
    &\E\Big[ \Big| \widetilde\sigma^2\big(Z^{(\delta)}_{u} \big)- \Big(\sum_{n=0}^{\infty} \Big( \widetilde\sigma \big(Z^{(\delta)}_{\tau_n}\big) + \int_{\tau_n}^u \widetilde\sigma \big(Z^{(\delta)}_{\tau_n}\big)d_{\widetilde\sigma} \big(Z^{(\delta)}_{\tau_n}\big)\diff W_s \Big)\mathds{1}_{(\tau_n,\tau_{n+1}]}(u) \Big)^2\Big|^q\Big]\\
    &\leq\E\Big[\Big| \widetilde\sigma\big(Z^{(\delta)}_{u} \big)- \sum_{n=0}^{\infty} \Big( \widetilde\sigma \big(Z^{(\delta)}_{\tau_n}\big) + \int_{\tau_n}^u \widetilde\sigma \big(Z^{(\delta)}_{\tau_n}\big)d_{\widetilde\sigma} \big(Z^{(\delta)}_{\tau_n}\big)\diff W_s \Big)\mathds{1}_{(\tau_n,\tau_{n+1}]}(u) \Big|^{2q}\Big]^{\frac{1}{2}}\\
    &\quad \cdot \E\Big[\Big| \widetilde\sigma\big(Z^{(\delta)}_{u} \big)+ \sum_{n=0}^{\infty} \Big( \widetilde\sigma \big(Z^{(\delta)}_{\tau_n}\big) + \int_{\tau_n}^u \widetilde\sigma \big(Z^{(\delta)}_{\tau_n}\big)d_{\widetilde\sigma} \big(Z^{(\delta)}_{\tau_n}\big)\diff W_s \Big)\mathds{1}_{(\tau_n,\tau_{n+1}]}(u)\Big|^{2q}\Big]^{\frac{1}{2}}.
    \end{aligned}
\end{equation}
For the first factor of \eqref{pOTE22a} we use
\begin{equation}\label{pOTE23}
    \begin{aligned}
    &\Big| \widetilde\sigma\big(Z^{(\delta)}_{u} \big)- \sum_{n=0}^{\infty} \Big( \widetilde\sigma \big(Z^{(\delta)}_{\tau_n}\big) + \int_{\tau_n}^u \widetilde\sigma \big(Z^{(\delta)}_{\tau_n}\big)d_{\widetilde\sigma} \big(Z^{(\delta)}_{\tau_n}\big)\diff W_s \Big)\mathds{1}_{(\tau_n,\tau_{n+1}]}(u) \Big|\\
    &= \sum_{n=0}^{\infty} \Big| \widetilde\sigma\big(Z^{(\delta)}_{u} \big)- \widetilde\sigma \big(Z^{(\delta)}_{\tau_n}\big) - \int_{\tau_n}^u \widetilde\sigma \big(Z^{(\delta)}_{\tau_n}\big)d_{\widetilde\sigma} \big(Z^{(\delta)}_{\tau_n}\big)\diff W_s \Big|\mathds{1}_{(\tau_n,\tau_{n+1}]}(u) \\
    &\leq  L_{\widetilde\sigma} \big| Z^{(\delta)}_{u} -Z^{(\delta)}_{\underline u}\big| + L_{\widetilde\sigma} c_{\widetilde\sigma} \big(1+ \big| Z^{(\delta)}_{\underline u}\big|\big) |W_u-W_{\underline u}|
    \end{aligned}
\end{equation}
and apply Lemma \ref{MW}, and Lemma \ref{ME} \eqref{LME1} and \eqref{LME2} to obtain for all $\delta\in(0,\delta_0]$ and all $u\in[0,t]$,
\begin{equation}\label{pOTE23a}
    \begin{aligned}
    &\E\Big[\Big| \widetilde\sigma\big(Z^{(\delta)}_{u} \big)- \sum_{n=0}^{\infty} \Big( \widetilde\sigma \big(Z^{(\delta)}_{\tau_n}\big) + \int_{\tau_n}^u \widetilde\sigma \big(Z^{(\delta)}_{\tau_n}\big)d_{\widetilde\sigma} \big(Z^{(\delta)}_{\tau_n}\big)\diff W_s \Big)\mathds{1}_{(\tau_n,\tau_{n+1}]}(u) \Big|^{2q}\Big]^{\frac{1}{2}}\\
    &\leq 2^{q-\frac{1}{2}}\Big( L_{\widetilde\sigma}^{2q}\, \E\Big[  \big| Z^{(\delta)}_{u} -Z^{(\delta)}_{\underline u}\big|^{2q} \Big] 
    + 2^{q-\frac{1}{2}} L_{\widetilde\sigma}^{2q} c_{\widetilde\sigma}^{2q}\, \E\Big[ \big(1+ \big| Z^{(\delta)}_{\underline u}\big|^{2q}\big) |W_u-W_{\underline u}|^{2q}\Big]\Big)^\frac{1}{2}\\
    &\leq 2^{q-\frac{1}{2}}\Big( L_{\widetilde\sigma}^{2q} c_2\delta^{q} 
    + 2^{q-\frac{1}{2}} L_{\widetilde\sigma}^{2q} c_{\widetilde\sigma}^{2q}  c_{W_{2q}} \delta^q\, \E\Big[1+ \big| Z^{(\delta)}_{\underline u}\big|^{2q} \Big]\Big)^\frac{1}{2}\\
    &\leq 2^{q-\frac{1}{2}} \Big( L_{\widetilde\sigma}^{2q} c_2
    +2^{q-\frac{1}{2}} L_{\widetilde\sigma}^{2q} c_{\widetilde\sigma}^{2q} c_{W_{2q}} (1+c_1) \Big)^\frac{1}{2}\delta^{\frac{q}{2}}.
    \end{aligned}
\end{equation}
For the second factor of \eqref{pOTE22a} we have
\begin{equation}
    \begin{aligned}
    &\Big| \widetilde\sigma\big(Z^{(\delta)}_{u} \big)+ \sum_{n=0}^{\infty} \Big( \widetilde\sigma \big(Z^{(\delta)}_{\tau_n}\big) + \int_{\tau_n}^u \widetilde\sigma \big(Z^{(\delta)}_{\tau_n}\big)d_{\widetilde\sigma} \big(Z^{(\delta)}_{\tau_n}\big)\diff W_s \Big)\mathds{1}_{(\tau_n,\tau_{n+1}]}(u) \Big|\\
    &\leq  c_{\widetilde\sigma}\big(1+ \big| Z^{(\delta)}_{u}\big|\big) + c_{\widetilde\sigma} \big(1+ \big| Z^{(\delta)}_{\underline u}\big|\big)
    + L_{\widetilde\sigma} c_{\widetilde\sigma} \big(1+ \big| Z^{(\delta)}_{\underline u}\big|\big) |W_u-W_{\underline u}|,
    \end{aligned}
\end{equation}
which together with Lemma \ref{MW}, and Lemma \ref{ME} \eqref{LME1} yields that there exists a constant $\widetilde c_3\in(0,\infty)$ such that for all $\delta \in(0,\delta_0]$ and all $u\in[0,t]$,
\begin{equation}\label{pOTE24a}
    \begin{aligned}
    &\E\Big[\Big| \widetilde\sigma\big(Z^{(\delta)}_{u} \big)+ \sum_{n=0}^{\infty} \Big( \widetilde\sigma \big(Z^{(\delta)}_{\tau_n}\big) + \int_{\tau_n}^u \widetilde\sigma \big(Z^{(\delta)}_{\tau_n}\big)d_{\widetilde\sigma} \big(Z^{(\delta)}_{\tau_n}\big)\diff W_s \Big)\mathds{1}_{(\tau_n,\tau_{n+1}]}(u)\Big|^2\Big]^{\frac{1}{2}}\\
    &\leq \widetilde c_3\, \Big( 2^{2q-1} c_{\widetilde\sigma}^{2q}\, \E\Big[ 1+ \big| Z^{(\delta)}_{u}\big|^{2q}\Big]
    +2^{2q-1} c_{\widetilde\sigma}^{2q}\, \E\Big[1+ \big| Z^{(\delta)}_{\underline u}\big|^{2q}\Big]\\
    &\quad\quad\quad\quad +2^{2q-1} L_{\widetilde\sigma}^{2q} c_{\widetilde\sigma}^{2q} \, \E\Big[\big(1+ \big| Z^{(\delta)}_{\underline u}\big|^{2q}\big) |W_u-W_{\underline u}|^{2q}\Big]\Big)^{\frac{1}{2}}\\
    &\leq \widetilde c_3\, \Big( 2^{2q} c_{\widetilde\sigma}^{2q} \,\E\Big[ 1+\sup_{t\in[0,T]} \big| Z^{(\delta)}_{t}\big|^{2q} \Big]
    +2^{2q-1} L_{\widetilde\sigma}^{2q} c_{\widetilde\sigma}^{2q} c_{W_{2q}} \delta^q\, \E\Big[ 1+ \big| Z^{(\delta)}_{\underline u}\big|^{2q}\Big]\Big)^{\frac{1}{2}}\\
    &\leq \widetilde c_3\, \Big( 4^{2q-1}  c_{\widetilde\sigma}^{2q} 
    +2^{2q-1} L_{\widetilde\sigma}^{2q} c_{\widetilde\sigma}^{2q} c_{W_{2q}} T \Big)^{\frac{1}{2}}(1+c_1)^\frac{1}{2} .
    \end{aligned}
\end{equation}
Plugging \eqref{pOTE23a} and \eqref{pOTE24a} into \eqref{pOTE22a} yields that there exists a constant $\widetilde c_4\in(0,\infty)$ such that for all $\delta \in(0,\delta_0]$ and all $u\in[0,t]$,
\begin{equation}\label{pOTE25}
    \begin{aligned}
    &\E\Big[ \Big| \widetilde\sigma^2\big(Z^{(\delta)}_{u} \big)- \Big(\sum_{n=0}^{\infty} \Big( \widetilde\sigma \big(Z^{(\delta)}_{\tau_n}\big) + \int_{\tau_n}^u \widetilde\sigma \big(Z^{(\delta)}_{\tau_n}\big)d_{\widetilde\sigma} \big(Z^{(\delta)}_{\tau_n}\big)\diff W_s \Big)\mathds{1}_{(\tau_n,\tau_{n+1}]}(u) \Big)^2\Big|^q\Big]
    \leq \widetilde c_4 \delta^{\frac{q}{2}}.
    \end{aligned}
\end{equation}
Next we use \eqref{pOTE21}, Hölder's inequality with $q\in(1,\infty)$, and \eqref{pOTE25} in \eqref{pOTE27} to obtain for all $\delta \in(0,\delta_0]$ and all $\varepsilon \in(0, \widetilde \varepsilon \wedge \varepsilon_0]$,
\begin{equation}\label{pOTE28}
    \begin{aligned}
        &\E\Big[\int_0^T f(|Z^{(\delta)}_{t}-\zeta_k|)\mathds{1}_{(\zeta_k-\varepsilon, \zeta_k+\varepsilon)}(Z^{(\delta)}_{t}) \diff 
        t \Big]\\
        &\leq \frac{1}{\widetilde c_2}2 c_4\int_0^{\varepsilon} f(a) \diff a\\
        &\quad +  \frac{1}{\widetilde c_2}  \sup_{x\in[0,\varepsilon]} f(x)  \Big(\int_0^T\E\Big[ \mathds{1}_{(\zeta_k-\varepsilon, \zeta_k+\varepsilon)}(Z^{(\delta)}_{t}) \Big] \diff t \Big)^\frac{1}{q}\\
        &\quad\,\,\,\, \Big(\int_0^T \E\Big[\Big( \widetilde\sigma^2 \big(Z^{(\delta)}_{t}\big) - \Big(\sum_{n=0}^{\infty} \Big( \widetilde\sigma \big(Z^{(\delta)}_{\tau_n}\big) + \int_{\tau_n}^t \widetilde\sigma \big(Z^{(\delta)}_{\tau_n}\big)d_{\widetilde\sigma} \big(Z^{(\delta)}_{\tau_n}\big)\diff W_s \Big)\mathds{1}_{(\tau_n,\tau_{n+1}]}(t) \Big)^{2} \Big)^{\frac{q}{q-1}}\Big]\diff t\Big)^{\frac{q-1}{q}}  \\
        &\leq \frac{1}{\widetilde c_2}2 c_4\int_0^{\varepsilon} f(a) \diff a
        +  \frac{1}{\widetilde c_2}  \widetilde c_4 \delta^{1/2} \sup_{x\in[0,\varepsilon]} f(x)  \Big(\int_0^T \P( |Z^{(\delta)}_{t} -\zeta_k|\leq \varepsilon) \diff t \Big)^\frac{1}{q}.
    \end{aligned}
\end{equation}
For $f\equiv 1$, \eqref{pOTE28} yields that for all $q\in(1,\infty)$ there exists a constant $\widetilde c_5\in(0,\infty)$ such that for all $\delta\in(0,\delta_0]$ and all $\varepsilon\in(0,\widetilde \varepsilon\wedge \varepsilon]$,
\begin{equation}\label{pOTE29}
    \begin{aligned}
        &\int_0^T \P( |Z^{(\delta)}_{t} -\zeta_k|\leq \varepsilon) \diff t 
        \leq \widetilde c_5 \varepsilon
        +  \widetilde c_5  \delta^{1/2} \Big(\int_0^T \P( |Z^{(\delta)}_{t} -\zeta_k|\leq \varepsilon) \diff t \Big)^\frac{1}{q}.
    \end{aligned}
\end{equation}
Iterating this and then estimating $\P( |Z^{(\delta)}_{t} -\zeta_k|\leq \varepsilon)\leq 1$ we obtain for all $\delta\in(0,\delta_0]$ and all $\varepsilon\in(0,\widetilde \varepsilon\wedge \varepsilon]$,
\begin{equation}
    \begin{aligned}
        &\int_0^T \P( |Z^{(\delta)}_{t} -\zeta_k|\leq \varepsilon) \diff t 
        \leq \widetilde c_5 \varepsilon
        +  \widetilde c_5  \delta^{1/2} \Big(\widetilde c_5 \varepsilon
        +  \widetilde c_5  \delta^{1/2} T \Big)^\frac{1}{q}.
    \end{aligned}
\end{equation}
Using Young's inequality and Jensen's inequality we obtain that for all $q\in(1,\infty)$ there exists a constant $\widetilde c_9\in(0,\infty)$ such that for all $\delta\in(0,\delta_0]$ and all $\varepsilon\in(0,\widetilde \varepsilon\wedge \varepsilon_0]$,
\begin{equation}\label{pOTE30}
    \begin{aligned}
        &\int_0^T \P( |Z^{(\delta)}_{t} -\zeta_k|\leq \varepsilon) \diff t 
        \leq \widetilde c_5 \varepsilon
        +  \widetilde c_5^\frac{q+1}{q} \frac{q}{q+1} \delta^{\frac{q+1}{2q}} + \Big(\widetilde c_5 \varepsilon
        +  \widetilde c_5  \delta^{1/2} T \Big)^{\frac{1}{q}(q+1)} \frac{1}{q+1}\\
        &\leq \widetilde c_5 \varepsilon
        +  \widetilde c_5^\frac{q+1}{q} \frac{q}{q+1} \delta^{\frac{q+1}{2q}} + 2^{\frac{q+1}{q}-1} \Big(\widetilde c_5^{\frac{q+1}{q}} \varepsilon^{\frac{q+1}{q}}
        +  \widetilde c_5^{\frac{q+1}{q}}  \delta^{\frac{q+1}{2q}} T^{\frac{q+1}{q}} \Big)\frac{1}{q+1}
        \leq \widetilde c_6 \varepsilon + \widetilde c_6 \delta^{\frac{1}{2}+ \frac{1}{2q}}.
    \end{aligned}
\end{equation}
Plugging \eqref{pOTE30} in \eqref{pOTE28} we get that there exists a constant $\widetilde c_{7}\in(0,\infty)$ such that for all $\delta\in(0,\delta_0]$ and all $\varepsilon\in(0,\widetilde \varepsilon\wedge \varepsilon]$,
\begin{equation}\label{pOTE31}
    \begin{aligned}
        &\E\Big[\int_0^T f(|Z^{(\delta)}_{t}-\zeta_k|)\mathds{1}_{(\zeta_k-\varepsilon, \zeta_k+\varepsilon)}(Z^{(\delta)}_{t}) \diff 
        t \Big]\\
        &\leq \frac{1}{\widetilde c_2}2 c_4\int_0^{\varepsilon} f(a) \diff a
        +  \frac{1}{\widetilde c_2}  \widetilde c_4 \delta^{1/2} \sup_{x\in[0,\varepsilon]} f(x)  \Big(\widetilde c_6 \varepsilon + \widetilde c_6 \delta^{\frac{1}{2}+ \frac{1}{2q}} \Big)^\frac{1}{q}  \\
        &\leq \widetilde c_{7}\int_0^{\varepsilon} f(a) \diff a
        +  \widetilde c_{7} \sup_{x\in[0,\varepsilon]} f(x)  \delta^{1/2} \Big(\varepsilon^\frac{1}{q} + \delta^{\frac{1}{2q}+ \frac{1}{2q^2}} \Big).
    \end{aligned}
\end{equation}
By Young's inequality we have for all $\delta\in(0,\delta_0]$ and all $\varepsilon\in(0,\widetilde \varepsilon\wedge \varepsilon]$,
\begin{equation}\label{pOTE32}
    \begin{aligned}
        & \delta^{1/2}\varepsilon^\frac{1}{q} \leq \frac{1}{3}\delta^{3/2} + \frac{2}{3}\varepsilon^{\frac{3}{2q}}.
    \end{aligned}
\end{equation}
Hence we obtain that there exists a constant $\widetilde c_{8}\in(0,\infty)$ such that for all $\delta\in(0,\delta_0]$ and all $\varepsilon\in(0,\widetilde \varepsilon\wedge \varepsilon]$,
\begin{equation}\label{pOTE33}
    \begin{aligned}
        &\E\Big[\int_0^T f(|Z^{(\delta)}_{t}-\zeta_k|)\mathds{1}_{(\zeta_k-\varepsilon, \zeta_k+\varepsilon)}(Z^{(\delta)}_{t}) \diff 
        t \Big]\\
        &\leq \widetilde c_{8}\int_0^{\varepsilon} f(a) \diff a
        +  \widetilde c_{8} \sup_{x\in[0,\varepsilon]} f(x) \Big(\varepsilon^\frac{3}{2q} + \delta^{\frac{1}{2}+\frac{1}{2q}+ \frac{1}{2q^2}} +\delta^\frac{3}{2} \Big).
    \end{aligned}
\end{equation}
This implies the claim for all $\varepsilon\in(0,\widetilde \varepsilon\wedge \varepsilon_0]$ as $\frac{3}{2q} \uparrow \frac{3}{2}$ for $q \downarrow 1$ and  $\frac{1}{2} + \frac{1}{2q} +\frac{1}{2q^2} \uparrow \frac{3}{2}$ for $q \downarrow 1$.
This is enough since in case that $\widetilde \varepsilon< \varepsilon_0 $ for  $\varepsilon\in(\widetilde \varepsilon,\varepsilon_0]$ we calculate for all $\delta\in(0,\delta_0]$,
\begin{equation}\label{pOTE34}
    \begin{aligned}
        &\E\Big[\int_0^T f(|Z^{(\delta)}_{t}-\zeta_k|)\mathds{1}_{(\zeta_k-\varepsilon, \zeta_k+\varepsilon)}(Z^{(\delta)}_{t}) \diff 
        t \Big]
        \leq \sup_{x\in[0,\varepsilon]} f(x) \E\Big[\int_0^T \mathds{1}_{(\zeta_k-\varepsilon, \zeta_k+\varepsilon)}(Z^{(\delta)}_{t}) \diff 
        t \Big]\\
        &\leq \sup_{x\in[0,\varepsilon]} f(x) T \varepsilon^{-\frac{3}{2}}  \varepsilon^{\frac{3}{2}} 
        \leq \sup_{x\in[0,\varepsilon]} f(x) T \widetilde \varepsilon^{-\frac{3}{2}}  \varepsilon^{\frac{3}{2}}.
    \end{aligned}
\end{equation}
This proves the claim.
\end{proof}
\newpage

\begin{lemma}\label{LemProbEst}
    Let Assumption \ref{assZ} hold, let $\alpha \in(0,\infty)$, and $q\in[1,\infty)$. Then there exists a constant $c_6 \in(0,\infty)$ such that for all $\delta\in(0,\delta_0]$ and all $t\in[0,T]$,
    \begin{itemize}
        \item[\namedlabel{ProbEsti}{(i)}] $\P(|Z^{(\delta)}_{t} -Z^{(\delta)}_{\underline{t}} | \geq \alpha \varepsilon_2^\delta, Z^{(\delta)}_{\underline{t}} \in \Theta^{\varepsilon_2^\delta}) \leq c_6 \delta^q$;
        \item[\namedlabel{ProbEstii}{(ii)}] $\P(|Z^{(\delta)}_{t} -Z^{(\delta)}_{\underline{t}} | \geq \alpha d( Z^{(\delta)}_{\underline{t}},\Theta),  Z^{(\delta)}_{\underline{t}} \in \Theta^{\varepsilon_1^\delta}\setminus \Theta^{\varepsilon_2^\delta}) \leq c_6 \delta^q$;
        \item[\namedlabel{ProbEstiii}{(iii)}] $\P(|Z^{(\delta)}_{t} -Z^{(\delta)}_{\underline{t}} | \geq \alpha\varepsilon_1^\delta,   Z^{(\delta)}_{\underline{t}} \in \Theta^{\varepsilon_0}\setminus \Theta^{\varepsilon_1^\delta}) \leq c_6 \delta^q$.
    \end{itemize}
\end{lemma}

\begin{proof}
    Define $\Phi\colon \R\times C([0,\infty),\R) \to C([0,\infty),\R)$ for all $y\in\R$, $u\in C([0,\infty),\R)$, and $t\in[0,\infty)$ by
    \begin{equation}
        \Phi(y,u)(t) = y+\widetilde \mu(y) t + \widetilde \sigma(y) u(t) + \frac{1}{2} \widetilde \sigma(y) d_{\widetilde \sigma}(y) (u^2(t)-t).
    \end{equation}
    Here $C([0,\infty),\R)$ is the space of all continuous functions from $[0,\infty)$ to $\R$.
    From \cite[Proof of Lemma 7]{Yaroslavtseva2022} we know the following three estimates. There exists a constant $\widetilde c_1\in(0,\infty)$ such that for all $\delta\in(0,\delta_0]$ and all $y\in\Theta^{\varepsilon_2^\delta}$,
    \begin{equation}
        \P(\sup_{s\in[0,h^\delta(y)]} |\Phi(y,W)(s) -y | \geq \alpha\varepsilon_2^\delta)\leq \widetilde c_1 \delta^q.
    \end{equation}
    There exists a constant $\widetilde c_2\in(0,\infty)$ such that for all $\delta\in(0,\delta_0]$ and all $y\in\Theta^{\varepsilon_1^\delta}\setminus \Theta^{\varepsilon_2^\delta}$,
    \begin{equation}
        \P(\sup_{s\in[0,h^\delta(y)]} |\Phi(y,W)(s) -y | \geq \alpha d(y,\Theta)) \leq \widetilde c_2 \delta^q.
    \end{equation}
    There exists a constant $\widetilde c_3\in(0,\infty)$ such that for all $\delta\in(0,\delta_0]$ and all $y\in\Theta^{\varepsilon_0}\setminus \Theta^{\varepsilon_1^\delta}$,
    \begin{equation}
        \P(\sup_{s\in[0,h^\delta(y)]} |\Phi(y,W)(s) -y | \geq \alpha\varepsilon_1^\delta)\leq \widetilde c_3 \delta^q.
    \end{equation}
    With these preliminary results we continue our estimations. For all $\delta\in[0,\delta_0)$ and all $t\in[0,T]$ we use Lemma \ref{BMRandomTime} to obtain that
    \begin{equation}
        \begin{aligned}
            &\P(|Z^{(\delta)}_{t} -Z^{(\delta)}_{\underline{t}} | \geq \alpha \varepsilon_2^\delta, Z^{(\delta)}_{\underline{t}} \in \Theta^{\varepsilon_2^\delta})
            = \P(| \Phi(Z^{(\delta)}_{\underline{t}},W^{\underline{t}})(t-\underline{t}) - Z^{(\delta)}_{\underline{t}}| \geq \alpha \varepsilon_2^\delta, Z^{(\delta)}_{\underline{t}} \in \Theta^{\varepsilon_2^\delta})\\
            &\leq \P(\sup_{s\in[0,h^{\delta}(Z^{(\delta)}_{\underline{t}})]} | \Phi(Z^{(\delta)}_{\underline{t}},W^{\underline{t}})(s) - Z^{(\delta)}_{\underline{t}}| \geq \alpha \varepsilon_2^\delta, Z^{(\delta)}_{\underline{t}} \in \Theta^{\varepsilon_2^\delta})\\
            &= \int_{\Theta^{\varepsilon_2^\delta}} \P(\sup_{s\in[0,h^{\delta}(y)]} | \Phi(y,W^{\underline{t}})(s) - y| \geq \alpha \varepsilon_2^\delta) \diff \P^{Z^{(\delta)}_{\underline{t}}} (y)
            \leq \widetilde c_1 \delta^q.
        \end{aligned}        
    \end{equation}
    With similar arguments we obtain that for all $\delta\in[0,\delta_0)$ and all $t\in[0,T]$
    \begin{equation}
        \begin{aligned}
            &\P(|Z^{(\delta)}_{t} -Z^{(\delta)}_{\underline{t}} | \geq \alpha d( Z^{(\delta)}_{\underline{t}},\Theta),  Z^{(\delta)}_{\underline{t}} \in \Theta^{\varepsilon_1^\delta}\setminus \Theta^{\varepsilon_2^\delta})\\
            & \leq \int_{\Theta^{\varepsilon_1^\delta}\setminus \Theta^{\varepsilon_2^\delta}} \P(\sup_{s\in[0,h^{\delta}(y)]} | \Phi(y,W^{\underline{t}})(s) - y| \geq \alpha d( y,\Theta) ) \diff \P^{Z^{(\delta)}_{\underline{t}}} (y)
            \leq \widetilde c_2 \delta^q,
        \end{aligned}        
    \end{equation}
    and that for all $\delta\in[0,\delta_0)$ and all $t\in[0,T]$
    \begin{equation}
        \begin{aligned}
            &\P(|Z^{(\delta)}_{t} -Z^{(\delta)}_{\underline{t}} | \geq \alpha\varepsilon_1^\delta,   Z^{(\delta)}_{\underline{t}} \in \Theta^{\varepsilon_0}\setminus \Theta^{\varepsilon_1^\delta})\\
            & \leq \int_{\Theta^{\varepsilon_0}\setminus \Theta^{\varepsilon_1^\delta}} \P(\sup_{s\in[0,h^{\delta}(y)]} | \Phi(y,W^{\underline{t}})(s) - y| \geq \alpha \varepsilon_1^\delta ) \diff \P^{Z^{(\delta)}_{\underline{t}}} (y)
            \leq \widetilde c_3 \delta^q.
        \end{aligned}        
    \end{equation}
\end{proof}

\begin{lemma}\label{MWwithInd}
    Let Assumption \ref{assZ} hold. Then for all $p\in\N_0$ and  $q\in\N$ there exists a constant $c_7\in(0,\infty)$ such that for all $t\in[0,T]$ and all $\delta \in(0,\delta_0]$, 
    \begin{equation}\label{MWwithIndeq}
        \begin{aligned}
        &\E\Big[ \big(1+ \big|Z^{(\delta)}_{\underline{t}}\big|^p\big)   |W_t-W_{\underline{t}}|^q \mathds{1}_{\Theta^{\varepsilon_2^\delta}}(Z_{\underline t}^{(\delta)})\Big]
        \leq c_7  \delta^q \log^{2q}(\delta^{-1}).
        \end{aligned}
    \end{equation}
\end{lemma}

The proof of this lemma works similarly to the proof of Lemma \ref{MW}. For completeness, it is stated in Appendix \ref{appendix}. 
Next, we define the set
\begin{equation}\label{setS}
    S = \Bigl(\bigcup_{k=1}^{\widetilde m+1} (\zeta_{k-1},\zeta_k)^2\Bigr)^c
\end{equation} 
and prove $p$-th moment estimates of the time average of $|Z_{t}^{(\delta)} - Z_{\underline t}^{(\delta)}|^2$ subject to an exception point $\zeta_k$, $k\in\{1,...,\widetilde m\}$ lying between $Z_t^{(\delta)}$ and $Z_{\underline t}^{(\delta)}$.

\begin{lemma}\label{LemExepSet}
    Let Assumption \ref{assZ} hold and let $p\in [1,\infty)$. Then  there exists  $c_8\in(0, \infty)$ such that for all $\delta\in(0,\delta_0]$, 
    \begin{equation}
        \E\Big[\Big|\int_0^T  |Z_{t}^{(\delta)} - Z_{\underline t}^{(\delta)}|^2\cdot \mathds{1}_{S}(Z_{t}^{(\delta)},Z_{\underline t}^{(\delta)}) \diff t \Big|^p\Big]^{1/p}\leq c_8 \delta^{2}. 
    \end{equation}
\end{lemma}

\begin{proof}
    Since $\R = (\Theta^{\varepsilon_0})^c \cup (\Theta^{\varepsilon_0}\setminus \Theta^{\varepsilon_1^\delta}) \cup (\Theta^{\varepsilon_1^\delta}\setminus \Theta^{\varepsilon_2^\delta}) \cup \Theta^{\varepsilon_2^\delta}$, we split for all $\delta\in(0,\delta_0]$
    \begin{equation}
        \begin{aligned}
            &\E\Big[\Big|\int_0^T |Z_{t}^{(\delta)} - Z_{\underline t}^{(\delta)}|^2\cdot \mathds{1}_{S}(Z_{t}^{(\delta)},Z_{\underline t}^{(\delta)}) \diff t \Big|^p\Big]\\
            & \leq T^{p-1}\E\Big[\int_0^T  |Z_{t}^{(\delta)} - Z_{\underline t}^{(\delta)}|^{2p}\cdot \mathds{1}_{S}(Z_{t}^{(\delta)},Z_{\underline t}^{(\delta)}) \diff t \Big]\\
            &= T^{p-1} \Big( \E\Big[\int_0^T  |Z_{t}^{(\delta)} - Z_{\underline t}^{(\delta)}|^{2p}\cdot \mathds{1}_{S}(Z_{t}^{(\delta)},Z_{\underline t}^{(\delta)}) \mathds{1}_{(\Theta^{\varepsilon_0})^c}(Z_{\underline t}^{(\delta)}) \diff t \Big] \\
            &\quad\quad\quad\quad+ \E\Big[\int_0^T  |Z_{t}^{(\delta)} - Z_{\underline t}^{(\delta)}|^{2p}\cdot \mathds{1}_{S}(Z_{t}^{(\delta)},Z_{\underline t}^{(\delta)}) \mathds{1}_{\Theta^{\varepsilon_0}\setminus \Theta^{\varepsilon_1^\delta}}(Z_{\underline t}^{(\delta)}) \diff t \Big]\\
            &\quad\quad\quad\quad + \E\Big[\int_0^T  |Z_{t}^{(\delta)} - Z_{\underline t}^{(\delta)}|^{2p}\cdot \mathds{1}_{S}(Z_{t}^{(\delta)},Z_{\underline t}^{(\delta)}) \mathds{1}_{\Theta^{\varepsilon_1^\delta}\setminus \Theta^{\varepsilon_2^\delta}}(Z_{\underline t}^{(\delta)}) \diff t \Big]\\
            &\quad\quad\quad\quad+ \E\Big[\int_0^T  |Z_{t}^{(\delta)} - Z_{\underline t}^{(\delta)}|^{2p}\cdot \mathds{1}_{S}(Z_{t}^{(\delta)},Z_{\underline t}^{(\delta)}) \mathds{1}_{\Theta^{\varepsilon_2^\delta}}(Z_{\underline t}^{(\delta)}) \diff t \Big]\Big).
        \end{aligned} 
    \end{equation}
    For the first summand we note that for all $t\in[0,T]$ and all $\delta\in(0,\delta_0]$,
    \begin{equation}
        \{(Z_{t}^{(\delta)}, Z_{\underline t}^{(\delta)})\in S\}\cap \{Z_{\underline t}^{(\delta)} \in (\Theta^{\varepsilon_0})^c\}\subseteq \{|Z_{t}^{(\delta)}- Z_{\underline t}^{(\delta)}|\geq \varepsilon_0\}.
    \end{equation}
    Using this, Hölder's inequality, Markov's inequality, and Lemma \ref{ME} we obtain for all $\delta\in(0,\delta_0]$,
    \begin{equation}
        \begin{aligned}
            &\E\Big[\int_0^T  |Z_{t}^{(\delta)} - Z_{\underline t}^{(\delta)}|^{2p} \mathds{1}_{S}(Z_{t}^{(\delta)},Z_{\underline t}^{(\delta)}) \mathds{1}_{(\Theta^{\varepsilon_0})^c}(Z_{\underline t}^{(\delta)}) \diff t \Big] 
            \leq \int_0^T \E\Big[|Z_{t}^{(\delta)} - Z_{\underline t}^{(\delta)}|^{2p} \mathds{1}_{\{|Z_{t}^{(\delta)}- Z_{\underline t}^{(\delta)}|\geq \varepsilon_0\}}\Big] \diff t \\
            &\leq \int_0^T \E\Big[|Z_{t}^{(\delta)} - Z_{\underline t}^{(\delta)}|^{4p}\Big]^{\frac{1}{2}}\cdot \P( |Z_{t}^{(\delta)}- Z_{\underline t}^{(\delta)}|\geq \varepsilon_0 )^{\frac{1}{2}} \diff t
            \leq \frac{1}{\varepsilon^{2p}_0} \int_0^T \E\Big[|Z_{t}^{(\delta)} - Z_{\underline t}^{(\delta)}|^{4p}\Big] \diff t
            \leq \frac{T c_2}{\varepsilon^{2p}_0}  \delta^{2p}.
        \end{aligned} 
    \end{equation}    
    For the second summand we calculate using Lemma \ref{ME}, Lemma \ref{LemProbEst} \ref{ProbEstiii} with $\alpha=1$ and $q=2p$ that for all $\delta\in(0,\delta_0]$,
    \begin{equation}
        \begin{aligned}
            &\E\Big[\int_0^T  |Z_{t}^{(\delta)} - Z_{\underline t}^{(\delta)}|^{2p}\cdot \mathds{1}_{S}(Z_{t}^{(\delta)},Z_{\underline t}^{(\delta)}) \mathds{1}_{\Theta^{\varepsilon_0}\setminus \Theta^{\varepsilon_1^\delta}}(Z_{\underline t}^{(\delta)}) \diff t \Big]\\
            &\leq \int_0^T \E\Big[  |Z_{t}^{(\delta)} - Z_{\underline t}^{(\delta)}|^{4p}\Big]^\frac{1}{2} \P\Big( (Z_{t}^{(\delta)},Z_{\underline t}^{(\delta)}) \in S, Z_{\underline t}^{(\delta)} \in \Theta^{\varepsilon_0}\setminus \Theta^{\varepsilon_1^\delta}\Big)^{\frac{1}{2}} \diff t\\
            &\leq \int_0^T c_2^{\frac{1}{2}} \delta^p\, \P\Big(  |Z_{t}^{(\delta)}- Z_{\underline t}^{(\delta)}|\geq \varepsilon^\delta_1, Z_{\underline t}^{(\delta)} \in \Theta^{\varepsilon_0}\setminus \Theta^{\varepsilon_1^\delta}\Big)^{\frac{1}{2}} \diff t
            \leq T c_2^{\frac{1}{2}} c_6^{\frac{1}{2}} \delta^{2p}.
        \end{aligned} 
    \end{equation}
    Similar for the third summand we calculate using Lemma \ref{ME}, Lemma \ref{LemProbEst} \ref{ProbEstii} with $\alpha=1$ and $q=2p$ that for all $\delta\in(0,\delta_0]$,
    \begin{equation}
        \begin{aligned}
            &\E\Big[\int_0^T  |Z_{t}^{(\delta)} - Z_{\underline t}^{(\delta)}|^{2p}\cdot \mathds{1}_{S}(Z_{t}^{(\delta)},Z_{\underline t}^{(\delta)}) \mathds{1}_{\Theta^{\varepsilon_1^\delta}\setminus \Theta^{\varepsilon_2^\delta}}(Z_{\underline t}^{(\delta)}) \diff t \Big]\\
            &\leq \int_0^T \E\Big[  |Z_{t}^{(\delta)} - Z_{\underline t}^{(\delta)}|^{4p}\Big]^\frac{1}{2} \P\Big( (Z_{t}^{(\delta)},Z_{\underline t}^{(\delta)}) \in S, Z_{\underline t}^{(\delta)} \in \Theta^{\varepsilon_1^\delta}\setminus \Theta^{\varepsilon_2^\delta}\Big)^{\frac{1}{2}} \diff t\\
            &\leq \int_0^T c_2^{\frac{1}{2}} \delta^p\, \P\Big(  |Z_{t}^{(\delta)}- Z_{\underline t}^{(\delta)}|\geq d(  Z_{\underline t}^{(\delta)}, \Theta), Z_{\underline t}^{(\delta)} \in \Theta^{\varepsilon_1^\delta}\setminus \Theta^{\varepsilon_2^\delta}\Big)^{\frac{1}{2}} \diff t
            \leq T c_2^{\frac{1}{2}} c_6^{\frac{1}{2}} \delta^{2p}.
        \end{aligned} 
    \end{equation}
    For the fourth summand we calculate using Hölder's inequality for all $\delta\in(0,\delta_0]$,
    \begin{equation}\label{ExepSeteq41}
        \begin{aligned}
            &\E\Big[\int_0^T  |Z_{t}^{(\delta)} - Z_{\underline t}^{(\delta)}|^{2p}\cdot \mathds{1}_{S}(Z_{t}^{(\delta)},Z_{\underline t}^{(\delta)}) \mathds{1}_{\Theta^{\varepsilon_2^\delta}}(Z_{\underline t}^{(\delta)}) \diff t \Big]\\
            &\leq \int_0^T \E\Big[  |Z_{t}^{(\delta)} - Z_{\underline t}^{(\delta)}|^{4p} \mathds{1}_{\Theta^{\varepsilon_2^\delta}}(Z_{\underline t}^{(\delta)})\Big]^\frac{1}{2} \P\big(  Z_{\underline t}^{(\delta)} \in \Theta^{\varepsilon_2^\delta}\big)^{\frac{1}{2}} \diff t.
        \end{aligned} 
    \end{equation}
    Further we calculate for all $\delta\in(0,\delta_0]$ and all $t\in[0,T]$,
    \begin{equation}\label{ExepSeteq42}
        \begin{aligned}
            &\E\Big[  |Z_{t}^{(\delta)} - Z_{\underline t}^{(\delta)}|^{4p} \mathds{1}_{\Theta^{\varepsilon_2^\delta}}(Z_{\underline t}^{(\delta)})\Big]\\
            &\leq c\Big(\E\Big[ \big|\widetilde \mu(Z^{(\delta)}_{\underline t}) (t-\underline t) \big|^{4p} \mathds{1}_{\Theta^{\varepsilon_2^\delta}}(Z_{\underline t}^{(\delta)}) \Big]
            + \E\Big[  \big|\widetilde \sigma (Z^{(\delta)}_{\underline t}) (W_t- W_{\underline t})\big|^{4p} \mathds{1}_{\Theta^{\varepsilon_2^\delta}}(Z_{\underline t}^{(\delta)}) \Big]\\
            &\quad\quad\quad + \E\Big[  \big|\widetilde \sigma  (Z^{(\delta)}_{\underline t}) d_{\widetilde \sigma} (Z^{(\delta)}_{\underline t}) \frac{1}{2}((W_t-W_{\underline t})^2 -(t-\underline{t}))\big|^{4p} \mathds{1}_{\Theta^{\varepsilon_2^\delta}}(Z_{\underline t}^{(\delta)}) \Big]\Big).
        \end{aligned} 
    \end{equation}  
    For all $n\in\N$, and all $\omega\in\{ Z_{\underline t}^{(\delta)} \in \Theta^{\varepsilon_2^\delta}\}$ we have $t-\underline t(\omega) \leq \delta^2 \log^4(\delta^{-1})$. Using this and Lemma \ref{ME} we calculate for all $\delta\in(0,\delta_0]$ and all $t\in[0,T]$,
    \begin{equation}\label{ExepSeteq43}
        \begin{aligned}
            &\E\Big[ \big|\widetilde \mu(Z^{(\delta)}_{\underline t}) (t-\underline t) \big|^{4p} \mathds{1}_{\Theta^{\varepsilon_2^\delta}}(Z_{\underline{t}}^{(\delta)}) \Big]
            \leq \E\Big[c_{\widetilde \mu}^{4p} 2^{4p-1} (1+ |Z^{(\delta)}_{\underline t}|^{4p}) \delta^{8p} \log^{16p}(\delta^{-1}) \mathds{1}_{\Theta^{\varepsilon_2^\delta}}(Z_{\underline t}^{(\delta)}) \Big]\\
            &\leq c_{\widetilde \mu}^{4p} 2^{4p-1} \delta^{8p} \log^{16p}(\delta^{-1})  \Big(1+ \E\Big[ \sup_{t\in[0,T]} |Z^{(\delta)}_{t}|^{4p}\Big]\Big)
            \leq c_{\widetilde \mu}^{4p} 2^{4p-1} (1+ c_1) \delta^{8p} \log^{16p}(\delta^{-1}).
        \end{aligned} 
    \end{equation} 
    For the second summand we use Lemma \ref{MWwithInd} to obtain for all $\delta\in(0,\delta_0]$ and all $t\in[0,T]$,
    \begin{equation}\label{ExepSeteq44}
        \begin{aligned}
            &\E\Big[ \big|\widetilde \sigma (Z^{(\delta)}_{\underline t}) (W_t- W_{\underline t})\big|^{4p} \mathds{1}_{\Theta^{\varepsilon_2^\delta}}(Z_{\underline t}^{(\delta)})\Big]
            \leq \E\Big[  c_{\widetilde \sigma}^{4p} 2^{4p-1} (1+ |Z^{(\delta)}_{\underline t}|^{4p})  (W_t- W_{\underline t})^{4p} \mathds{1}_{\Theta^{\varepsilon_2^\delta}}(Z_{\underline t}^{(\delta)}) \Big]\\
            &\leq c_{\widetilde \sigma}^{4p} 2^{4p-1} c_7  \delta^{4p} \log^{8p}(\delta^{-1}).
        \end{aligned} 
    \end{equation}  
    For the third summand we combine the previous arguments to obtain for all $\delta\in(0,\delta_0]$ and all $t\in[0,T]$,
    \begin{equation}\label{ExepSeteq45}
        \begin{aligned}
            &\E\Big[ \big|\widetilde (Z^{(\delta)}_{\underline t}) \sigma d_{\widetilde \sigma} (Z^{(\delta)}_{\underline t}) \frac{1}{2}((W_t-W_{\underline t})^2 -(t-\underline t))\big|^{4p} \mathds{1}_{\Theta^{\varepsilon_2^\delta}}(Z_{\underline t}^{(\delta)}) \Big]\\
            &\leq c_{\widetilde \sigma}^{4p} L_{\widetilde\sigma}^{4p}2^{4p-2} \E\Big[ (1+ |Z^{(\delta)}_{\underline t}|^{4p}) \big(|W_t-W_{\underline t}|^{8p} +|t-\underline t|^{4p}\big)  \mathds{1}_{\Theta^{\varepsilon_2^\delta}}(Z_{\underline t}^{(\delta)}) \Big]\\
            &\leq c_{\widetilde \sigma}^{4p} L_{\widetilde\sigma}^{4p}2^{4p-2} \Big(\E\Big[ \big(1+ \big|Z^{(\delta)}_{\underline{t}}\big|^{4p}\big)   |W_t-W_{\underline{t}}|^{4p} \mathds{1}_{\Theta^{\varepsilon_2^\delta}}(Z_{\underline t}^{(\delta)})\Big] + \delta^{8p}\log^{16p}(\delta^{-1}) \E\big[ 1+ \big|Z^{(\delta)}_{\underline{t}}\big|^{4p}\big] \Big)\\
            &\leq c_{\widetilde \sigma}^{4p} L_{\widetilde\sigma}^{4p}2^{4p-2} \Big( c_7  \delta^{4p} \log^{8p}(\delta^{-1}) + \delta^{8p}\log^{16p}(\delta^{-1}) \Big)(1+ c_1)\\
        \end{aligned} 
    \end{equation}  
    Hence combining \eqref{ExepSeteq42}, \eqref{ExepSeteq43}, \eqref{ExepSeteq44} and \eqref{ExepSeteq45} there exists a constant $\widetilde c_1\in (0,\infty)$ such that for all $\delta\in(0,\delta_0]$ and all $t\in[0,T]$,
    \begin{equation}\label{ExepSeteq46}
        \begin{aligned}
            &\E\Big[  |Z_{t}^{(\delta)} - Z_{\underline t}^{(\delta)}|^{4p} \mathds{1}_{\Theta^{\varepsilon_2^\delta}}(Z_{\underline t}^{(\delta)})\Big]
            \leq \widetilde c_1 \delta^{4p} \log^{8p}(\delta^{-1}).
        \end{aligned} 
    \end{equation}  
    Further we calculate using Jensen's inequality, Lemma \ref{EVofOTF} with $f\equiv 1$ and $\gamma = 1/2$ and Lemma \ref{LemProbEst} \ref{ProbEsti} with $\alpha =1$ and $q =1$ that there exists a constant $\widetilde c_2\in(0,\infty)$ such that for all $\delta\in(0,\delta_0]$,
    \begin{equation}\label{ExepSeteq47}
        \begin{aligned}
            &\int_0^T \P\big(  Z_{\underline t}^{(\delta)} \in \Theta^{\varepsilon_2^\delta}\big)^{\frac{1}{2}} \diff t
            \leq \widetilde c_2 \Big(\int_0^T \P(  Z_{\underline t}^{(\delta)} \in \Theta^{\varepsilon_2^\delta}) \diff t\Big)^{\frac{1}{2}}\\
            &= \widetilde c_2\Big(\int_0^T \P(|Z_{t}^{(\delta)} - Z_{\underline t}^{(\delta)}| <\varepsilon_2^\delta,   Z_{\underline t}^{(\delta)} \in \Theta^{\varepsilon_2^\delta}) \diff t + \int_0^T \P(|Z_{t}^{(\delta)} - Z_{\underline t}^{(\delta)}| \geq \varepsilon_2^\delta,   Z_{\underline t}^{(\delta)} \in \Theta^{\varepsilon_2^\delta}) \diff t \Big)^{\frac{1}{2}}\\
            &\leq \widetilde c_2 \Big(\int_0^T \P(Z_{t}^{(\delta)} \in \Theta^{2\varepsilon_2^\delta}) \diff t + \int_0^T \P(|Z_{t}^{(\delta)} - Z_{\underline t}^{(\delta)}| \geq \varepsilon_2^\delta,   Z_{\underline t}^{(\delta)} \in \Theta^{\varepsilon_2^\delta}) \diff t \Big)^{\frac{1}{2}}\\
            &\leq \widetilde c_2\big( 2c_5 (2 \varepsilon_2^\delta +\delta) + c_6 \delta\big)^{1/2}
            \leq \widetilde c_2 \big((5c_5 +c_6) \delta \log^4(\delta^{-1})\big)^{1/2}.
        \end{aligned} 
    \end{equation}
    Plugging \eqref{ExepSeteq46} and \eqref{ExepSeteq47} in \eqref{ExepSeteq41} we get that there exists $\widetilde c_3, \widetilde c_4 \in(0,\infty)$ such that for all $\delta\in(0,\delta_0]$,
    \begin{equation}\label{ExepSeteq48}
        \begin{aligned}
            &\E\Big[\int_0^T  |Z_{t}^{(\delta)} - Z_{\underline t}^{(\delta)}|^{2p}\cdot \mathds{1}_{S}(Z_{t}^{(\delta)},Z_{\underline t}^{(\delta)}) \mathds{1}_{\Theta^{\varepsilon_2^\delta}}(Z_{\underline t}^{(\delta)}) \diff t \Big]\\
            &\leq \big(\widetilde c_1 \delta^{4p} \log^{8p}(\delta^{-1})\big)^{1/2} \widetilde c_2 \big( (5c_5 +c_6) \delta \log^4(\delta^{-1}) \big)^{1/2}
            \leq \widetilde c_3 \delta^{2p+1/2} \log^{4p+2}(\delta^{-1}) \leq \widetilde c_4 \delta^{2p}.
        \end{aligned} 
    \end{equation}
    This finishes the proof.
\end{proof}

\subsection{Convergence result}

Now we are able to prove the convergence rate for the jump-adapted quasi-Milstein scheme in the next theorem. The proof of the following two results follows the general strategy introduced in \cite[Proof of Theorem 3]{muellergronbach2019b}, which was extended in \cite[Section 5.4]{Yaroslavtseva2022} to adaptive step-size functions and in \cite[Proof of Theorem 3.14]{PSS2024JMS} to jump-adapted schemes. We make use of the general strategy as well as both extensions to prove for SDE \eqref{eq:TSDE} and the approximation scheme \eqref{SumAppr} the same convergence result as in \cite{Yaroslavtseva2022}. 
We prove the following auxiliary lemma first.

\begin{lemma}\label{AuxLemma}
    Let Assumption \ref{assZ} hold. Then for all $p\in[1,\infty)$ there exists a constant $c_{9}\in(0,\infty)$ such that for all $\delta\in (0,\delta_0]$,
    \begin{equation}
        \begin{aligned}
        \E\Big[\sup_{s\in[0,T]}\Big|\int_0^s  \widetilde\sigma(Z^{(\delta)}_{\underline u}) d_{\widetilde\mu}(Z^{(\delta)}_{\underline u}) (W_u-W_{\underline u})  \diff u\Big|^p \Big]^{\frac{1}{p}} \leq c_{9} \delta.
        \end{aligned}
    \end{equation}
\end{lemma}

\begin{proof}
    By Hölder's inequality it is enough to prove the claim for $p\in\N\setminus \{1\}$.
    We define for all $\delta\in[0,\delta_0)$ and $s\in[0,T]$,
    \begin{equation}\label{M13}
        \begin{aligned}
        &U_{\delta,s} = \int_0^s  \widetilde\sigma(Z^{(\delta)}_{\underline u}) d_{\widetilde\mu}(Z^{(\delta)}_{\underline u}) (W_u-W_{\underline u})  \diff u.
        \end{aligned}
    \end{equation}
    It holds for all $\delta\in[0,\delta_0)$, $n\in\N$, and $s\in[\tau_n,\tau_{n+1}]$ that
    \begin{equation}\label{M14}
        \begin{aligned}
        &U_{\delta,s} =  U_{\delta,\tau_n} + \int_{\tau_n}^s \widetilde\sigma(Z^{(\delta)}_{\tau_n}) d_{\widetilde\mu}(Z^{(\delta)}_{\tau_n}) (W_u-W_{\tau_n}) \diff u.
        \end{aligned}
    \end{equation}
    In the following we prove that $(U_{\delta,\tau_n})_{n\in\N}$ is a discrete martingale with respect to the filtration $(\widetilde{\mathbb{F}}_{\tau_n})_{n\in\N}$. 
    By Lemma \ref{MW} and Lemma \ref{ME}, for all $n\in\N$ and all $\delta\in(0,\delta_0]$,
    \begin{equation}\label{M15}
        \begin{aligned}
        &\E[|U_{\delta,\tau_n}|] 
        = \E\Big[\Big| \int_0^{\tau_n}    \widetilde\sigma(Z^{(\delta)}_{\underline u}) d_{\widetilde\mu}(Z^{(\delta)}_{\underline u}) (W_u-W_{\underline u}) \diff u \Big|\Big]
        \leq L_{\widetilde\mu} c_{\widetilde\sigma} \int_0^{T} \E\Big[  (1+ |Z^{(\delta)}_{\underline u}|) \big| W_u-W_{\underline u}\big|\Big] \diff u \\
        &\leq L_{\widetilde\mu} c_{\widetilde\sigma}c_{W_1} \delta^{\frac{1}{2}} \int_0^{T} \Big( 1+ \E\Big[ \sup_{t\in[0,T]}|Z^{(\delta)}_{t}| \Big]\Big)\diff u 
        \leq L_{\widetilde\mu} c_{\widetilde\sigma}c_{W_1}  \delta^{\frac{1}{2}} 2T ( 1+ c_1) <\infty.
        \end{aligned}
    \end{equation}
    Note that $U_{\delta,\tau_n}$ is $\widetilde{\mathbb{F}}_{\tau_n}$-measurable for all $n\in\N$.
    Hence it remains to prove the martingale property. For this recall that by the proof of Lemma \ref{MW} we can express $\tau_{n+1}$ as a measurable function $f$ of $\tau_n$, $Z^{(\delta)}_{\tau_n}$, and $(N_{\tau_n+s} -N_{\tau_n})_{s\geq 0}$ and calculate for all $n\in\N$ and all $\delta\in(0,\delta_0]$,
    \begin{equation}\label{M19}
        \begin{aligned}
        &\E\Big[ U_{\delta,\tau_{n+1}} - U_{\delta,\tau_{n}}  \Big|\widetilde{\mathbb{F}}_{\tau_n} \Big]
        = \E\Big[ \int_{\tau_n}^{\tau_{n+1}} \widetilde\sigma(Z^{(\delta)}_{\tau_n}) d_{\widetilde\mu}(Z^{(\delta)}_{\tau_n}) (W_u-W_{\tau_n}) \diff u \Big|\widetilde{\mathbb{F}}_{\tau_n} \Big] \\
        &= \E\Big[ \int_{0}^{\tau_{n+1}-\tau_n} \widetilde\sigma(Z^{(\delta)}_{\tau_n}) d_{\widetilde\mu}(Z^{(\delta)}_{\tau_n}) (W_{u+\tau_n}-W_{\tau_n}) \diff u \Big|\widetilde{\mathbb{F}}_{\tau_n} \Big] \\
        &= \int_{0}^{\tau_{n+1}-\tau_n} \widetilde\sigma(Z^{(\delta)}_{\tau_n}) d_{\widetilde\mu}(Z^{(\delta)}_{\tau_n})  \E\Big[ (W_{u+\tau_n}-W_{\tau_n})  \Big|\widetilde{\mathbb{F}}_{\tau_n} \Big] \diff u \\
        &= \int_{0}^{\tau_{n+1}-\tau_n} \widetilde\sigma(Z^{(\delta)}_{\tau_n}) d_{\widetilde\mu}(Z^{(\delta)}_{\tau_n})  \cdot 0 \diff u = 0
        \end{aligned}
    \end{equation}
    $\P$-almost surely.
    Observe that for all $\delta\in(0,\delta_0]$,
    \begin{equation}\label{M25}
        \begin{aligned}
        &\sup_{0\leq s\leq T} |U_{\delta,s}| = \max_{n\in\N} \sup_{\tau_n\leq s \leq \tau_{n+1}} |U_{\delta,s}| \\
        &\leq \max_{n\in\N} |U_{\delta,\tau_n}| + \max_{n\in\N} \sup_{\tau_n\leq s \leq \tau_{n+1}} \Big| \int_{\tau_n}^s \widetilde\sigma(Z^{(\delta)}_{\tau_n}) d_{\widetilde\mu}(Z^{(\delta)}_{\tau_n}) (W_u-W_{\tau_n}) \diff u \Big| \\
        &\leq \max_{n\in\N} |U_{\delta,\tau_n}| + c_{\widetilde\sigma} L_{\widetilde\sigma} \max_{n\in\N} \int_{\tau_n}^{\tau_{n+1}} \big( 1 +  |Z^{(\delta)}_{\tau_n}|\big)|W_u-W_{\tau_n}| \diff u.
        \end{aligned}
    \end{equation}
    Next we estimate the expectation of the $p$-th moment separately. For the first summand we obtain by the discrete Burkholder-Davis-Gundy inequality and Jensen's inequality that there exists a constant $\widetilde c_1\in(0,\infty)$ such that for all $\delta\in(0,\delta_0]$,
    \begin{equation}\label{M27}
        \begin{aligned}
        &\E\Big[\max_{n\in\N} |U_{\delta,\tau_n}|^p\Big] 
        \leq \widetilde c_1\, \E\Big[\Big( \sum_{n\in\N} |U_{\delta, \tau_{n+1}}- U_{\delta, \tau_n}|^2\Big)^{\frac{p}{2}}\Big] \\
        &\leq \widetilde c_1\, \E\Big[\Big( \sum_{n\in\N} \delta \int_{\tau_n}^{\tau_{n+1}} | \widetilde\sigma(Z^{(\delta)}_{\tau_n}) d_{\widetilde\mu}(Z^{(\delta)}_{\tau_n})(W_u - W_{\tau_n}) |^2 \diff u\Big)^{\frac{p}{2}}\Big] \\
        &\leq \widetilde c_1\,\delta^{\frac{p}{2}}\, c_{\widetilde\sigma}^{p} L_{\widetilde\mu}^{p} \E\Big[\Big( \sum_{n\in\N}  \int_{0}^{T} \big(1+|Z^{(\delta)}_{\tau_n}|\big)^2 |W_u - W_{\tau_n}|^2 \mathds{1}_{(\tau_n,\tau_{n+1}]} (u) \diff u\Big)^{\frac{p}{2}}\Big] \\
        &\leq \widetilde c_1\,\delta^{\frac{p}{2}}\, c_{\widetilde\sigma}^{p} L_{\widetilde\mu}^{p} T^{\frac{p}{2}-1} \int_{0}^{T} \E\Big[\sum_{n\in\N}   \big(1+|Z^{(\delta)}_{\tau_n}|\big)^p |W_u - W_{\tau_n}|^p \mathds{1}_{(\tau_n,\tau_{n+1}]} (u) \Big] \diff u\\
        &\leq \widetilde c_1\,\delta^{\frac{p}{2}}\, c_{\widetilde\sigma}^{p} L_{\widetilde\mu}^{p} T^{\frac{p}{2}-1} \int_{0}^{T} \sum_{n\in\N} \E\Big[   \big(1+|Z^{(\delta)}_{\tau_n}|\big)^p \E\big[ \sup_{s\in[0,\delta]} |W_{s+\tau_n} - W_{\tau_n}|^p \big| \widetilde{\mathbb{F}}_{\tau_n} \big] \mathds{1}_{(\tau_n,\tau_{n+1}]} (u) \Big] \diff u\\
        &\leq \widetilde c_1\,\delta^{p}\, c_{\widetilde\sigma}^{p} L_{\widetilde\mu}^{p} T^{\frac{p}{2}-1} c_{W_p} 2^{p-1} \int_{0}^{T}  \sum_{n\in\N} \E\Big[   \big(1+\sup_{t\in[0,T]}|Z^{(\delta)}_{t}|^p\big) \mathds{1}_{(\tau_n,\tau_{n+1}]} (u) \Big] \diff u\\
        &\leq \widetilde c_1\,\delta^{p}\, c_{\widetilde\sigma}^{p} L_{\widetilde\mu}^{p} T^{\frac{p}{2}-1} c_{W_p} 2^{p-1} T(1+c_1).
        \end{aligned}
    \end{equation}
    For the expectation of the second summand of \eqref{M25} we calculate with similar steps as in \eqref{M27} that there exists a constant $\widetilde c_2\in(0,\infty)$ such that for all $\delta\in(0,\delta_0]$,
    \begin{equation}\label{M29}
        \begin{aligned}
        & \E\Big[\Big( \max_{n\in\N} \int_{\tau_n}^{\tau_{n+1}} \big( 1 +  |Z^{(\delta)}_{\tau_n}|\big)|W_u-W_{\tau_n}| \diff u \Big)^p\Big]\\
        & \leq \delta^{p-1} \E\Big[ \max_{n\in\N} \int_{\tau_n}^{\tau_{n+1}} \big( 1 +  |Z^{(\delta)}_{\tau_n}|\big)^p|W_u-W_{\tau_n}|^p \diff u \Big]\\
        & \leq \delta^{p-1} \int_{0}^{T}\E\Big[ \max_{n\in\N}  \big( 1 +  |Z^{(\delta)}_{\tau_n}|\big)^p|W_u-W_{\tau_n}|^p \mathds{1}_{(\tau_n,\tau_{n+1}]} (u)  \Big] \diff u\\
        &\leq \delta^{p-1} \int_{0}^{T}\E\Big[ \sum_{n\in\N}  \big( 1 +  |Z^{(\delta)}_{\tau_n}|\big)^p|W_u-W_{\tau_n}|^p \mathds{1}_{(\tau_n,\tau_{n+1}]} (u)  \Big] \diff u\\
        &\leq \delta^{p-1} \delta^{\frac{p}{2}} c_{W_p} 2^{p-1} T(1+c_1)
        \leq \widetilde c_2 \delta^{p}.
        \end{aligned}
    \end{equation}
    Combining \eqref{M25}, \eqref{M27}, and \eqref{M29} we obtain that there exists a constant $ c_{9}\in(0,\infty)$ such that for all $\delta\in(0,\delta_0]$,
    \begin{equation}\label{M30}
        \begin{aligned}
        &  \E\Big[\sup_{0\leq s\leq T} |U_{\delta,s}|^p\Big]\leq  c_{9} \delta^{p},
        \end{aligned}
    \end{equation}
    which proves the claim.
\end{proof}

\begin{theorem}\label{ConvResTSDE}
Let Assumption \ref{assZ} hold. Then for all $p\in[1,\infty)$ there exists a constant $c_{10}\in(0,\infty)$ such that for all $\delta\in (0,\delta_0]$,
\begin{equation}
    \begin{aligned}
    \E\Big[\sup_{t\in[0,T]}\big| Z(t) - Z^{(\delta)}(t)\big|^p \Big]^{\frac{1}{p}} \leq c_{10} \delta.
    \end{aligned}
\end{equation}
\end{theorem}

\begin{proof}
By Hölder's inequality it is enough to prove the claim for $p\in\N\setminus \{1\}$.  
For all $\delta\in(0,\delta_0]$ and all $s\in[0,T]$,
\begin{equation}\label{M1}
    \begin{aligned}
    &\E\Big[\sup_{t\in[0,s]}|Z_t - Z_t^{(\delta)}|^p\Big]\\
    &\leq 3^{p-1} \Big(\E\Big[\sup_{t\in[0,s]} \Big|  \int_0^t \sum_{n=0}^{\infty} \big(\widetilde\mu(Z_u) - \widetilde\mu(Z^{(\delta)}_{\tau_n})\big) \mathds{1}_{(\tau_n,\tau_{n+1}]}(u) \diff u\Big|^p\Big]\\
    &\quad\quad\quad + \E\Big[\sup_{t\in[0,s]} \Big| \int_0^t \sum_{n=0}^{\infty} \Big(\widetilde\sigma(Z_u) -\widetilde\sigma(Z^{(\delta)}_{\tau_n}) - \int_{\tau_n}^u \widetilde\sigma (Z^{(\delta)}_{\tau_n}) d_{\widetilde\sigma} (Z^{(\delta)}_{\tau_n}) \diff W_v \Big)\mathds{1}_{(\tau_n,\tau_{n+1}]}(u) \diff W_u\Big|^p\Big]\\
    &\quad\quad\quad +\E\Big[\sup_{t\in[0,s]} \Big| \int_0^t \big(\widetilde\rho(Z_{u-}) - \widetilde\rho(Z^{(\delta)}_{u-})\big) \diff N_u\Big|^p\Big]\Big).
    \end{aligned}
\end{equation}
We estimate the first summand for all $\delta\in(0,\delta_0]$ and all $s\in[0,T]$ by
\begin{equation}\label{M2a}
    \begin{aligned}
    &\E\Big[\sup_{t\in[0,s]} \Big|  \int_0^t \sum_{n=0}^{\infty} \big(\widetilde\mu(Z_u) - \widetilde\mu(Z^{(\delta)}_{\tau_n})\big) \mathds{1}_{(\tau_n,\tau_{n+1}]}(u) \diff u\Big|^p\Big]\\
    &\leq 2^{p-1}\Big( \E\Big[\sup_{t\in[0,s]} \Big|  \int_0^t  \big(\widetilde\mu(Z_u) - \widetilde\mu(Z^{(\delta)}_{\underline u}) - \widetilde\sigma(Z^{(\delta)}_{\underline u}) d_{\widetilde\mu}(Z^{(\delta)}_{\underline u})(W_u-W_{\underline u})\big)  \diff u\Big|^p\Big] \\
    &\quad\quad\quad +\E\Big[\sup_{t\in[0,s]} \Big|  \int_0^t  \widetilde\sigma(Z^{(\delta)}_{\underline u}) d_{\widetilde\mu}(Z^{(\delta)}_{\underline u}) (W_u-W_{\underline u}) \diff u\Big|^p\Big]\Big) .
    \end{aligned}
\end{equation}
For all $u\in [0,T]$ and all $\delta\in(0,\delta_0]$ we obtain using Remark \ref{ConAss},
\begin{equation}\label{M3}
    \begin{aligned}
    &\big|\widetilde\mu(Z_u) - \widetilde\mu(Z^{(\delta)}_{\underline u}) - \widetilde\sigma(Z^{(\delta)}_{\underline u}) d_{\widetilde\mu}(Z^{(\delta)}_{\underline u})(W_u-W_{\underline u})\big|\\
    &\leq \big|\widetilde\mu ( Z_{u}) - \widetilde\mu (Z^{(\delta)}_{u})\big|
    + \big|\widetilde\mu (Z^{(\delta)}_{u}) - \widetilde\mu (Z^{(\delta)}_{\underline u}) -d_{\widetilde\mu} (Z^{(\delta)}_{\underline u})(Z^{(\delta)}_{u} - Z^{(\delta)}_{\underline u})\big|\mathds{1}_{S^c}(Z^{(\delta)}_{u},Z^{(\delta)}_{\underline u}) \\
    &\quad + \big|\widetilde\mu (Z^{(\delta)}_{u}) - \widetilde\mu (Z^{(\delta)}_{\underline u}) -d_{\widetilde\mu} (Z^{(\delta)}_{\underline u})(Z^{(\delta)}_{u} - Z^{(\delta)}_{\underline u})\big|\mathds{1}_{S}(Z^{(\delta)}_{u},Z^{(\delta)}_{\underline u})\\
    &\quad + \Big|d_{\widetilde\mu} (Z^{(\delta)}_{\underline u})\Big( \widetilde\mu(Z^{(\delta)}_{\underline u})(u-\underline u) +\frac{1}{2} \widetilde\sigma (Z^{(\delta)}_{\underline u})d_{\widetilde\sigma} (Z^{(\delta)}_{\underline u})\big((W_u-W_{\underline u})^2-(u-\underline u)\big)\Big)\Big|\\
    &\leq L_{\widetilde\mu} |Z_{u} - Z^{(\delta)}_{u}|
    + b_{\widetilde\mu } |Z^{(\delta)}_{u} - Z^{(\delta)}_{\underline u} |^2
    + 2L_{\widetilde\mu} |Z^{(\delta)}_{u} - Z^{(\delta)}_{\underline u}|\mathds{1}_{S}(Z^{(\delta)}_{u},Z^{(\delta)}_{\underline u})\\
    &\quad + L_{\widetilde\mu} \Big(c_{\widetilde\mu} + \frac{1}{2}L_{\widetilde\sigma} c_{\widetilde\sigma}\Big)(1+|Z^{(\delta)}_{\underline u}|) |u-\underline u| +\frac{1}{2}L_{\widetilde\mu}L_{\widetilde\sigma} c_{\widetilde\sigma}(1+|Z^{(\delta)}_{\underline u}|) |W_u-W_{\underline u}|^2.
    \end{aligned}
\end{equation}
Using this and Jensen's inequality, we get for the first summand of \eqref{M2a} that there exists a constant $\widetilde c_1\in(0,\infty)$ such that for all $\delta\in(0,\delta_0]$ and all $s\in[0,T]$,
\begin{equation}\label{M4}
    \begin{aligned}
    &\E\Big[\sup_{t\in[0,s]} \Big|  \int_0^t \sum_{n=0}^{\infty} \big(\widetilde\mu(Z_u) - \widetilde\mu(Z^{(\delta)}_{\tau_n}) - \widetilde\sigma(Z^{(\delta)}_{\tau_n}) d_{\widetilde\mu}(Z^{(\delta)}_{\tau_n})(W_u-W_{\tau_n})\big) \mathds{1}_{(\tau_n,\tau_{n+1}]}(u) \diff u\Big|^p\Big] \\ 
    &\leq \widetilde c_1\, \Big(\E\Big[  \int_0^s  |Z_{u} - Z^{(\delta)}_{u}|^p  \diff u\Big]
    + \E\Big[  \int_0^s  |Z^{(\delta)}_{u} - Z^{(\delta)}_{\underline u} |^{2p} \diff u\Big] \\
    &\quad\quad\quad +\E\Big[ \Big(  \int_0^s   |Z^{(\delta)}_{u} - Z^{(\delta)}_{\underline u}|\mathds{1}_{S}(Z^{(\delta)}_{u},Z^{(\delta)}_{\underline u}) \diff u\Big)^p\Big]
    + \E\Big[ \int_0^s  (1+|Z^{(\delta)}_{\underline u}|)^p |u-\underline u|^p \diff u \Big]\\
    &\quad\quad\quad+ \E\Big[  \int_0^s  (1+|Z^{(\delta)}_{\underline u}|)^p |W_u-W_{\underline u}|^{2p}  \diff u\Big] \Big).
    \end{aligned}
\end{equation}
For the first summands of \eqref{M4} we get for all $\delta\in(0,\delta_0]$ and all $s\in[0,T]$,
\begin{equation}\label{M5}
    \begin{aligned}
    &\E\Big[ \int_0^s  \big|Z_{u} - Z^{(\delta)}_{u}\big|^p \diff u \Big] 
    \leq \int_0^s \E\Big[ \sup_{v\in[0,u]} \big|Z_{v} - Z^{(\delta)}_{v}\big|^p\Big]  \diff u.
    \end{aligned}
\end{equation}
For the second summand of \eqref{M4} we use Lemma \ref{ME} to obtain for all $\delta\in(0,\delta_0]$ and all $s\in[0,T]$,
\begin{equation}\label{M6}
    \begin{aligned}
    &\E\Big[ \int_0^s  \big|Z^{(\delta)}_{u} - Z^{(\delta)}_{\underline u} \big|^{2p} \diff u\Big]
    \leq T c_2 \delta^{p}.
    \end{aligned}
\end{equation}
For the third summand of \eqref{M4} we use Lemma \ref{LemExepSet} to obtain for all $\delta\in(0,\delta_0]$ and all $s\in[0,T]$,
\begin{equation}\label{M7}
    \begin{aligned}
    &\E\Big[\Big( \int_0^s \big|Z^{(\delta)}_{u} - Z^{(\delta)}_{\underline u}\big|\mathds{1}_{S}(Z^{(\delta)}_{u},Z^{(\delta)}_{\underline u}) \diff u\Big)^{p}\Big] \leq c_{8}\, \delta^{p}.
    \end{aligned}
\end{equation}
For the fourth summand of \eqref{M4} we get by Lemma \ref{ME} for all $\delta\in(0,\delta_0]$ and all $s\in[0,T]$, 
\begin{equation}\label{M8}
    \begin{aligned}
    &\E\Big[ \int_0^s (1+|Z^{(\delta)}_{\underline u}|)^p |u-\underline u|^p \diff u\Big]
    \leq 2^{p-1} \delta^p  \int_0^s\Big(1 +\E\Big[ \sup_{v\in[0,T]} |Z^{(\delta)}_{v}|^p \Big]\Big)\diff u 
    \leq 2^{p-1} T (1 +c_1) \delta^{p}.
    \end{aligned}
\end{equation}
For the fifth summand of \eqref{M4} we use Lemma \ref{MW} and Lemma \ref{ME} to obtain for all $\delta\in(0,\delta_0]$ and all $s\in[0,T]$,
\begin{equation}\label{M9}
    \begin{aligned}
    &\E\Big[ \int_0^s (1+|Z^{(\delta)}_{\tau_n}|)^p |W_u-W_{\underline u}|^{2p} \diff u\Big]
    \leq 2^{p-1} c_{W_{2p}}\, \delta^p \int_0^s \Big(1+ \E\big[\sup_{v\in[0,T]}  |Z^{(\delta)}_{v}|^p\big]\Big)  \diff u\\
    &\leq 2^{p-1} c_{W_{2p}}\, T (1+ c_1 )\delta^{p} .
    \end{aligned}
\end{equation}
Plugging \eqref{M5}, \eqref{M6}, \eqref{M7}, \eqref{M8}, and \eqref{M9} into \eqref{M4} we obtain that there exists $\widetilde c_2\in(0,\infty)$ such that for all $\delta\in(0,\delta_0]$ and all $s\in[0,T]$,
\begin{equation}\label{M10}
    \begin{aligned}
    &\E\Big[\sup_{t\in[0,s]} \Big|  \int_0^t \big(\widetilde\mu(Z_u) - \widetilde\mu(Z^{(\delta)}_{\underline u}) - \widetilde\sigma (Z^{(\delta)}_{\underline u})d_{\widetilde\mu}(Z^{(\delta)}_{\underline u})(W_u-W_{\underline u})\big) \diff u\Big|^p\Big] \\
    &\leq \widetilde c_2\, \int_0^s \E\Big[ \sup_{v\in[0,u]} \big|Z_{v} - Z^{(\delta)}_{v}\big|^p\Big]  \diff u 
    + \widetilde c_2 \delta^{p}.
    \end{aligned}
\end{equation}
For the second summand of \eqref{M2a} we apply Lemma \ref{AuxLemma}. Hence we obtain that there exists a constant $\widetilde c_3 \in(0,\infty)$ such that for all $\delta\in(0,\delta_0]$ and all $s\in[0,T]$,
\begin{equation}\label{M2}
    \begin{aligned}
    &\E\Big[\sup_{t\in[0,s]} \Big|  \int_0^t \sum_{n=0}^{\infty} \big(\widetilde\mu(Z_u) - \widetilde\mu(Z^{(\delta)}_{\tau_n})\big) \mathds{1}_{(\tau_n,\tau_{n+1}]}(u) \diff u\Big|^p\Big]\\
    &\leq \widetilde c_3\, \int_0^s \E\Big[ \sup_{v\in[0,u]} \big|Z_{v} - Z^{(\delta)}_{v}\big|^p\Big]  \diff u 
    + \widetilde c_3 \delta^{p}.
    \end{aligned}
\end{equation}
For estimating the second summand of \eqref{M1}, observe that for all $u\in(\tau_n,\tau_{n+1}]$ and all $\delta\in(0,\delta_0]$ by Remark \ref{ConAss},
\begin{equation}\label{M52}
    \begin{aligned}
    &\Big|\widetilde\sigma(Z_u) -\widetilde\sigma(Z^{(\delta)}_{\tau_n}) - \int_{\tau_n}^u \widetilde\sigma(Z^{(\delta)}_{\tau_n}) d_{\widetilde\sigma} (Z^{(\delta)}_{\tau_n}) \diff W_v\Big|\\
    &\leq \big|\widetilde\sigma ( Z_{u}) - \widetilde\sigma (Z^{(\delta)}_{u})\big|
    + \big|\widetilde\sigma (Z^{(\delta)}_{u}) - \widetilde\sigma (Z^{(\delta)}_{\underline u}) -d_{\widetilde\sigma} (Z^{(\delta)}_{\underline u})(Z^{(\delta)}_{u} - Z^{(\delta)}_{\underline u})\big|\mathds{1}_{S^c}(Z^{(\delta)}_{u},Z^{(\delta)}_{\underline u}) \\
    &\quad + \big|\widetilde\sigma (Z^{(\delta)}_{u}) - \widetilde\sigma (Z^{(\delta)}_{\underline u}) -d_{\widetilde\sigma} (Z^{(\delta)}_{\underline u})(Z^{(\delta)}_{u} - Z^{(\delta)}_{\underline u})\big|\mathds{1}_{S}(Z^{(\delta)}_{u},Z^{(\delta)}_{\underline u})\\
    &\quad + \Big|d_{\widetilde\sigma} (Z^{(\delta)}_{\underline u})\Big( \widetilde\mu(Z^{(\delta)}_{\underline u})(u-\underline u) +\frac{1}{2} \widetilde\sigma (Z^{(\delta)}_{\underline u})d_{\widetilde\sigma} (Z^{(\delta)}_{\underline u})\big((W_u-W_{\underline u})^2-(u-\underline u)\big)\Big)\Big|\\
    &\leq L_{\widetilde\sigma} |Z_{u} - Z^{(\delta)}_{u}|
    + b_{\widetilde\sigma } |Z^{(\delta)}_{u} - Z^{(\delta)}_{\underline u} |^2 
    + 2L_{\widetilde\sigma} |Z^{(\delta)}_{u} - Z^{(\delta)}_{\underline u}|\mathds{1}_{S}(Z^{(\delta)}_{u},Z^{(\delta)}_{\underline u})\\
    &\quad + L_{\widetilde\sigma} \Big(c_{\widetilde\mu} + \frac{1}{2}L_{\widetilde\sigma}^2 c_{\widetilde\sigma}\Big)(1+|Z^{(\delta)}_{\underline u}|) |u-\underline u| +\frac{1}{2}\|d_{\widetilde\sigma}\|^2_\infty c_{\widetilde\sigma}(1+|Z^{(\delta)}_{\underline u}|) |W_u-W_{\underline u}|^2.
    \end{aligned}
\end{equation}
Using this, the Burkholder-Davis-Gundy inequality, and Jensen's inequality we obtain for the second summand of \eqref{M1} that there exist constants $\widetilde c_4, \widetilde c_{5} \in (0,\infty)$ such that for all $\delta\in(0,\delta_0]$ and all $s\in[0,T]$,
\begin{equation}\label{M53}
    \begin{aligned}
    &\E\Big[\sup_{t\in[0,s]} \Big| \int_0^t \sum_{n=0}^{\infty} \Big(\widetilde\sigma(Z_u) -\widetilde\sigma(Z^{(\delta)}_{\tau_n}) - \int_{\tau_n}^u \widetilde\sigma(Z^{(\delta)}_{\tau_n}) d_{\widetilde\sigma} (Z^{(\delta)}_{\tau_n}) \diff W_v \Big)\mathds{1}_{(\tau_n,\tau_{n+1}]}(u) \diff W_u\Big|^p\Big]\\
    &\leq \widetilde c_4 \, \E\Big[\Big( \int_0^s \sum_{n=0}^{\infty} \Big|\widetilde\sigma(Z_u) -\widetilde\sigma(Z^{(\delta)}_{\tau_n}) - \int_{\tau_n}^u \widetilde\sigma(Z^{(\delta)}_{\tau_n}) d_{\widetilde\sigma} (Z^{(\delta)}_{\tau_n}) \diff W_v \Big|^2\mathds{1}_{(\tau_n,\tau_{n+1}]}(u) \diff u\Big)^{\frac{p}{2}}\Big]\\
    &\leq \widetilde c_{5}\, \Big( \E\Big[ \int_0^s  \big|Z_{u} - Z^{(\delta)}_{u}\big|^p \diff u \Big]
    + \E\Big[ \int_0^s  \big|Z^{(\delta)}_{u} - Z^{(\delta)}_{\underline u} \big|^{2p} \diff u\Big] \\
    &\quad\quad\quad +\E\Big[\Big( \int_0^s \big|Z^{(\delta)}_{u} - Z^{(\delta)}_{\underline u}\big|^2\mathds{1}_{S}(Z^{(\delta)}_{u},Z^{(\delta)}_{\underline u}) \diff u\Big)^{\frac{p}{2}}\Big]
    + \E\Big[ \int_0^s (1+|Z^{(\delta)}_{\underline u}|)^p |u-\underline u|^p \diff u\Big]\\
    &\quad\quad\quad +\E\Big[ \int_0^s  (1+|Z^{(\delta)}_{\underline u}|)^p |W_u-W_{\underline u}|^{2p}\Big) \diff u\Big]\Big).
    \end{aligned}
\end{equation}
Combining \eqref{M53}, \eqref{M5}, \eqref{M6}, \eqref{M7}, \eqref{M8}, and \eqref{M9} we obtain that there exists a constant $\widetilde c_{6}\in(0,\infty)$ such that for all $\delta\in(0,\delta_0]$ and all $s\in[0,T]$,
\begin{equation}\label{M59}
    \begin{aligned}
    &\E\Big[\sup_{t\in[0,s]} \Big| \int_0^t \sum_{n=0}^{\infty} \Big(\widetilde\sigma(Z_u) -\widetilde\sigma(Z^{(\delta)}_{\tau_n}) - \int_{\tau_n}^u \widetilde\sigma d_{\widetilde\sigma}(Z^{(\delta)}_{\tau_n}) (Z^{(\delta)}_{\tau_n}) \diff W_v \Big)\mathds{1}_{(\tau_n,\tau_{n+1}]}(u) \diff W_u\Big|^p\Big]\\
    &\leq \widetilde c_{6}\, \int_0^s \E\Big[ \sup_{v\in[0,u]} \big|Z_{v} - Z^{(\delta)}_{v}\big|^p\Big]  \diff u 
    + \widetilde c_{6}\,  \delta^{p}.\\
    \end{aligned}
\end{equation}
For the third summand of \eqref{M1}, Lemma \ref{BDGaKunita} ensures for all $\delta\in(0,\delta_0]$ and all $s\in[0,T]$,
\begin{equation}\label{M60}
    \begin{aligned}
    &\E\Big[\sup_{t\in[0,s]} \Big| \int_0^t \big(\widetilde\rho(Z_{u-}) -\widetilde\rho(Z^{(\delta)}_{u-}) \big) \diff N_u\Big|^p\Big]
    \leq \hat c\, L_{\widetilde\rho}^p  \int_0^s \E\Big[ \sup_{v\in[0,u]} \big| Z_{v} - Z^{(\delta)}_{v} \big|^p\Big] \diff u.
    \end{aligned}
\end{equation}
Plugging \eqref{M2}, \eqref{M59}, and \eqref{M60} into  \eqref{M1} we obtain that there exists a constant $\widetilde c_{7}\in(0, \infty)$ such that for all $\delta\in(0,\delta_0]$ and all $s\in[0,T]$,
\begin{equation}\label{M61}
    \begin{aligned}
    &\E\Big[\sup_{t\in[0,s]}|Z_t - Z_t^{(\delta)}|^p\Big]
    \leq \widetilde c_{7}\, \Big(\int_0^s \E\Big[ \sup_{v\in[0,u]} \big| Z_{v} - Z^{(\delta)}_{v} \big|^p\Big] \diff u + \delta^{p}\Big). 
    \end{aligned}
\end{equation}
Note that $\E\Big[\sup_{v\in[0,T]} \big| Z_{v} - Z^{(\delta)}_{v} \big|^p\Big]<\infty$ and that $s\mapsto \E\Big[\sup_{v\in[0,s]} \big| Z_{v} - Z^{(\delta)}_{v} \big|^p\Big]$ is a Borel measurable mapping, because it is monotonically increasing. Hence we apply Gronwall's inequality and obtain that there exists a constant $c_{10}\in(0,\infty)$ such that for all $\delta\in(0,\delta_0]$,
\begin{equation}\label{M62}
    \begin{aligned}
    &\E\Big[\sup_{t\in[0,T]}|Z_t - Z_t^{(\delta)}|^p\Big]
    \leq c_{10}  \delta^{p}.
    \end{aligned}
\end{equation}
This proves the claim.
\end{proof}

\subsection{Cost Analysis}

In this section we estimate the computational cost of our numerical algorithm $Z^{(\delta)}$ on the interval $[0,T]$. For this we study the number of arithmetic operations, function evaluations, and simulations of random variables. Note that this is proportional to the number of steps that are needed to reach time $T$. Hence we estimate the value
\begin{equation}
    n(Z^{(\delta)}_T) = \min\{k\in\N : \tau_k = T\}
\end{equation}
to determine the dependence of the computational cost on $\delta$. The same quantity has, for example, been studied in \cite{NSS19, Yaroslavtseva2022}. The result and proof of the following theorem are closely based on \cite[Section 5.5]{Yaroslavtseva2022}.

\begin{theorem}\label{thm:cost}
    Let Assumption \ref{assZ} hold. Then there exists a constant $c_{11}\in(0,\infty)$ such that for all $\delta\in[0,\delta_0)$
    \begin{equation}
        \E[n(Z^{(\delta)}_T)]\leq c_{11}(\delta^{-1}+\E[N_T]).
    \end{equation}
\end{theorem}

\begin{proof}
    Note that for all $\delta\in[0,\delta_0)$ and all $k\in\N$
    \begin{equation}
        \int_{\tau_k}^{\tau_{k+1}} \frac{1}{\tau_{k+1}-\tau_k + \mathds{1}_{\{\tau_{k+1}=\tau_k\}}} \mathds{1}_{\{\tau_{k+1}\neq\tau_k\}} \diff t = \mathds{1}_{\{\tau_{k+1}\neq\tau_k\}}.
    \end{equation}
    Hence for all $\delta\in[0,\delta_0)$,
    \begin{equation}
        \begin{aligned}
            &n(Z^{(\delta)}_T) = \sum_{k\in\N} \int_{\tau_k}^{\tau_{k+1}} \frac{1}{\tau_{k+1}-\tau_k+ \mathds{1}_{\{\tau_{k+1}=\tau_k\}}} \mathds{1}_{\{\tau_{k+1}\neq\tau_k\}} \diff t\\
            &= \sum_{k\in\N} \int_{\tau_k}^{\tau_{k+1}} \frac{1}{\tau_{k+1}-\tau_k+ \mathds{1}_{\{\tau_{k+1}=\tau_k\}}} \mathds{1}_{\{\tau_{k+1}\neq\tau_k\}} \mathds{1}_{\{N_{\tau_{k+1}}-N_{\tau_k} = 1\}}\diff t\\
            &\quad + \sum_{k\in\N} \int_{\tau_k}^{\tau_{k+1}} \frac{1}{\tau_{k+1}-\tau_k+ \mathds{1}_{\{\tau_{k+1}=\tau_k\}}} \mathds{1}_{\{\tau_{k+1}\neq\tau_k\}} \mathds{1}_{\{N_{\tau_{k+1}}-N_{\tau_k} = 0\}}\diff t\\
            &= \sum_{k\in\N} \mathds{1}_{\{N_{\tau_{k+1}}-N_{\tau_k} = 1\}} \int_{\tau_k}^{\tau_{k+1}} \frac{1}{\tau_{k+1}-\tau_k} \diff t\\
            &\quad + \sum_{k\in\N} \int_{\tau_k}^{\tau_{k+1}} \frac{1}{h^\delta(Z^{(\delta)}_{\tau_k})} \mathds{1}_{\{\tau_{k+1}\neq\tau_k\}} \mathds{1}_{\{N_{\tau_{k+1}}-N_{\tau_k} = 0\}}\diff t.
        \end{aligned}
    \end{equation}
    This implies for all $\delta\in[0,\delta_0)$,
    \begin{equation}\label{eq:Cost1}
        \begin{aligned}
            &n(Z^{(\delta)}_T) \leq  \sum_{k\in\N} \mathds{1}_{\{N_{\tau_{k+1}}-N_{\tau_k} = 1\}} 
            + \sum_{k\in\N} \int_{\tau_k}^{\tau_{k+1}} \frac{1}{h^\delta(Z^{(\delta)}_{\tau_k})} \mathds{1}_{\{\tau_{k+1}\neq\tau_k\}} \diff t\\
            &\leq N_T + \int_0^T  \frac{1}{h^\delta(Z^{(\delta)}_{\underline t})}  \diff t\\
            &= N_T + \int_0^T  \frac{1}{h^\delta(Z^{(\delta)}_{\underline t})}  \mathds{1}_{(\Theta^{\varepsilon_1^\delta})^c}(Z^{(\delta)}_{\underline t}) \diff t  \\
            &\quad + \int_0^T  \frac{1}{h^\delta(Z^{(\delta)}_{\underline t})} \mathds{1}_{\Theta^{\varepsilon_1^\delta}\setminus\Theta^{\varepsilon_2^\delta} }(Z^{(\delta)}_{\underline t}) \diff t  
            + \int_0^T  \frac{1}{h^\delta(Z^{(\delta)}_{\underline t})}  \mathds{1}_{\Theta^{\varepsilon_2^\delta} }(Z^{(\delta)}_{\underline t}) \diff t.
        \end{aligned}
    \end{equation}
    Now we estimate the expectation of each summand separately. We calculate for all $\delta\in[0,\delta_0)$,
    \begin{equation}\label{eq:Cost2}
        \begin{aligned}
            &\E\Big[\int_0^T  \frac{1}{h^\delta(Z^{(\delta)}_{\underline t})}  \mathds{1}_{(\Theta^{\varepsilon_1^\delta})^c}(Z^{(\delta)}_{\underline t}) \diff t \Big] 
            = \delta^{-1} \int_0^T \P(Z^{(\delta)}_{\underline t} \in (\Theta^{\varepsilon_1^\delta})^c)\diff t \leq T \delta^{-1}. \\
        \end{aligned}
    \end{equation}    
    Further, using similar calculations to \eqref{ExepSeteq47} we get that there exists a constant $\widetilde c_1\in(0,\infty)$ such that for all $\delta\in[0,\delta_0)$,
    \begin{equation}\label{eq:Cost3}
        \begin{aligned}
            &\E\Big[\int_0^T  \frac{1}{h^\delta(Z^{(\delta)}_{\underline t})}  \mathds{1}_{\Theta^{\varepsilon_2^\delta} }(Z^{(\delta)}_{\underline t}) \diff t \Big]
            = \delta^{-2} \log^{-4}(\delta^{-1})\int_0^T \P(Z^{(\delta)}_{\underline t} \in \Theta^{\varepsilon_2^\delta})\diff t \leq \widetilde c_1 \delta^{-1}. \\
        \end{aligned}
    \end{equation} 
    It remains to prove that there exists a constant $\widetilde c_2\in(0,\infty)$ such that for all $\delta\in[0,\delta_0)$, 
    \begin{equation}\label{eq:Cost4}
        \begin{aligned}
            &\E\Big[ \int_0^T  \frac{1}{h^\delta(Z^{(\delta)}_{\underline t})} \mathds{1}_{\Theta^{\varepsilon_1^\delta}\setminus\Theta^{\varepsilon_2^\delta} }(Z^{(\delta)}_{\underline t}) \diff t\Big] 
            \leq \widetilde  c_2 \delta^{-1},
        \end{aligned}
    \end{equation}  
    then \eqref{eq:Cost1}, \eqref{eq:Cost2}, \eqref{eq:Cost3}, and \eqref{eq:Cost4} imply the claim.
    To prove \eqref{eq:Cost4} we define for all $\delta\in[0,\delta_0)$ and all $t\in[0,T]$ the set 
    \begin{equation}\label{eq:Cost5}
        D_t^{\delta} = \Big\{| Z^{(\delta)}_t- Z^{(\delta)}_{\underline t}| \leq \frac{1}{2} d(Z^{(\delta)}_{\underline t},\Theta) \Big\}
    \end{equation}
    and use that for all $\delta\in(0,\delta_0]$,
    \begin{equation}\label{eq:Cost6}
        \begin{aligned}
            &\E\Big[ \int_0^T  \frac{1}{h^\delta(Z^{(\delta)}_{\underline t})} \mathds{1}_{\Theta^{\varepsilon_1^\delta}\setminus\Theta^{\varepsilon_2^\delta} }(Z^{(\delta)}_{\underline t}) \diff t\Big] \\
            &= \E\Big[ \int_0^T  \frac{1}{h^\delta(Z^{(\delta)}_{\underline t})} \mathds{1}_{\Theta^{\varepsilon_1^\delta}\setminus\Theta^{\varepsilon_2^\delta} }(Z^{(\delta)}_{\underline t})\mathds{1}_{D_t^{\delta}} \diff t\Big]
            + \E\Big[ \int_0^T  \frac{1}{h^\delta(Z^{(\delta)}_{\underline t})} \mathds{1}_{\Theta^{\varepsilon_1^\delta}\setminus\Theta^{\varepsilon_2^\delta} }(Z^{(\delta)}_{\underline t}) \mathds{1}_{(D_t^{\delta})^c}\diff t\Big].
        \end{aligned}
    \end{equation}  
    Since the distance function $d(\cdot, \Theta)\colon\R\to [0, \infty)$ is Lipschitz continuous with Lipschitz constant 1, we get for all $\delta\in(0,\delta_0]$ and all $t\in[0,T]$,
    \begin{equation}\label{eq:Cost7}
        \begin{aligned}
            &\{Z^{(\delta)}_{\underline t}\in \Theta^{\varepsilon_1^\delta}\setminus\Theta^{\varepsilon_2^\delta}\} \cap D_t^\delta \subset \{ Z^{(\delta)}_{t} \in \Theta^{\frac{3}{2}\varepsilon_1^\delta}\setminus\Theta^{\frac{1}{2}\varepsilon_2^\delta}\} \cap \Big\{ \frac{1}{2}d(Z^{(\delta)}_{\underline t}, \Theta)\leq d(Z^{(\delta)}_{t}, \Theta)\leq \frac{3}{2}d(Z^{(\delta)}_{\underline t}, \Theta) \Big\}.
        \end{aligned}
    \end{equation}     
    Hence, for all $\delta\in(0,\delta_0]$,
    \begin{equation}\label{eq:Cost8}
        \begin{aligned}
            &\E\Big[ \int_0^T  \frac{1}{h^\delta(Z^{(\delta)}_{\underline t})} \mathds{1}_{\Theta^{\varepsilon_1^\delta}\setminus\Theta^{\varepsilon_2^\delta} }(Z^{(\delta)}_{\underline t})\mathds{1}_{D_t^{\delta}} \diff t\Big]\\
            &= \log^4(\delta^{-1}) \E\Big[\int_0^T \frac{1}{d(Z^{(\delta)}_{\underline t}, \Theta)^2}\mathds{1}_{\Theta^{\varepsilon_1^\delta}\setminus\Theta^{\varepsilon_2^\delta} }(Z^{(\delta)}_{\underline t})\mathds{1}_{D_t^{\delta}} \diff t\Big]\\
             &\leq \frac{9}{4}  \log^4(\delta^{-1}) \E\Big[\int_0^T\frac{1}{d(Z^{(\delta)}_{t}, \Theta)^2}\mathds{1}_{ \Theta^{\frac{3}{2}\varepsilon_1^{\delta}}\setminus \Theta^{\frac{1}{2}\varepsilon_2^{\delta}}}(Z^{(\delta)}_{t})\,dt\Bigr].
        \end{aligned}
    \end{equation}  
    Now we define $\bar \varepsilon^\delta = \delta^{3/4}\log^{3}(\delta^{-1})$ and note that $\varepsilon_2^{\delta}\leq \overline\varepsilon^\delta\leq \varepsilon_1^{\delta}$ for all $\delta\in(0, \delta_0]$.
    This implies for all $\delta\in(0,\delta_0]$,
    \begin{equation}\label{eq:Cost9}
        \begin{aligned}
            &\E\Big[ \int_0^T  \frac{1}{h^\delta(Z^{(\delta)}_{\underline t})} \mathds{1}_{\Theta^{\varepsilon_1^\delta}\setminus\Theta^{\varepsilon_2^\delta} }(Z^{(\delta)}_{\underline t})\mathds{1}_{D_t^{\delta}} \diff t\Big]\\
             &\leq \frac{9}{4}  \log^4(\delta^{-1}) \Big(\E\Big[\int_0^T\frac{1}{d(Z^{(\delta)}_{t}, \Theta)^2}\mathds{1}_{ \Theta^{\frac{3}{2}\varepsilon_1^{\delta}}\setminus \Theta^{\bar \varepsilon^{\delta}}}(Z^{(\delta)}_{t})\,dt\Bigr]
             + \E\Big[\int_0^T\frac{1}{d(Z^{(\delta)}_{t}, \Theta)^2}\mathds{1}_{ \Theta^{\bar \varepsilon^{\delta}} \setminus \Theta^{\frac{1}{2}\varepsilon_2^{\delta}}}(Z^{(\delta)}_{t})\,dt\Bigr]\Big)\\
             &\leq \frac{9}{4}  \log^4(\delta^{-1}) \Big( \E\Big[\int_0^T\frac{1}{(\max\{ \bar \varepsilon^{\delta},d(Z^{(\delta)}_{t}, \Theta)\})^2}\mathds{1}_{ \Theta^{\frac{3}{2}\varepsilon_1^{\delta}}}(Z^{(\delta)}_{t})\,dt\Bigr]\\
             &\quad\quad\quad\quad \quad\quad\quad +  \E\Big[\int_0^T\frac{1}{(\max\{ \frac{1}{2}\varepsilon_2^{\delta} ,d(Z^{(\delta)}_{t}, \Theta)\})^2} \mathds{1}_{ \Theta^{\bar \varepsilon^{\delta}}}(Z^{(\delta)}_{t})\,dt\Bigr]\Big).
        \end{aligned}
    \end{equation}  
    Applying Lemma \ref{EVofOTF} choosing $f=(\max\{\overline\varepsilon^\delta,\cdot\})^{-2}$ and $\gamma=1/2$ resp.~$f=(\max\{\tfrac{1}{2}\varepsilon_2^{\delta},\cdot\})^{-2}$ and $\gamma=1/6$ we obtain that there exist constants $\widetilde c_3, \widetilde c_4 \in (0,\infty)$ such that for all $\delta\in(0,\delta_0]$,
    \begin{equation}\label{eq:Cost10}
        \begin{aligned}
            &\E\Big[ \int_0^T  \frac{1}{h^\delta(Z^{(\delta)}_{\underline t})} \mathds{1}_{\Theta^{\varepsilon_1^\delta}\setminus\Theta^{\varepsilon_2^\delta} }(Z^{(\delta)}_{\underline t})\mathds{1}_{D_t^{\delta}} \diff t\Big]\\
            &\leq \frac{9}{4} \log^4(\delta^{-1}) \Big( c_5 \Big(\int_0^{\frac{3}{2}\varepsilon_1^{\delta}} (\max\{\overline\varepsilon^\delta,x\})^{-2} \diff x + \sup_{x\in[0,\frac{3}{2}\varepsilon_1^{\delta} ]} (\max\{\overline\varepsilon^\delta,x\})^{-2} \Big(\Big(\frac{3}{2}\varepsilon_1^{\delta}\Big)^{\frac{3}{2}-\frac{1}{2}} + \delta^{\frac{3}{2}- \frac{1}{2}}\Big) \Big)\\
            &\quad\quad\quad\quad \quad\quad\quad +   c_5 \Big(\int_0^{\overline\varepsilon^\delta} (\max\{\tfrac{1}{2}\varepsilon_2^{\delta},x\})^{-2} \diff x + \sup_{x\in[0,\overline\varepsilon^\delta ]} (\max\{\tfrac{1}{2}\varepsilon_2^{\delta},x\})^{-2} \Big((\overline\varepsilon^\delta)^{\frac{3}{2}-\frac{1}{6}} + \delta^{\frac{3}{2}- \frac{1}{6}}\Big) \Big) \Big)\\
            &\leq \widetilde c_3 \log^4(\delta^{-1}) \Big(\varepsilon_1^{\delta}(\overline\varepsilon^\delta)^{-2} + (\overline\varepsilon^\delta)^{-2}(  \varepsilon_1^{\delta} +\delta) + \big( 2(\tfrac{1}{2} \varepsilon_2^{\delta})^{-1} - (\overline\varepsilon^\delta)^{-1} \big)  + (\varepsilon_2^{\delta})^{-2}( (\overline\varepsilon^\delta)^\frac{4}{3} + \delta^\frac{4}{3})\Big)\\
            &= \widetilde c_3 \big( \delta^{-1} + \delta^{-1} + \delta^{-\frac{1}{2}}\log^{-2}(\delta^{-1}) + 4 \delta^{-1}\log^{-4}(\delta^{-1}) - \delta^{-\frac{3}{4}} \log^{-3}(\delta^{-1}) + \delta^{-1} + \delta^{-\frac{2}{3}}\log^{-8}(\delta^{-1})\big)\\ 
            &\leq \widetilde c_4 \delta^{-1}.
        \end{aligned}
    \end{equation} 
    Next we use that $\delta^2\log^{4}(\delta^{-1}) \leq h^{\delta}$ and Lemma \ref{LemProbEst} \ref{ProbEstii} with $\alpha=1/2$ and $q=2$ to obtain that there exists a constant $\widetilde c_5 \in(0,\infty)$ such that for all $\delta\in(0,\delta_0]$,
    \begin{equation}\label{eq:Cost11}
        \begin{aligned}
            &\E\Big[ \int_0^T  \frac{1}{h^\delta(Z^{(\delta)}_{\underline t})} \mathds{1}_{\Theta^{\varepsilon_1^\delta}\setminus\Theta^{\varepsilon_2^\delta} }(Z^{(\delta)}_{\underline t}) \mathds{1}_{(D_t^{\delta})^c}\diff t\Big]\\
            &\leq \delta^{-2} \log^{-4}(\delta^{-1}) \int_0^T \P\Big( | Z^{(\delta)}_t- Z^{(\delta)}_{\underline t}| \geq \frac{1}{2} d(Z^{(\delta)}_{\underline t},\Theta), Z^{(\delta)}_{\underline t} \in\Theta^{\varepsilon_1^\delta}\setminus\Theta^{\varepsilon_2^\delta} \Big) \diff t\\
            &\leq c_6 \log^{-4}(\delta^{-1}) \leq \widetilde c_5 \delta^{-1}.
        \end{aligned}
    \end{equation} 
    Plugging \eqref{eq:Cost10} and \eqref{eq:Cost11} in \eqref{eq:Cost6} proves \eqref{eq:Cost4}.
\end{proof}

\section{Convergence of the transformation-based doubly-adaptive quasi-Milstein scheme}\label{ConvDAqMS}

In this section we are going to use the results developed in Section \ref{SecQMS} for the SDE \eqref{eq:SDE} satisfying Assumption \eqref{assX}. For this we use the transformation originally introduced in \cite{LS17b}, which is commonly used to work with SDEs with piecewise Lipschitz continuous drift coefficient and transforms the drift coefficient locally around the potential discontinuities, in a slightly modified way as in \cite{muellergronbach2019b}. We shortly recall the definition of the transformation and its properties for the convenience of the reader. 

The transformation function $G\colon\R\to\R$ is for all $x\in\R$ defined by 
\begin{align}\label{eq:G-1d}
G(x)=x+ \sum_{i=1}^k \alpha_i\, \phi\!\left(\frac{x-\zeta_i}{\nu}\right)(x-\zeta_i)|x-\zeta_i|, 
\end{align}
where 
\begin{align}\label{eq:bump}
\phi(u)=
\begin{cases}
(1-u^2)^4 & \text{if } |u|\leq 1,\\
0 & \text{otherwise},
\end{cases}
\end{align}
\begin{align*}
\nu\in\Big(0,\min\bigg\{\min\limits_{1\le i\le m}\frac{1}{8|\alpha_i|},\min\limits_{1\le i\le m-1}\frac{\zeta_{i+1}-\zeta_i}{2}\bigg\} \Big),
\end{align*}
using the conventions that $\min \emptyset =\infty$ and $\frac{1}{0} = \infty$ and 
\begin{equation}\label{alpha}
\begin{aligned}
\alpha_i = \frac{\mu(\zeta_i -)- \mu(\zeta_i +)}{2\sigma^2(\zeta_i)}
\end{aligned}
\end{equation}
for all $i\in\{1,\ldots,m\}$.

\begin{remark}
    By \cite[Lemma 1]{muellergronbach2019b} $G$ is Lipschitz continuous, differentiable, and invertible. Its inverse $G^{-1}$ is Lipschitz continuous. Its derivative $G'$ is bounded, positive, bounded away from 0 and twice continuously differentiable in $(\zeta_{i-1},\zeta_i)$ for all $i\in\{1,\dots,m+1\}$. Further for the second derivative $G''$ for all $i\in\{1,\dots,m\}$ the one-sided limits  $G''(\zeta_i-) $ and $G''(\zeta_i+)$ exist and satisfy $ G''(\zeta_i-) = -2\alpha_i$ resp.~$G''(\zeta_i+) = 2\alpha_i$.
\end{remark}

Hence the function $G''\colon\bigcup_{i=1}^m (\zeta_{i-1},\zeta_i) \to \R$ is well defined and can be extended to $G''\colon\R\to\R$ by
\begin{equation}\label{exG2}
G''(\zeta_i) = 2\alpha_i + 2 \frac{\mu(\zeta_i+)-\mu(\zeta_i)}{\sigma^2(\zeta_i)}, \quad i\in\{1,\ldots,m\}.
\end{equation}
for all $i\in\{1,\ldots ,m\}$.

Next define the process $Z\colon[0,T]\times\Omega\to \R$ by $Z = G(X)$. Then by \cite[Theorem 3.1]{PS20} $Z$ satisfies 
\begin{align}
\diff Z_t &=\widetilde \mu(Z_t)  \diff t + \widetilde \sigma(Z_t) \diff W_t+\widetilde \rho(Z_{t-})\diff N_t , \quad t\in[0,T], \quad Z_0=\widetilde \xi,
\end{align}
where 
\begin{equation}\label{eqTCoeff}
\begin{aligned}
&\widetilde \mu = (G'\cdot \mu +\frac{1}{2} G''\cdot\sigma^2) \circ G^{-1}, \quad \widetilde \sigma = (G'\cdot \sigma) \circ G^{-1}, \quad \widetilde \rho = G(G^{-1} +\rho(G^{-1})) - \operatorname{id}.
\end{aligned}
\end{equation}

\begin{remark}
    If $\mu$, $\sigma$, and $\rho$ satisfy Assumption \ref{assX}, then $\widetilde \mu$, $\widetilde\sigma$, and $\widetilde\rho$ from \eqref{eqTCoeff} satisfy Assumption \ref{assZ} and the process $Z=G(X)$ is the unique solution of SDE \eqref{eq:TSDE} by \cite[Lemma 4.2 and Lemma 4.3]{PSS2024JMS}.
\end{remark}

With this knowledge we are able to prove convergence rate $1$ of the transformation-based doubly adaptive quasi-Milstein scheme similar to \cite[Theorem 4.4]{PSS2024JMS}.

\begin{theorem}\label{MainWeakAss}
Let Assumption \ref{assX} hold. Let $p\in[1,\infty)$, $G$ be the transformation defined in \eqref{eq:G-1d}, let $Z\colon[0,T]\times\Omega\to \R$ be defined by $Z = G(X)$, and let for all $\delta\in(0,\delta_0]$, $Z^{(\delta)}$ be the doubly-adaptive quasi-Milstein scheme for $Z$ as introduced in \eqref{SumAppr}.
Then there exist constants $c_{12}, c_{13} \in(0,\infty)$ such that for all $\delta\in(0,\delta_0]$,
\begin{equation}
    \begin{aligned}
    \E\Big[ \sup_{t\in[0,T]} |X_t- G^{-1}(Z^{(\delta)}_t) |^p\Big]^\frac{1}{p} \leq c_{12} \delta;
    \end{aligned}
\end{equation}
\begin{equation}
    \begin{aligned}
        \E\Big[n\big(G^{-1}(Z^{(\delta)}_T)\big)\Big]\leq c_{13}(\delta^{-1}+\E[N_T]).
    \end{aligned}
\end{equation}
\end{theorem}

\begin{proof}
As $G$ is invertible, its inverse is Lipschitz continuous and $Z$ satisfies Assumption \ref{assZ} we obtain using Theorem \ref{ConvResTSDE} that there exists a constant $c_{12}\in(0,\infty)$ such that
\begin{equation}
    \begin{aligned}
    &\E\Big[ \sup_{t\in[0,T]} |X_t- G^{-1}(Z^{(\delta)}_t) |^p\Big]^\frac{1}{p}
    = \E\Big[ \sup_{t\in[0,T]} |G^{-1}(Z_t) - G^{-1}(Z^{(\delta)}_t) |^p\Big]^\frac{1}{p} \\
    &\leq L_{G^{-1}} \E\Big[ \sup_{t\in[0,T]} |Z_t - Z^{(\delta)}_t|^p\Big]^\frac{1}{p} 
    \leq L_{G^{-1}} c_{10} \delta 
    \leq c_{12}\, \delta^1.
    \end{aligned}
\end{equation}
For the estimation of the cost we note that $\E[n(G^{-1}(Z^{(\delta)}_T))] = \E[n(Z^{(\delta)}_T)]$ and apply Theorem \ref{thm:cost}.
\end{proof}

\begin{remark}
    The cost $n(G^{-1}(Z^{(\delta)}_T))$ of the transformation-based doubly adaptive quasi-Milstein scheme 
    is the same as for the doubly adaptive quasi-Milstein scheme, 
    but due to the additional applications of the transformation function $G$, its derivatives and its inverse the computational effort in every timestep increases.  Moreover, the inverse of the transformation can in general be not calculated explicitly. In this case additional numerical approximations of the inverse are needed in every step. In this sense the scheme is an implicit but still implementable scheme. 
\end{remark}

\section*{Acknowledgements}
The author would like to thank Michaela Szölgyenyi for her support and the stimulating discussions on the topic of this article.

\appendix

\section{Additional proofs}\label{appendix}

\begin{proof}[Proof of Lemma \ref{MW}]
By \eqref{EqTauGen} we get
\begin{equation}
    \begin{aligned}\label{MW2}
        \tau_{n+1} = & \big(\tau_n+ h^\delta(Z^{(\delta)}_{\tau_n})\big)\wedge \big(\tau_n + \inf\{s\geq 0 : N_{\tau_n+s}-N_{\tau_n}\geq 0\} \big) \wedge T,
    \end{aligned}
\end{equation}
which implies that $\tau_{n+1}$ can be expressed as a measurable function $f$ of $\tau_n$, $Z^{(\delta)}_{\tau_n}$ and $(N_{\tau_n+s} -N_{\tau_n})_{s\geq 0}$. Using this we obtain for all $n\in\N$, all $k\in\N$, all $\delta\in(0,\delta_0]$ and all $t\in[0,T]$,
\begin{equation}
    \begin{aligned}\label{MW3}
    &\E\big[\big(1+ \big|Z^{(\delta)}_{\tau_n}\big|^p\big)   |W_t-W_{\tau_n}|^q \mathds{1}_{[\tau_n,\tau_{n+1})}(t)  \mathds{1}_{\{N_t = k\}} \big]\\
    &= \E\big[\big(1+ \big|Z^{(\delta)}_{\tau_n}\big|^p\big)   |W_u-W_{\tau_n}|^q \mathds{1}_{[\tau_n,\tau_{n+1})}(t)  \mathds{1}_{\{N_{\tau_n} = k\}} \big]\\
    &\leq \E\Big[\big(1+ \big|Z^{(\delta)}_{\tau_n}\big|^p\big)  \sup_{s\in[0,\delta]} |W_{s+\tau_n}-W_{\tau_n}|^q \mathds{1}_{\{\tau_n \leq t\}} \mathds{1}_{\{f(\tau_n,Z^{(\delta)}_{\tau_n}, (N_{s+\tau_n}-N_{\tau_n})_{s\geq 0})> t\}} \mathds{1}_{\{N_{\tau_n} = k\}} \Big].
    \end{aligned}
\end{equation}
Applying \cite[p.~33]{Grikhman2004} we get for all $n\in\N$, all $k\in\N$, all $\delta\in(0,\delta_0]$ and all $t\in[0,T]$,
\begin{equation}
    \begin{aligned}\label{MW4}
    &\E\big[\big(1+ \big|Z^{(\delta)}_{\tau_n}\big|^p\big)   |W_t-W_{\tau_n}|^q \mathds{1}_{[\tau_n,\tau_{n+1})}(t)  \mathds{1}_{\{N_t = k\}} \big]\\
    &\leq \E\Big[\E\Big[\big(1+ \big|Z^{(\delta)}_{\tau_n}\big|^p\big)  \sup_{s\in[0,\delta]} |W_{s+\tau_n}-W_{\tau_n}|^q \mathds{1}_{\{\tau_n \leq t\}} \mathds{1}_{\{f(\tau_n, Z^{(\delta)}_{\tau_n}, (N_{s+\tau_n}-N_{\tau_n})_{s\geq 0})> t\}} \mathds{1}_{\{N_{\tau_n} = k\}} \Big| \F_{\tau_n} \Big] \Big] \\
    &= \E\Big[\big(1+ \big|Z^{(\delta)}_{\tau_n}\big|^p\big) \mathds{1}_{\{\tau_n \leq t\}} \mathds{1}_{\{N_{\tau_n} = k\}} \E\Big[ \sup_{s\in[0,\delta]} |W_{s+\tau_n}-W_{\tau_n}|^q  \mathds{1}_{\{f(\tau_n, Z^{(\delta)}_{\tau_n}, (N_{s+\tau_n}-N_{\tau_n})_{s\geq 0})> t\}}  \Big| \F_{\tau_n} \Big] \Big] \\
    &= \E\Big[\big(1+ \big|Z^{(\delta)}_{\tau_n}\big|^p\big) \mathds{1}_{\{\tau_n \leq t\}} \mathds{1}_{\{N_{\tau_n} = k\}} \\
    &\quad\quad\quad \cdot\E\Big[ \sup_{s\in[0,\delta]} |W_{s+\tau_n}-W_{\tau_n}|^q  \mathds{1}_{\{f(z,\widetilde z, (N_{s+\tau_n}-N_{\tau_n})_{s\geq 0})> t\}}  \Big| \F_{\tau_n} \Big]\Big|_{z =\tau_n, \widetilde z =Z^{(\delta)}_{\tau_n} } \Big].
    \end{aligned}
\end{equation}
By Lemma \ref{PropDiscGrid}, $(W_{\tau_n+s}- W_{\tau_n})_{s\geq 0}$ and $(N_{\tau_n+s}- N_{\tau_n})_{s\geq 0}$ are independent of $\F_{\tau_n}$ as well as $(W_{\tau_n+s}- W_{\tau_n})_{s\geq 0}$ is independent of $(N_{\tau_n+s}- N_{\tau_n})_{s\geq 0}$. This, Lemma \ref{MEW}, and \cite[p.~33]{Grikhman2004} imply for all $n\in\N$, all $k\in\N$, all $\delta\in(0,\delta_0]$ and all $t\in[0,T]$,
\begin{equation}
    \begin{aligned}\label{MW4a}
    &\E\big[\big(1+ \big|Z^{(\delta)}_{\tau_n}\big|^p\big)   |W_t-W_{\tau_n}|^q \mathds{1}_{[\tau_n,\tau_{n+1})}(t)  \mathds{1}_{\{N_{\tau_n} = k\}} \big]\\
    &\leq c_{W_q} \delta^\frac{q}{2} \E\Big[\big(1+ \big|Z^{(\delta)}_{\tau_n}\big|^p\big) \mathds{1}_{\{\tau_n \leq t\}} \mathds{1}_{\{N_{\tau_n} = k\}} \E\Big[ \mathds{1}_{\{f(z,\widetilde z, (N_{s+\tau_n}-N_{\tau_n})_{s\geq0})> t\}} \Big]\Big|_{z =\tau_n, \widetilde z =Z^{(\delta)}_{\tau_n}} \Big]\\
    &= c_{W_q} \delta^\frac{q}{2} \E\Big[\big(1+ \big|Z^{(\delta)}_{\tau_n}\big|^p\big) \mathds{1}_{\{\tau_n \leq t\}}  \mathds{1}_{\{f(\tau_n,Z^{(\delta)}_{\tau_n}, (N_{s+\tau_n}-N_{\tau_n})_{s\geq0})> t\}} \mathds{1}_{\{N_{\tau_n} = k\}} \Big]\\
    &= c_{W_q} \delta^\frac{q}{2} \E\Big[\big(1+ \big|Z^{(\delta)}_{\tau_n}\big|^p\big) \mathds{1}_{[\tau_n, \tau_{n+1})}(t)  \mathds{1}_{\{N_{\tau_n} = k\}} \Big].
    \end{aligned}
\end{equation}
Further, for all $\delta\in(0,\delta_0]$ and all $t\in[0,T]$,
\begin{equation}
    \begin{aligned}\label{MW5}
    &\E\Big[ \big(1+ \big|Z^{(\delta)}_{\underline{t}}\big|^p\big)   |W_t-W_{\underline{t}}|^q \Big]
    =\E\Big[\sum_{n=0}^\infty \big(1+ \big|Z^{(\delta)}_{\tau_n}\big|^p\big)   |W_t-W_{\tau_n}|^q \mathds{1}_{[\tau_n,\tau_{n+1})}(t) \Big]\\
    &= \sum_{k=0}^\infty \sum_{n=0}^{\infty} \E\Big[ \big(1+ \big|Z^{(\delta)}_{\tau_n}\big|^p\big)   |W_t-W_{\tau_n}|^q \mathds{1}_{[\tau_n,\tau_{n+1})}(t) \mathds{1}_{\{N_{\tau_n}=k\}} \Big]\\
    &\leq \sum_{k=0}^\infty \sum_{n=0}^{\infty} c_{W_q} \delta^\frac{q}{2} \E\Big[\big(1+ \big|Z^{(\delta)}_{\tau_n}\big|^p\big) \mathds{1}_{[\tau_n,\tau_{n+1})}(t)  \mathds{1}_{\{N_{\tau_n} = k\}} \Big]\\
    &= c_{W_q} \delta^\frac{q}{2} \E\Big[\sum_{n=0}^{\infty} \big(1+ \big|Z^{(\delta)}_{\tau_n}\big|^p\big) \mathds{1}_{[\tau_n,\tau_{n+1})}(t) \Big]
    = c_{W_q} \delta^\frac{q}{2} \E\Big[1+ \big|Z^{(\delta)}_{\underline t}\big|^p \Big].
    \end{aligned}
\end{equation}
\end{proof}

\begin{proof}[Proof of Lemma \ref{FiniteMom}]
By Hölder's inequality it is enough to prove the claim for $p\in[2,\infty) \cap \N$. 
For the first summand it is enough to consider instead for all $\delta\in(0,\delta_0]$,
\begin{equation}
    \begin{aligned}\label{FM3}
    &\int_0^T \E \big[\big|Z^{(\delta)}_{\underline t}\big|^p\big] \diff t 
    = \int_0^T \E \Big[\sum_{n=0}^\infty \big|Z^{(\delta)}_{\tau_n}  \big|^p \mathds{1}_{[\tau_n,\tau_{n+1})}(t)\Big] \diff t \\
    &\leq T \E\Big[ \sum_{n=0}^{N_T + M}  \big|Z^{(\delta)}_{\tau_{n}}\big|^p\Big] 
    = T \sum_{k=0}^{\infty} \sum_{n=0}^{k+M} \E\Big[  \big|Z^{(\delta)}_{\tau_{n}}\big|^p \mathds{1}_{\{N_T =k\}}\Big].
    \end{aligned}
\end{equation}
By \eqref{EqDefMS} and the fact that $N_{\tau_n}- N_{\tau_{n-1}} \in\{0,1\}$ we get that for all $n$, $k\in\N$ there exists a constant $\widetilde c_1 \in (0,\infty)$ such that  for all $\delta\in(0,\delta_0]$,
\begin{equation}
    \begin{aligned}\label{FM4}
    &\E\Big[\big|Z^{(\delta)}_{\tau_{n}}\big|^p\mathds{1}_{\{N_T = k\}}\Big]
    \leq 2^{p-1}\Big( \E\Big[\big|Z^{(\delta)}_{\tau_{n}-}\big|^p\mathds{1}_{\{N_T = k\}}\Big]  + \E\Big[\big|\widetilde \rho(Z^{(\delta)}_{\tau_{n}-})\big|^p \big|N_{\tau_n} - N_{\tau_{n-1}}\big|^p\mathds{1}_{\{N_T = k\}}\Big] \Big)\\
    &\leq \widetilde c_1\, \P( N_T =k) + \widetilde c_1\, \E\Big[\big|Z^{(\delta)}_{\tau_{n}-}\big|^p\mathds{1}_{\{N_T = k\}}\Big].
    \end{aligned}
\end{equation}
Further by Lemma \ref{PropDiscGrid} and Lemma \ref{MEW} we obtain for all $q\in\N$, all $n\in\N$, all $k\in\N$ and all $\delta \in (0,\delta_0]$,
\begin{equation}
    \begin{aligned}\label{FM6}
    &\E\Big[ \big( 1+ \big|Z^{(\delta)}_{\tau_{n-1}}\big|^p\big) \big|W_{\tau_n} - W_{\tau_{n-1}}\big|^q \mathds{1}_{\{N_T =k\}}\Big]\\ 
    &\leq \E\Big[ \big( 1+ \big|Z^{(\delta)}_{\tau_{n-1}}\big|^p\big) \mathds{1}_{\{N_T =k\}} \E\Big[   \sup_{s\in[0,\delta]} \big|W_{s+\tau_{n-1}} - W_{\tau_{n-1}}\big|^q \Big| \widetilde \F_{\tau_{n-1}}\Big] \Big]\\
    &\leq c_{W_q} \delta^\frac{q}{2} \Big(\P(N_T = k) +  \E\Big[ \big|Z^{(\delta)}_{\tau_{n-1}}\big|^p  \mathds{1}_{\{N_T = k \}}  \Big]\Big).
    \end{aligned}
\end{equation}
By \eqref{EqDefMSLL} and \eqref{FM6} there exists constants $\widetilde c_2, \widetilde c_3, \widetilde c_4 \in(0,\infty)$ such that for all $n\in\N$, all $k\in\N$ and all $\delta \in (0,\delta_0]$,
\begin{equation}
    \begin{aligned}\label{FM5}
    &\E\Big[\big|Z^{(\delta)}_{\tau_{n}-}\big|^p\mathds{1}_{\{N_T = k\}} \Big]\\
    &\leq \widetilde c_2 \Big(  \E\Big[\big|Z^{(\delta)}_{\tau_{n-1}}\big|^p\mathds{1}_{\{N_T = k\}} \Big] + \E\Big[\big|\widetilde\mu(Z^{(\delta)}_{\tau_{n-1}}) (\tau_n- \tau_{n-1})\big|^p\mathds{1}_{\{N_T =k\}} \Big] \\
    &\quad\quad\quad +\E\Big[\big|\widetilde\sigma(Z^{(\delta)}_{\tau_{n-1}}) (W_{\tau_n} - W_{\tau_{n-1}})\big|^p\mathds{1}_{\{N_T = k\}} \Big]\\ 
    &\quad\quad\quad +\E\Big[\Big| \frac{1}{2} \widetilde\sigma(Z^{(\delta)}_{\tau_{n-1}}) d_{\widetilde\sigma} (Z^{(\delta)}_{\tau_{n-1}}) \big((W_{\tau_n} - W_{\tau_{n-1}})^2 - (\tau_n- \tau_{n-1})\big)\Big|^p\mathds{1}_{\{N_T = k\}} \Big]\Big)\\
    &\leq \widetilde c_3 \Big(  \E\Big[\big|Z^{(\delta)}_{\tau_{n-1}}\big|^p\mathds{1}_{\{N_T =k\}} \Big] + \E\Big[\big( 1+ \big|Z^{(\delta)}_{\tau_{n-1}}\big|^p\big) \big|\tau_n- \tau_{n-1}\big|^p\mathds{1}_{\{N_T =k\}}\Big] \\
    &\quad\quad\quad + \E\Big[ \big( 1+ \big|Z^{(\delta)}_{\tau_{n-1}}\big|^p\big)
    \big|W_{\tau_n} - W_{\tau_{n-1}}\big|^p\mathds{1}_{\{N_T =k\}} \Big]\\ 
    &\quad\quad\quad + \E\Big[ \big( 1+ \big|Z^{(\delta)}_{\tau_{n-1}}\big|^p\big) |W_{\tau_n} - W_{\tau_{n-1}}|^{2p} \mathds{1}_{\{N_T =k\}}\Big]\Big)\\
    &\leq \widetilde c_4 \Big( \P(N_T =k) + \E\Big[\big|Z^{(\delta)}_{\tau_{n-1}}\big|^p\mathds{1}_{\{N_T =k\}}\Big]\Big).
    \end{aligned}
\end{equation}
Together with \eqref{FM4} we obtain that there exist constants $\widetilde c_5,\widetilde c_6 \in(1,\infty)$ such that for all $n\in\N$, all $k\in\N$ and all $\delta \in (0,\delta_0]$,
\begin{align}
    &\E\Big[\big|Z^{(\delta)}_{\tau_{n}}\big|^p\mathds{1}_{\{N_T =k\}}\Big]
    \leq \widetilde c_5\Big(\P(N_T =k)+ \E\Big[\big|Z^{(\delta)}_{\tau_{n-1}}\big|^p\mathds{1}_{\{N_T =k\}}\Big]\Big); \label{FM7}\\
    &\E\Big[\big|Z^{(\delta)}_{\tau_{n}-}\big|^p\mathds{1}_{\{N_T =k\}}\Big]
    \leq \widetilde c_6\Big( \P(N_T =k)+  \E\Big[\big|Z^{(\delta)}_{\tau_{n-1}-}\big|^p\mathds{1}_{\{N_T =k\}}\Big]\Big).\label{FM8}
\end{align}
Using \eqref{FM7} recursively we get that there exist constants $\widetilde c_7, \widetilde c_8 \in(0,\infty)$ such that for all $n\in\N$, all $k\in\N$ and all $\delta \in (0,\delta_0]$,
\begin{equation}
    \begin{aligned}\label{FM9}
    &\E\Big[\big|Z^{(\delta)}_{\tau_{n}}\big|^p\mathds{1}_{\{N_T =k\}}\Big]
    \leq \sum_{j=0}^{n-1} \widetilde c_5^{\,j+1}\, \P(N_T =k) + \widetilde c_5^{\,n}\, \E\Big[ |\widetilde \xi| \mathds{1}_{\{N_T =k\}} \Big]\\
    &\leq \widetilde c_7 (\widetilde c_5^{\,n+1}-1) \, \P(N_T =k) + |\xi| \widetilde c_5^{n+1} \, \P(N_T =k) \leq \widetilde c_8 \,\widetilde c_5^{\,n+1} \, \P(N_T =k).
    \end{aligned}
\end{equation}
Analogously we obtain using recusively \eqref{FM8}, and \eqref{FM6} that there exists a constant $\widetilde c_9\in(0,\infty)$ such that for all $n\in\N$, all $k\in\N$ and all $\delta \in (0,\delta_0]$,
\begin{equation}
    \begin{aligned}\label{FM9a}
    &\E\Big[\big|Z^{(\delta)}_{\tau_{n}-}\big|^p\mathds{1}_{\{N_T =k\}}\Big]
    \leq \sum_{j=0}^{n-2} \widetilde c_6^{\,j+1}\, \P( N_T =k) + \widetilde c_6^{\,n-1}\, \E\Big[\big|Z^{(\delta)}_{\tau_{1}-}\big|^p\mathds{1}_{\{N_T =k\}}\Big]\\
    &\leq \sum_{j=0}^{n-2} \widetilde c_6^{\,j+1}\, \P(N_T =k) + \widetilde c_6^{\,n-1} \widetilde c_4\Big(\P(N_T = k) + \E\Big[\big|\widetilde\xi \big|^p\mathds{1}_{\{N_T =k\}}\Big]\Big)
    \leq \widetilde c_9 \,\widetilde c_6^{\,n} \, \P(N_T =k).
    \end{aligned}
\end{equation}
Combining \eqref{FM3} and \eqref{FM9} and using that $\widetilde c_5 \in(1,\infty)$ we obtain for all $n\in\N$, all $k\in\N$ and all $\delta \in (0,\delta_0]$,
\begin{equation}
    \begin{aligned}\label{FM10}
    & \int_0^T \E \big[\big|Z^{(\delta)}_{\underline t}\big|^p\big] \diff t 
    \leq T \, \widetilde c_8\sum_{k=0}^{\infty} \sum_{n=0}^{k+M}  \widetilde c_5^{\,n+1} \P(N_T = k) 
    = T \, \widetilde c_8  \sum_{k=0}^{\infty} \P(N_T = k)\, \widetilde c_5 \, \frac{1- \widetilde c_5^{\,k+M+1}}{1- \widetilde c_5}\\
    &\leq \frac{T \, \widetilde c_8\,  \widetilde c_5^{\,M+2}}{ \widetilde c_5 -1 }   \sum_{k=0}^{\infty} \P(N_T = k)\, \widetilde c_5^{\,k}
    = \frac{T \, \widetilde c_8\,  \widetilde c_5^{\,M+2}}{ \widetilde c_5 -1 } \exp(\lambda T(\widetilde c_5-1)).
    \end{aligned}
\end{equation}
For the second summand we use that $N_T$ has finite $p$-th moments and calculate similar to \eqref{FM10} using the fact that $\widetilde c_6\in(1,\infty)$, we obtain for all $n\in\N$, all $k\in\N$ and all $\delta \in (0,\delta_0]$,
\begin{equation}
    \begin{aligned}\label{FM12}
    &\E\Big[ N_T^{p-1} \sum_{n=0}^{N_T + M} \big|Z^{(\delta)}_{\tau_{n}-}\big|^p \Big] 
    = \sum_{k=0}^\infty\sum_{n=0}^{k + M} k^{p-1} \E\Big[ \big|Z^{(\delta)}_{\tau_{n}-}\big|^p \mathds{1}_{\{N_T = k\}} \Big] 
    \leq \sum_{k=0}^\infty\sum_{n=0}^{k + M} k^{p-1} \widetilde c_9 \, \widetilde c_6^{\,n} \, \P(N_T =k) \\
    &= \widetilde c_9 \sum_{k=0}^\infty k^{p-1}  \P(N_T =k) \frac{1- \widetilde c_6^{\,k+M+1}}{1-\widetilde c_6}
    \leq \frac{\widetilde c_9\, \widetilde c_6^{\,M+1}}{\widetilde c_6-1}  \sum_{k=0}^\infty k^{p-1}  \P(N_T =k) \widetilde c_6^{\,k}\\
    &= \frac{\widetilde c_9\, \widetilde c_6^{\,M+1}}{\widetilde c_6-1}  \sum_{k=0}^\infty k^{p-1} \frac{(\lambda T \widetilde c_6)^k}{k!} e^{-\lambda T \widetilde c_6} e^{\lambda T(\widetilde c_6 -1)}
    = \frac{\widetilde c_{10} \widetilde c_9 \widetilde c_6^{\,M+1}}{\widetilde c_6-1} \exp(\lambda T(\widetilde c_6-1)),
    \end{aligned}
\end{equation}
where $\widetilde c_{10}$ is the $(p-1)$-th moment of a Poisson distributed random variable with parameter $\widetilde c_6 \lambda T$.
\end{proof}

\begin{proof}[Proof of Lemma \ref{ME}]
By Hölder's inequality it is enough to prove the claim for $p\in[2,\infty)\cap\N$.
Recalling equation \eqref{SumAppr} we obtain for all $s\in[0,T]$ and all $\delta\in(0,\delta_0]$,
\begin{equation}
    \begin{aligned}\label{ME1}
    &\E\Big[\sup_{t\in[0,s]} |Z^{(\delta)}_{t}|^p\Big]  \\
    & \leq 4^{p-1} \Big(  |\widetilde\xi|^p
    + \E\Big[\sup_{t\in[0,s]}\Big|\int_0^t \sum_{n=0}^{\infty}  \widetilde \mu\big(Z^{(\delta)}_{\tau_{n}}\big)\mathds{1}_{(\tau_n,\tau_{n+1}]}(u) \diff u\Big|^p\Big] \\
    &\quad\quad\quad\quad+ \E\Big[\sup_{t\in[0,s]}\Big|\int_0^t \sum_{n=0}^{\infty} \Big(\widetilde \sigma\big(Z^{(\delta)}_{\tau_{n}}\big) + \widetilde\sigma \big(Z^{(\delta)}_{\tau_n}\big) d_{\widetilde\sigma} \big(Z^{(\delta)}_{\tau_n}\big) (W_{u}-W_{\tau_n}) \Big)\mathds{1}_{(\tau_n,\tau_{n+1}]}(u) \diff W_u\Big|^p\Big] \\
    &\quad\quad\quad\quad+ \E\Big[\sup_{t\in[0,s]}\Big|\int_0^t\rho\big(Z^{(\delta)}_{u-}\big) \diff N_u\Big|^p\Big] \Big).
\end{aligned}
\end{equation}
Now we estimate each summand separately. For the second one we apply Jensen's inequality, and Lemma \ref{FiniteMom} to obtain for all $s\in[0,T]$ and all $\delta\in(0,\delta_0]$,
\begin{equation}
    \begin{aligned}\label{ME2}
    &\E\Big[\sup_{t\in[0,s]}\Big|\int_0^t \sum_{n=0}^{\infty} \widetilde \mu\big(Z^{(\delta)}_{\tau_{n}}\big)\mathds{1}_{(\tau_n,\tau_{n+1}]}(u) \diff u\Big|^p\Big] 
    \leq \E\Big[\sup_{t\in[0,s]} T^{p-1} \int_0^t  \sum_{n=0}^{\infty} \big|\widetilde \mu\big(Z^{(\delta)}_{\tau_{n}}\big)\big|^p  \mathds{1}_{(\tau_n,\tau_{n+1}]}(u) \diff u\Big] \\
    &\leq 2^{p-1}\, T^{p-1}\, c_{\widetilde\mu}^p\, \E\Big[\int_0^s  1 + |Z^{(\delta)}_{\underline u}|^p \diff u\Big] \leq c_\delta.
\end{aligned}
\end{equation}
For the third summand we apply Lemma \ref{BDGaKunita} to show that there exists $\hat c\in(0,\infty)$ such that for all $s\in[0,T]$ and all $\delta\in(0,\delta_0]$,
\begin{equation}
    \begin{aligned}\label{ME13}
    &\E\Big[\sup_{t\in[0,s]}\Big|\int_0^t \sum_{n=0}^{\infty} \Big(\widetilde \sigma\big(Z^{(\delta)}_{\tau_{n}}\big) + \widetilde\sigma \big(Z^{(\delta)}_{\tau_n}\big) d_{\widetilde\sigma} \big(Z^{(\delta)}_{\tau_n}\big) (W_{u}-W_{\tau_n}) \Big)\mathds{1}_{(\tau_n,\tau_{n+1}]}(u) \diff W_u\Big|^p\Big]\\
    &\leq \hat c\, \E\Big[\int_0^s \sum_{n=0}^{\infty} \big|\widetilde \sigma\big(Z^{(\delta)}_{\tau_{n}}\big) + \widetilde\sigma \big(Z^{(\delta)}_{\tau_n}\big) d_{\widetilde\sigma} \big(Z^{(\delta)}_{\tau_n}\big) (W_{u}-W_{\tau_n}) \big|^p\mathds{1}_{(\tau_n,\tau_{n+1}]}(u) \diff u \Big]\\
    &\leq 2^{p-1} \hat c\, \Big( \E\Big[\int_0^s \sum_{n=0}^{\infty} \big|\widetilde \sigma\big(Z^{(\delta)}_{\tau_{n}}\big) \Big|^p \mathds{1}_{(\tau_n,\tau_{n+1}]}(u) \diff u \big] \\
    &\quad\quad\quad\quad + \E\Big[\int_0^s \sum_{n=0}^{\infty} \big|\widetilde\sigma \big(Z^{(\delta)}_{\tau_n}\big) d_{\widetilde\sigma} \big(Z^{(\delta)}_{\tau_n}\big) (W_{u}-W_{\tau_n}) \big|^p \mathds{1}_{(\tau_n,\tau_{n+1}]}(u) \diff u \Big]\Big).
\end{aligned}
\end{equation}
Using Lemma \ref{MW} and Lemma \ref{FiniteMom} we obtain that there exist $c_\delta, c_{W_p}\in(0,\infty)$ such that for all $s\in[0,T]$ and all $\delta\in(0,\delta_0]$,
\begin{equation}
    \begin{aligned}\label{ME16}
    &\E\Big[\sup_{t\in[0,s]}\Big|\int_0^t \sum_{n=0}^{\infty} \Big(\widetilde \sigma\big(Z^{(\delta)}_{\tau_{n}}\big) + \widetilde\sigma \big(Z^{(\delta)}_{\tau_n}\big) d_{\widetilde\sigma} \big(Z^{(\delta)}_{\tau_n}\big) (W_{u}-W_{\tau_n}) \Big)\mathds{1}_{(\tau_n,\tau_{n+1}]}(u) \diff W_u\Big|^p\Big]\\
    &\leq 2^{p-1} \hat c\, \Big(  2^{p-1} c_{\widetilde \sigma}^p\, \int_0^s \E\Big[ 1+ \big|Z^{(\delta)}_{\underline u}\big|^p\Big] \diff u  
    + 2^{p-1} c_{\widetilde\sigma}^p L_{\widetilde\sigma}^p \, \int_0^s \E\Big[ \big(1+ \big|Z^{(\delta)}_{\underline u}\big|^p\big)  |W_{u}-W_{\underline u}|^p \Big]  \diff u \Big)\\
    &\leq  4^{p-1} \hat c\, c_{\widetilde \sigma}^p \big( 1
    + L_{\widetilde\sigma}^p c_{W_p} \delta^{\frac{p}{2}}\big) \, \int_0^s \E\big[ 1+ \big|Z^{(\delta)}_{\underline u}\big|^p\big]  \diff u
    \leq  4^{p-1} \hat c\, c_{\widetilde \sigma}^p \big( 1
    + L_{\widetilde\sigma}^p c_{W_p} \delta^{\frac{p}{2}}\big) c_\delta.
\end{aligned}
\end{equation}
For the last summand, we recall that $\nu_i$ denotes the $i$-th jump time of $N$ and use Lemma \ref{FiniteMom} to obtain that for all $\delta\in(0,\delta_0]$ there exists $c_\delta\in(0,\infty)$ such that for all $s\in[0,T]$,
\begin{equation}
    \begin{aligned}\label{ME17}
    &\E\Big[\sup_{t\in[0,s]}\Big|\int_0^t \widetilde \rho\big(Z^{(\delta)}_{u-}\big)\diff N_u\Big|^p\Big]
    = \E\Big[\sup_{t\in[0,s]}\Big| \sum_{i = 1}^{N_t} \widetilde \rho\big(Z^{(\delta)}_{\nu_i-}\big) \Big|^p\Big]
    \leq \E\Big[\Big( \sum_{i = 1}^{N_s} \big|\widetilde \rho\big(Z^{(\delta)}_{\nu_i-}\big)\big| \Big)^p\Big]\\
    &\leq 2^{p-1} c_{\widetilde\rho}^p\, \E\Big[N_s^{p-1} \sum_{i = 1}^{N_s} \big( 1+ \big|Z^{(\delta)}_{\nu_i-}\big|^p\big)\Big]
    \leq 2^{p-1} c_{\widetilde\rho}^p\, \E\Big[N_T^{p-1} \sum_{n=0}^{N_T +M} \big( 1+ \big|Z^{(\delta)}_{\tau_n-}\big|^p\big)\Big]
    \leq 2^{p-1} c_{\widetilde\rho}^p c_\delta,
\end{aligned}
\end{equation}
where $M = \Big\lceil \frac{T}{\delta^2 \log^4(\delta^{-1})} \Big\rceil$.
Plugging \eqref{ME2}, \eqref{ME16}, and \eqref{ME17} in \eqref{ME1} leads that for all $s\in[0,T]$, all $p\in\N$, and all $\delta\in(0,\delta_0]$,
\begin{equation}
    \begin{aligned}\label{ME18}
    &\E\Big[\sup_{t\in[0,s]} |Z^{(\delta)}_{t}|^p\Big]  <\infty.
\end{aligned}
\end{equation}
Next we change our previous estimates with the goal to apply Gronwall's inequality. We again start with \eqref{ME1}. We change the estimate \eqref{ME2} for all $s\in[0,T]$ and all $\delta\in(0,\delta_0]$ to 
\begin{equation}
    \begin{aligned}\label{ME19}
    &\E\Big[\sup_{t\in[0,s]}\Big|\int_0^t \sum_{n=0}^{\infty} \widetilde \mu\big(Z^{(\delta)}_{\tau_{n}}\big)\mathds{1}_{(\tau_n,\tau_{n+1}]}(u) \diff u\Big|^p\Big] 
   \leq 2^{p-1}\, T^{p-1}\, c_{\widetilde\mu}^p\, \E\Big[\int_0^s  1 + |Z^{(\delta)}_{\underline u}|^p \diff u\Big] \\
   &\leq  2^{p-1}\, T^{p}\, c_{\widetilde\mu}^p +  2^{p-1}\,T^{p-1}\, c_{\widetilde\mu}^p\int_0^s\E\Big[ \sup_{v\in[0,u]}  \big|Z^{(\delta)}_{v}\big|^p\Big]\diff u.
\end{aligned}
\end{equation}
We change the estimate in \eqref{ME16} for all $s\in[0,T]$ and all $\delta\in(0,\delta_0]$ to 
\begin{equation}
    \begin{aligned}\label{ME24}
    &\E\Big[\sup_{t\in[0,s]}\Big|\int_0^t \sum_{n=0}^{\infty} \Big(\widetilde \sigma\big(Z^{(\delta)}_{\tau_{n}}\big) + \widetilde\sigma \big(Z^{(\delta)}_{\tau_n}\big) d_{\widetilde\sigma} \big(Z^{(\delta)}_{\tau_n}\big) (W_{u}-W_{\tau_n}) \Big)\mathds{1}_{(\tau_n,\tau_{n+1}]}(u) \diff W_u\Big|^p\Big]\\
    &\leq 4^{p-1} \hat c\, c_{\widetilde \sigma}^p \big( 1
    + L_{\widetilde\sigma}^p c_{W_p} \delta^{\frac{p}{2}}\big)  \, \int_0^s \E\big[ 1+ \big|Z^{(\delta)}_{\underline u}\big|^p\big]  \diff u\\
    &\leq  4^{p-1} \hat c\, c_{\widetilde \sigma}^p \big( 1
    + L_{\widetilde\sigma}^p c_{W_p} \delta^{\frac{p}{2}}\big)  T +  4^{p-1} \hat c\, c_{\widetilde \sigma}^p \big( 1
    + L_{\widetilde\sigma}^p c_{W_p} \delta^{\frac{p}{2}}\big) \int_0^s\E\Big[ \sup_{v\in[0,u]}  \big|Z^{(\delta)}_{v}\big|^p\Big]\diff u.
\end{aligned}
\end{equation}
For the last summand of \eqref{ME1} by Lemma \ref{BDGaKunita} there exists $\hat c \in(0,\infty)$ such that for all $s\in[0,T]$ and all $\delta\in(0,\delta_0]$,
\begin{equation}
    \begin{aligned}\label{ME25}
    &\E\Big[\sup_{t\in[0,s]}\Big|\int_0^t \widetilde \rho\big(Z^{(\delta)}_{u-}\big)\diff N_u\Big|^p\Big] 
    \leq \hat c \int_0^s \E\Big[\big| \widetilde \rho\big(Z^{(\delta)}_{u-}\big)\big|^p\Big]\diff u\\
    &\leq 2^{p-1} \,\hat c\, c_{\widetilde\rho}^p\, T + 2^{p-1} \,\hat c\, c_{\widetilde\rho}^p \int_0^s \E\Big[ \sup_{v\in[0,u]}\big|Z^{(\delta)}_{v}\big|^p\Big]\diff u.
\end{aligned}
\end{equation}
Combining \eqref{ME1}, \eqref{ME19}, \eqref{ME24}, and \eqref{ME25} we get that there exists a constants $\widetilde c_1 \in(0,\infty)$ such that for all $s\in[0,T]$ and all $\delta\in(0,\delta_0]$
\begin{equation}
    \begin{aligned}\label{ME26}
    \E\Big[\sup_{t\in[0,s]} |Z^{(\delta)}_{t}|^p\Big] 
    \leq \widetilde c_1 + \widetilde c_1 \int_0^s \E\Big[ \sup_{v\in[0,u]}\big|Z^{(\delta)}_{v}\big|^p\Big]\diff u. 
    \end{aligned}
\end{equation}
As $[0,T] \ni s \mapsto \E\big[ \sup_{v\in[0,s]}\big|Z^{(\delta)}_{v}\big|^p\big]$ is a finite valued, Borel measurable function we obtain using Gronwall's inequality that there exists a constant $c_1\in(0,\infty)$ such that for all $t\in[0,T]$ and all $\delta\in(0,\delta_0]$,
\begin{equation}
    \begin{aligned}\label{ME27}
    &\E\Big[\sup_{t\in[0,T]} |Z^{(\delta)}_{t}|^p\Big] \leq c_1.
    \end{aligned}
\end{equation}

Next we observe that there exists a constant $\widetilde c_2\in(0,\infty)$ such that for all $t\in[0,T]$ and all $\delta\in(0,\delta_0]$,
\begin{equation}\label{ME28}
    \begin{aligned}
    &\E\big[  \big|Z^{(\delta)}_t- Z^{(\delta)}_{\underline t}\big|^p\big] \\
    &\leq \widetilde c_2 \E\Big[ \big| \widetilde \mu ( Z^{(\delta)}_{\underline t})(t- \underline t)\big|^p + \big|\widetilde\sigma(Z^{(\delta)}_{\underline t}) (W_t-W_{\underline t})\big|^p
    + \Big|\frac{1}{2}\widetilde\sigma(Z^{(\delta)}_{\underline t})d_{\widetilde\sigma}(Z^{(\delta)}_{\underline t}) \big((W_t-W_{\underline t})^2 - (t-\underline t)\big)\Big|^p\Big]\\
    &\leq \widetilde c_2 \Big( 2^{p-1} c_{\widetilde\mu}^p \delta^p\,  \E\big[  1+| Z^{(\delta)}_{\underline t}|^p \big]
    +2^{p-1} c_{\widetilde\sigma}^p\, \E\big[  \big( 1+| Z^{(\delta)}_{\underline t}|^p\big) |W_t-W_{\underline t}|^p\big]\\
    &\quad\quad\quad+\frac{1}{2} c_{\widetilde\sigma}^p L_{\widetilde\sigma}^p\, \E\big[  \big( 1+| Z^{(\delta)}_{\underline t}|^p\big) |W_t-W_{\underline t}|^{2p}\big] 
    +\frac{1}{2} c_{\widetilde\sigma}^p L_{\widetilde\sigma}^p\delta^p\,  \E\big[  1+| Z^{(\delta)}_{\underline t}|^p \big] \Big).
    \end{aligned}
\end{equation}
Applying Lemma \ref{MEW} to the second and third summand and using \eqref{ME27} we get that there exists a constant $\widetilde c_3\in(0,\infty)$ such that for all $t\in[0,T]$ and all $\delta\in(0,\delta_0]$,
\begin{equation}\label{ME29}
    \begin{aligned}
    &\E\big[  \big|Z^{(\delta)}_t- Z^{(\delta)}_{\underline t}\big|^p\big]
    \leq \widetilde c_3\big(\delta^p+ \delta^\frac{p}{2} + \delta^p +\delta^p\big)  \E\big[  1+| Z^{(\delta)}_{\underline t}|^p \big]\\
    &\leq 4 \widetilde c_3 \delta^\frac{p}{2} \Big(  1+ \E\Big[ \sup_{v\in[0,T]} | Z^{(\delta)}_{v}|^p\big) \Big] \Big)
    \leq 4 \widetilde c_3 \delta^\frac{p}{2} (  1+ c_1).
    \end{aligned}
\end{equation}
This proves the second statement. 

For the last statement by Lemma \ref{BDGaKunita} there exists a constant $\hat c\in(0,\infty)$ such that for all $t\in[0,T-\Delta]$ and all $\delta\in(0,\delta_0]$,
\begin{equation}\label{ME34}
    \begin{aligned}
    &\E\Big[ \sup_{s\in[t,t+\Delta]} \big|Z^{(\delta)}_s- Z^{(\delta)}_{t}\big|^p\Big]\\
    &\leq 3^{p-1}\hat c \Big(\int_{t}^{t+\Delta}\E\big[  \big|   \widetilde\mu (Z^{(\delta)}_{\underline u}) \big|^p\big] \diff u
    + \int_{t}^{t+\Delta}\E\big[\big|   \widetilde\sigma (Z^{(\delta)}_{\underline u}) + \widetilde\sigma (Z^{(\delta)}_{\underline u}) d_{\widetilde\sigma}(Z^{(\delta)}_{\underline u}) (W_u -W_{\underline u}) \big|^p\big] \diff u\\
    &\quad\quad\quad\quad\quad +\int_{t}^{t+\Delta} \E\Big[\big| \widetilde\rho (Z^{(\delta)}_{u-}) \big|^p \Big]\diff u\Big).
    \end{aligned}
\end{equation}
For the first summand we estimate for all $u\in[0,T]$ and all $\delta\in(0,\delta_0]$,
\begin{equation}\label{ME35}
    \begin{aligned}
    &\E\big[  \big|   \widetilde\mu (Z^{(\delta)}_{\underline u}) \big|^p\big] 
    \leq 2^{p-1} c_{\widetilde\mu}^p\, \E\big[  1+ \big|Z^{(\delta)}_{\underline u} \big|^p\big] 
    &\leq 2^{p-1} c_{\widetilde\mu}^p\, \Big( 1+  \E\big[ \sup_{t\in[0,T]}\big|Z^{(\delta)}_{t}\big|^p\big]\Big)
    \leq  2^{p-1} c_{\widetilde\mu}^p\, \big( 1+  c_1\big).
    \end{aligned}
\end{equation}
For the second summand we obtain by Lemma \ref{MW} that there exists $c_{W_p}\in (0,\infty)$ such that for all $u\in[0,T]$ and all $\delta\in(0,\delta_0]$,
\begin{equation}\label{ME36}
    \begin{aligned}
    &\E\big[\big|   \widetilde\sigma (Z^{(\delta)}_{\underline u}) + \widetilde\sigma (Z^{(\delta)}_{\underline u}) d_{\widetilde\sigma}(Z^{(\delta)}_{\underline u}) (W_u -W_{\underline u}) \big|^p\big]\\
    &\leq \E\Big[  2^{p-1} \big( \big|\widetilde\sigma (Z^{(\delta)}_{\underline u})\big|^p + \big|\widetilde\sigma (Z^{(\delta)}_{\underline u}) d_{\widetilde\sigma}(Z^{(\delta)}_{\underline u}) (W_u -W_{\underline u})\big|^p\big) \Big]\\
    &\leq 4^{p-1} c_{\widetilde\sigma}^p\, \E\big[   1+\big|Z^{(\delta)}_{\underline u}\big|^p\big]
    + 4^{p-1} c_{\widetilde\sigma}^p L_{\widetilde\sigma}^p\, \E\big[  \big(1+\big|Z^{(\delta)}_{\underline u}\big|^p\big) |W_u -W_{\underline u}|^p\big]\\
    &\leq 4^{p-1} c_{\widetilde\sigma}^p\, \big( 1+ \E\big[\sup_{t\in[0,T]}\big|Z^{(\delta)}_{t}\big|^p\big]\big)
    + 4^{p-1} c_{\widetilde\sigma}^p L_{\widetilde\sigma}^p c_{W_p} \delta^\frac{p}{2} \, \E\big[   1+\big|Z^{(\delta)}_{\underline u}\big|^p\big]\\
    &\leq4^{p-1} c_{\widetilde\sigma}^p ( 1 +  L_{\widetilde\sigma}^p c_{W_p} \delta^{\frac{p}{2}} )( 1+ c_1).
    \end{aligned}
\end{equation}
For the third summand we get for all $u\in[0,T]$ and all $\delta\in(0,\delta_0]$,
\begin{equation}\label{ME37}
    \begin{aligned}
    &\E\Big[\big| \widetilde\rho (Z^{(\delta)}_{u-}) \big|^p \Big] 
    \leq 2^{p-1} c_{\widetilde\rho}^p \, \big(1+\E\big[\sup_{t\in[0,T]}\big|Z^{(\delta)}_{t}\big|^p\big]\big)
    \leq 2^{p-1} c_{\widetilde\rho}^p \,(1+c_1).
    \end{aligned}
\end{equation}
Plugging \eqref{ME35}, \eqref{ME36}, and \eqref{ME37} into \eqref{ME34} yields that there exists a constant $\widetilde c_4\in(0,\infty)$ such that for all $t\in[0,T-\Delta]$ and all $\delta\in(0,\delta_0]$,
\begin{equation}\label{ME38}
    \begin{aligned}
    &\E\Big[\sup_{s\in[t,t+\delta]} \big|Z^{(\delta)}_s- Z^{(\delta)}_{t}\big|^p\Big]
    \leq \widetilde c_4\, \int_{t}^{t+\Delta}( 1+ c_1) \diff u \leq \widetilde c_4\,(1+  c_1 )\, \Delta,
    \end{aligned}
\end{equation}
which finishes the proof.
\end{proof}

\begin{proof}[Proof of Lemma \ref{MWwithInd}]
Recall that $\tau_{n+1}$ can be expressed as a measurable function $f$ of $\tau_n$, $Z^{(\delta)}_{\tau_n}$ and $(N_{\tau_n+s} -N_{\tau_n})_{s\geq 0}$. Hence for all $n\in\N$, all $k\in\N$, all $\delta\in(0,\delta_0]$ and all $t\in[0,T]$,
\begin{equation}
    \begin{aligned}
    &\E\big[\big(1+ \big|Z^{(\delta)}_{\tau_n}\big|^p\big)   |W_t-W_{\tau_n}|^q \mathds{1}_{\Theta^{\varepsilon_2^\delta}}(Z_{\tau_n}^{(\delta)}) \mathds{1}_{[\tau_n,\tau_{n+1})}(t)  \mathds{1}_{\{N_t = k\}} \big]\\
    &= \E\big[\big(1+ \big|Z^{(\delta)}_{\tau_n}\big|^p\big)   |W_t-W_{\tau_n}|^q  \mathds{1}_{\Theta^{\varepsilon_2^\delta}}(Z_{\tau_n}^{(\delta)}) \mathds{1}_{[\tau_n,\tau_{n+1})}(t) \mathds{1}_{\{N_{\tau_n} = k\}} \big]\\
    &\leq \E\Big[\big(1+ \big|Z^{(\delta)}_{\tau_n}\big|^p\big)  \sup_{s\in[0,\delta^2\log^4(\delta^{-1})]} |W_{s+\tau_n}-W_{\tau_n}|^q \mathds{1}_{\Theta^{\varepsilon_2^\delta}}(Z_{\tau_n}^{(\delta)}) \\
    &\quad\quad\quad\quad \quad\quad\quad \cdot \mathds{1}_{\{\tau_n \leq t\}} \mathds{1}_{\{f(\tau_n,Z^{(\delta)}_{\tau_n}, (N_{s+\tau_n}-N_{\tau_n})_{s\geq 0}) > t\}} \mathds{1}_{\{N_{\tau_n} = k\}} \Big].
    \end{aligned}
\end{equation}
By \cite[p.~33]{Grikhman2004}, for all $n\in\N$, all $k\in\N$, all $\delta\in(0,\delta_0]$ and all $t\in[0,T]$,
\begin{equation}
    \begin{aligned}
    &\E\big[\big(1+ \big|Z^{(\delta)}_{\tau_n}\big|^p\big)   |W_t-W_{\tau_n}|^q \mathds{1}_{\Theta^{\varepsilon_2^\delta}}(Z_{\tau_n}^{(\delta)})  \mathds{1}_{[\tau_n,\tau_{n+1})}(t)  \mathds{1}_{\{N_t = k\}} \big]\\
    &\leq \E\Big[\big(1+ \big|Z^{(\delta)}_{\tau_n}\big|^p\big) \mathds{1}_{\{\tau_n \leq t\}} \mathds{1}_{\{N_{\tau_n} = k\}} \\
    &\quad\quad\quad \E\Big[ \sup_{s\in[0,\delta^2\log^4(\delta^{-1})]} |W_{s+\tau_n}-W_{\tau_n}|^q  \mathds{1}_{\{f(\tau_n, Z^{(\delta)}_{\tau_n}, (N_{s+\tau_n}-N_{\tau_n})_{s\geq 0})> t\}}  \Big| \F_{\tau_n} \Big] \Big] \\
    &= \E\Big[\big(1+ \big|Z^{(\delta)}_{\tau_n}\big|^p\big) \mathds{1}_{\{\tau_n \leq t\}} \mathds{1}_{\{N_{\tau_n} = k\}} \\
    &\quad\quad\quad \cdot\E\Big[ \sup_{s\in[0,\delta^2\log^4(\delta^{-1})]} |W_{s+\tau_n}-W_{\tau_n}|^q  \mathds{1}_{\{f(z,\widetilde z, (N_{s+\tau_n}-N_{\tau_n})_{s\geq 0})> t\}}  \Big| \F_{\tau_n} \Big]\Big|_{z =\tau_n, \widetilde z =Z^{(\delta)}_{\tau_n} } \Big].
    \end{aligned}
\end{equation}
By Lemma \ref{PropDiscGrid}, Lemma \ref{MEW}, and \cite[p.~33]{Grikhman2004} we get for all $n\in\N$, all $k\in\N$, all $\delta\in(0,\delta_0]$ and all $t\in[0,T]$,
\begin{equation}
    \begin{aligned}
    &\E\big[\big(1+ \big|Z^{(\delta)}_{\tau_n}\big|^p\big)   |W_t-W_{\tau_n}|^q \mathds{1}_{\Theta^{\varepsilon_2^\delta}}(Z_{\tau_n}^{(\delta)})  \mathds{1}_{[\tau_n,\tau_{n+1})}(t)  \mathds{1}_{\{N_t = k\}} \big]\\
    &\leq c_{W_q} \delta^q \log^{2q}(\delta^{-1}) \E\Big[\big(1+ \big|Z^{(\delta)}_{\tau_n}\big|^p\big) \mathds{1}_{\{\tau_n \leq t\}} \mathds{1}_{\{N_{\tau_n} = k\}} \E\Big[ \mathds{1}_{\{f(z,\widetilde z, (N_{s+\tau_n}-N_{\tau_n})_{s\geq0})> t\}} \Big]\Big|_{z =\tau_n, \widetilde z =Z^{(\delta)}_{\tau_n}} \Big]\\
    &= c_{W_q} \delta^q \log^{2q}(\delta^{-1}) \E\Big[\big(1+ \big|Z^{(\delta)}_{\tau_n}\big|^p\big) \mathds{1}_{\{\tau_n \leq t\}}  \mathds{1}_{\{f(\tau_n,Z^{(\delta)}_{\tau_n}, (N_{s+\tau_n}-N_{\tau_n})_{s\geq0})> t\}} \mathds{1}_{\{N_{\tau_n} = k\}} \Big]\\
    &= c_{W_q} \delta^q \log^{2q}(\delta^{-1}) \E\Big[\big(1+ \big|Z^{(\delta)}_{\tau_n}\big|^p\big) \mathds{1}_{[\tau_n, \tau_{n+1})}(t)  \mathds{1}_{\{N_{t} = k\}} \Big].
    \end{aligned}
\end{equation}
Hence, for all $n\in\N$, all $k\in\N$, all $\delta\in(0,\delta_0]$ and all $t\in[0,T]$,
\begin{equation}
    \begin{aligned}
    &\E\Big[ \big(1+ \big|Z^{(\delta)}_{\underline{t}}\big|^p\big)   |W_t-W_{\underline{t}}|^q \mathds{1}_{\Theta^{\varepsilon_2^\delta}}(Z_{\underline t}^{(\delta)})\Big]\\
    &= \sum_{k=0}^\infty \sum_{n=0}^{\infty} \E\Big[ \big(1+ \big|Z^{(\delta)}_{\tau_n}\big|^p\big)   |W_t-W_{\tau_n}|^q \mathds{1}_{\Theta^{\varepsilon_2^\delta}}(Z_{\tau_n}^{(\delta)}) \mathds{1}_{[\tau_n,\tau_{n+1})}(t) \mathds{1}_{\{N_t=k\}} \Big]\\
    &\leq \sum_{k=0}^\infty \sum_{n=0}^{\infty} c_{W_q} \delta^q \log^{2q}(\delta^{-1}) \E\Big[\big(1+ \big|Z^{(\delta)}_{\tau_n}\big|^p\big) \mathds{1}_{[\tau_n, \tau_{n+1}) }(t)  \mathds{1}_{\{N_{t} = k\}} \Big]\\
    &\leq c_{W_q} \delta^q \log^{2q}(\delta^{-1}) \E\Big[\sum_{n=0}^{\infty} \big(1+ \big|Z^{(\delta)}_{\tau_n}\big|^p\big) \mathds{1}_{[\tau_n, \tau_{n+1}) }(t) \Big]
    \leq c_{W_q}  \delta^q \log^{2q}(\delta^{-1}) \Big(1+ \E\big[\sup_{t\in[0,T]} \big|Z^{(\delta)}_{t}\big|^p\big] \Big)\\
    &\leq (1+ c_1^p) c_{W_q}  \delta^q \log^{2q}(\delta^{-1}).
    \end{aligned}
\end{equation}
\end{proof}

\setlength{\bibsep}{0pt plus 0.3ex}

\vspace{2em}
\centerline{\underline{\hspace*{16cm}}}

\noindent Verena Schwarz \Letter \\
Department of Statistics, University of Klagenfurt, Universit\"atsstra\ss{}e 65-67, 9020 Klagenfurt, Austria\\
verena.schwarz@aau.at\\

\end{document}